\documentclass[11pt, a4paper, twoside, reqno]{amsbook}

\usepackage[utf8]{inputenc}
\usepackage[T1]{fontenc}
\usepackage{amsmath, mathtools, amsthm}
\usepackage{amssymb,url,xspace}
\usepackage{cases}

\usepackage[dvipsnames]{xcolor}
\usepackage[hidelinks]{hyperref}	
\hypersetup{
	colorlinks,
	citecolor=BrickRed,
	filecolor=black,
	linkcolor=blue,
	urlcolor=blue
}

\usepackage{anysize}
\marginsize{4cm}{4cm}{3cm}{3cm}

\newcommand{\R}{\mathbb{R}}
\newcommand{\N}{\mathbb{N}}
\newcommand{\eqn}[1]{\(#1\)}
\newcommand{\BC}{{\rm BC}}
\newcommand{\cal}{\mathcal}
\newcommand{\sgn}{\rm sgn}
\newcommand{\Span}{\rm Span}
\renewcommand{\div}{\rm div\,}

\newcommand{\delfrac}[4]{\biggl#1{#2\over#3}\biggr#4}
\newcommand{\np}{{{N\over2}\left({\scriptscriptstyle 1}
		-{1\over p}\right)}}
\newcommand{\op}{{\frac{1}{2}\left({\scriptscriptstyle 1}
		-\frac{1}{p}\right)}}

\newtheorem{thm}{Theorem}[chapter]
\newtheorem{lem}[thm]{Lemma}
\newtheorem{prop}[thm]{Proposition}

\theoremstyle{definition}

\theoremstyle{remark}

\newtheorem{obs}[thm]{Observation}

\numberwithin{section}{chapter}
\numberwithin{equation}{chapter}

\makeindex

\begin{document}
	
	\frontmatter
	
	\title{Asymptotic behavior of scalar convection-diffusion equations\footnote{Translated from the original Spanish version \cite{zuazua.notes} and bibliographically updated by Borjan Geshkovski.}}
	
	\author{Enrique Zuazua \\ \vspace{2cm} \small March 19, 2020}
	
	\newcommand{\Adresses}{
		{
			\bigskip
			\footnotesize
			\textbf{Enrique Zuazua} \newline \indent
			\textsc{Chair in Applied Analysis}, \newline \indent 
				\textsc{Alexander von Humboldt-Professorship}
				\newline \indent
				\textsc{Department of Mathematics}, \newline \indent
				\textsc{Friedrich-Alexander-Universit\"at Erlangen-N\"urnberg}
			\newline \indent
			\textsc{91058 Erlangen, Germany}
			\newline \indent \hspace{2.5cm} \textit{and} \newline \indent
			\textsc{Chair of Computational Mathematics}
			\newline \indent 
			\textsc{Fundación Deusto}
			\newline \indent
			\textsc{Av. de the Universidades, 24}
			\newline \indent
			\textsc{48007 Bilbao, Basque Country, Spain}
			\newline \indent \hspace{2.5cm} \textit{and} \newline \indent
			\textsc{Departamento de Matemáticas}, \newline \indent
				\textsc{Universidad Autónoma de Madrid,}
			\newline \indent
			\textsc{28049 Madrid, Spain}
			}
			\newline \indent 
			{\footnotesize \textit{email}: \texttt{enrique.zuazua@fau.de}}
		}

	\thanks{}
	\date{19 March, 2020}
	
	\maketitle
	
	\tableofcontents

		\textbf{Acknowledgments:}  These Notes originate in a course that I delivered at the Institute of Mathematics  of the Universidad Federal de Rio de Janeiro (UFRJ) in March 1991. I had the honor of visiting that institution under the invitation and hospitality of Prof. Carlos Frederico Vasconcellos.
		At that time the Notes were published in Spanish with the title "Comportamiento   asint\'otico   de   ecuaciones   escalares   de
convecci\'on-difusi\'on" in 1993 in the Lecture Notes Series of the Institute. 

Recently, our PhD student Borjan Geshkovski, who has worked in this area, among others, was kind enough to carefully translate the Notes to English, updating the bibliography and catching up some typos.

During the preparation of this last version we received funding from the Alexander von Humboldt-Professorship program, the European Research Council (ERC) Advanced Grant DyCon, the European Unions Horizon 2020 research and innovation programme under the Marie Sklodowska-Curie grant agreement No.765579-ConFlex, the Air Force Office of Scientific Research under Award NO: FA9550-18-1-0242, the Grant MTM2017-92996-C2-1-R COSNET of MINECO (Spain), the Grant ICON-ANR-16-ACHN-0014 of the French ANR and the  Transregio 154 Project "Mathematical Modelling, Simulation and Optimization Using the Example of Gas Networks" of the German DFG.
Mathematical Modelling, Simulation and Optimization Using the Example of Gas Networks.

	
	\mainmatter

		\chapter*{Introduction}
		
	In these notes, we address the problem of asymptotic behaviour when time goes to infinity of the solutions to scalar convection-diffusion equations of the form
		\begin{align} 
			u_t-\Delta u &= a\cdot\nabla\bigl(F(u)\bigr) \hspace{1cm} \text{ in } \R^N \times(0,\infty) \label{eq: 1.1}\\ 
			u(x,0)&=u_0(x) \label{eq: 1.2}
		\end{align}
		where $a$ is a constant vector in $\R^N$, $F\in C^1(\R)$ is such that
		$F(0)=0$ and $u_0\in L^1(\R^N)$. Henceforth $\cdot$ denotes the scalar product
		in $\R^N$.
		The case $a=0$ corresponds to the linear heat equation.
		
	In spatial dimension $N=1$ and when
		\begin{equation*} 
			F(u)=\frac12 u^2
		\end{equation*}
		and $a=1$, equation \eqref{eq: 1.1} was first introduced by H.\ Bateman \cite{B} and J.~M.~Burgers \cite{Bu1, Bu2} as an approximation for equations of fluid-flow, and has since been known as {\em the Burgers equation}. 
It is a very simple model which combines nonlinear wave propagation with the effect of heat conduction.
		
These equations are also a model for the displacement of a fluid in a porous medium, when the effect of capilarity is taken into account (cf.\ S. E. Buckley and M. C. Leverett \cite{BuLe}; see also D. W. Peaceman \cite{P} for a numerical treatment).

More generally, \eqref{eq: 1.1} is a models the regularizing effect of a (small) viscosity in a hyperbolic conservation law.
		
The asymptotic behaviour of the solutions to \eqref{eq: 1.1}--\eqref{eq: 1.2} is strongly linked with the behavior of the initial data $u_0$  when  $|x|\to\infty$. 
In order to show this fact, we consider the case when the spatial dimension $N$ equals one: $N=1$.
Let $a=-1$, let $F$ be a convex function and $$u_\pm =\lim_{|x|\to\pm\infty} u_0(x).$$

Then:

		\begin{enumerate}
			\item[(a)]
			If $u_-<u_+$, the solutions to \eqref{eq: 1.1}--\eqref{eq: 1.2},  as $t\to\infty$,
			converge in $L_{\rm loc}^p(\R)$, $1\le p<\infty$, to a rarefaction wave $p(x/t)$ which is an entropy solution of the Riemann problem with data $u_-$ and $u_+$ for the hyperbolic conservation law: 
	$$u_t+\bigl(F(u)\bigr)_x=0,$$
			(cf.\ A.~M.~Il'in and O.~A.~Oleinik \cite{IO1}, \cite{IO2} and E.~Harabetian
			\cite{H1}). This type of results have been further extended by Z.~P.~Xin \cite{X1} for
			 $2\times 2$ systems in dimension $N=1$, and by Z.~P.~Xin \cite{X2} y
			E.~Harabetian \cite{H2} for scalar equations in multiple spatial dimensions dimensions (for the case where when $F$ is not convex, see H.~F.~Weinberger
			\cite{W}).
			
			\item[(b)]
			If $u_+>u_-$, A.~M.~Il'in and O.~A.~Oleinik \cite{IO1, IO2} proved the existence and stability of solutions of  ``travelling waves'' type: $u=u(x-ct)$. This type of results has been further extended to systems in multiple spatial dimensions by J.~Goodman \cite{G1} and T.~P.~Liu \cite{L} and scalar equations in multiple spatial dimensions by J.~Goodman
			\cite{G2}.
			
			\item[(c)]
			if $u_0\in L^1(\R)$ (which corresponds to the case $u_-=u_+=0$) and if
			\begin{equation*}
				F(u)=\frac12 u^2
			\end{equation*}
			(Burgers equation), thanks to the Hopf-Cole transformation, it can be shown that \eqref{eq: 1.1} admits a one-parameter family of self-similar solutions
			$$u_M(x,t)=t^{-1/2}f_M\biggl(\frac{x}{\sqrt{t}}\biggr)$$
			where the profiles $f_M$ satisfy
			$$\int^\infty_{-\infty}f_M(x)\,dx=M.$$
		\end{enumerate}
		
		On another hand,  it can be shown that if $u_0\in L^1(\R)$ and
		$$\int^\infty_{-\infty}u_0(x)\,dx=M,$$ the solution to \eqref{eq: 1.1}--\eqref{eq: 1.2}
		behaves like $u_M$  when  $t\to\infty$. therefore, in this case, the asymptotic behaviour of the system is described by a one-parameter family of self-similar solutions $\{u_M\}_{M\in\R}$. The parameter $M$ corresponds to the mass of the solution, a quantity which is conserved along the entire trajectory:
		$$\int^\infty_{-\infty}u(x,t)\,dx=\int^\infty_{-\infty }u_0(x)\,dx,
		\qquad\forall t>0.$$
		This type of results have been extended by S.~Kawashine \cite{Kw1},
		T.~P.~Liu \cite{L} and I.~L.~Chern and T.~P.~Liu \cite{CL} for systems of viscous conservation laws in dimension $N=1$.
		The extension to scalar equations in dimensions $N>1$ was done by M.~Escobedo and E.~Zuazua \cite{EZ1}, \cite{EZ2}.
		
		In these notes we address exclusively the case of initial data
		$u_0\in L^1(\R^N)$.
		To fix ideas, we first consider the case of a particular homogeneous nonlinearity such as
		\begin{equation} \label{eq: 1.3}
		\begin{cases}
		u_t-\Delta u=a\cdot\nabla(|u|^{q-1}u) &\hbox{en $\R^N\times(0,\infty)$}\\
		 u(x,0)=u_0(x) &\hbox{en $\R^N$.}
		\end{cases}
		\end{equation}
		Integrating the equation \eqref{eq: 1.3} over $\R^N$ with respect to the spatial variable
		$x$ it can be observed that
		$$\frac{d}{dt}\int_{\R^N} u(x,t)\,dx=0.$$
		Thus, the mass of the solution is conserved over time.
		
		The question at hand is the study of how the mass is
distributed in space when  $t\to\infty$.
		
		As we will show later, we have the following estimate
of the decay of solutions to \eqref{eq: 1.3} in $L^p(\R^N)$:
		\begin{equation} \label{eq: 1.4}
		\|u(t)\|_p\leq C_pt^{-\np},\qquad \forall t>0  
		\end{equation}
		for all $p\in [1,\infty]$ with $C_p>0$ depending on $\|u_0\|_1$.
		Henceforth $\|\cdot\|_p$ denotes the usual norm in $L^p(\R^N)$.
		
		The problem at hand, formulated more precisely is the
behavior of $t^\np u(t)$ in $L^p(\R^N)$  when  $t\to\infty$.
		
		When $a=0$, \eqref{eq: 1.3} is the linear heat equation and its solutions behave like $u_M=MG$, where $M$ is the mass of the solution
		$$M=\int_{\R^N}u_0(x)\,dx$$ and $G$ is the heat kernel, i.e., the fundamental solution of the linear heat equation:
		$$G(x,t)=(4\pi t)^{-N/2} \exp\biggl(-\frac{|x|^2}{4t}\biggr).$$
		That is, for all $p\in[1,\infty]$ it holds
		\begin{equation} \label{eq: 1.5}
			t^\np\|u(t)-u_M(t)\|_p\to 0 \quad \text{  when  } \quad t\to\infty
		\end{equation}
		It is again observed that asymptotic behavior is governed by a one-parameter family of self-similar solutions: $u_M=MG$.
		
		The first question that arises in the general case \eqref{eq: 1.3} is the existence of self-similar solutions.
		
		The equation \eqref{eq: 1.3} is invariant with respect to the scaling
		\begin{equation} \label{eq: 1.6}
		u_\lambda (x,t)=\lambda^{1/(q-1)} u(\lambda x,\lambda^2t),
		\qquad\lambda>0.
		\end{equation}
		The self-similar solutions of \eqref{eq: 1.3}  are the ones that remain invariant with respect to this scaling, i.e.,
		\begin{equation} \label{eq: 1.7}
			u_\lambda (x,t)=u(x,t),\qquad\forall\lambda>0, \forall(x,t)\in
			\R^N \times(0,\infty).
		\end{equation}
		It is easy to check that $u$ satisfies \eqref{eq: 1.7} if and only if it is of the form
		\begin{equation} \label{eq: 1.8}
			u(x,t)=t^{-1/[2(q-1)]} f\biggl(\frac{x}{\sqrt{t}}\biggr).
		\end{equation}
In view of the conservation of mass \eqref{eq: 1.4}, a function $u$ that admits the expression \eqref{eq: 1.8} with profile $f\in L^1(\R^N)$ can only be a solution of \eqref{eq: 1.3} if \eqn{q=\frac{1}{N}+1}.
		
		We therefore see that the system \eqref{eq: 1.3} can only admit self-similar solutions in $L^1(\R^N)$ for the exponent $q=1+\frac{1}{N}.$ In dimension $N=1$ it holds $q=2$ which corresponds to the Burgers equation and which, as we saw earlier, admits self-similar solutions.
		
		When \eqn{q=1+\frac{1}{N}}, the self-similar solutions are 
		of the form
		\begin{equation} \label{eq: 1.9}
			u(x,t)= t^{-N/2} f\biggl(\frac{x}{\sqrt{t}}\biggr).
		\end{equation}
		The heat kernel admits this expression for the profile
		$$f(x)= (4\pi)^{-N/2} \exp\biggl(-\frac{|x|^2}{4}\biggr).$$
		Finding a solution of the form \eqref{eq: 1.9} to the equation \eqref{eq: 1.3} with
		$q=1+\frac{1}{N}$ is equivalent to solving the elliptic equation:
		\begin{equation} \label{eq: 1.10}
			-\Delta f-\frac{x\cdot\nabla f}{2}-\frac{N}{2} f =a\cdot\nabla
			(|f|^{1/N}f) \quad \text{ in } \R^N.
		\end{equation}
		In \cite{AEZ}, using a fixed point method we prove that for any $M\in\R$ the equation \eqref{eq: 1.10} admits a unique solution $f_M$ such that
		$$\int_{\R^N}f_M(x)\,dx=M.$$
		This solution decays exponentially to zero  when  $|x|\to\infty$.
		
		Later, in \cite{EZ1, EZ2} we prove that the general solution to \eqref{eq: 1.3} with
		$q=1+\frac{1}{N}$ behaves like the self-similar solution 
		$$u_M(x,t)=t^{-N/2}f_M\delfrac({x}{\sqrt{t}})$$ if		
	$$\int_{\R^N}u_0(x)\,dx=M.$$ 
	That is, we prove that  \eqref{eq: 1.5} is verified. This is an extension of the results mentioned above for the Burgers equation in dimension $N=1$ and for the linear heat equation.
		
		When $q>1+\frac{1}{N}$, as we proved in \cite{EZ1, EZ2}, the
		solutions to \eqref{eq: 1.3} behave like the heat kernel, that is, \eqref{eq: 1.5} holds with $u_M=MG$. In this case, we say that the system has
		{\em weakly non-linear behavior}.
		
		Let us consider next the equation with linear convection
		\begin{equation} \label{eq: 1.11}
		\begin{cases}
		u_t-\Delta u=a\cdot\nabla u &\hbox{en $\R^N\times(0,\infty)$,}\\
		 u(0)=u_0
		\end{cases}
		\end{equation}
		which corresponds to $q=1$.
		
		The function
		$$v(x,t)= u(x-at,t)$$
		satisfies
		\begin{equation*}
		\begin{cases}
		v_t-\Delta v=0, \\
		v(0)=u_0.
		\end{cases}
		\end{equation*}
		
		In the case of the linear heat equation, $v$ behaves like $u_M=MG$, where $$M=\int_{\R^N}u_0(x)\,dx.$$ Thus,  when 
		$t\to\infty$, $u$ behaves like $MG(x+at,t)$, that is,
		\eqref{eq: 1.5} holds for $u_M(x,t)=MG(x+at,t)$.
		
		The asymptotic behavior of the system  when  $1<q<1+\frac{1}{N}$ is of different nature. To describe it, it is appropriate to distinguish the case $N=1$
		from the case $N\ge 2$. We begin with the one-dimensional case $N=1$ which was solved by M.~Escobedo, J.~L.~Vázquez and E.~Zuazua in \cite{EVZ1}.
		
		In \cite{EVZ1} it is proven that if $N=1$ and $1<q<2$ the solution to \eqref{eq: 1.1} satisfies 
		\begin{equation} \label{eq: 1.12}
			t^{\frac{1}{q}
				\left({\scriptscriptstyle
					1}-\frac{1}{p}\right)}
			\|u(t)-u_M(t)\|_p\to 0 \quad \text{  when  } \quad t\to\infty
		\end{equation}
		
		for all $1\le p<\infty$ where $u_M=u_M(x,t)$ is the unique entropic solution of the hyperbolic system
		\begin{align} \label{eq: 1.13}
		\begin{cases}
		u_t=a(|u|^{q-1}u)_x&\hbox{en $\R \times(0,\infty)$}\\
		u(0)=M\delta.
		\end{cases}
		\end{align}
		The uniqueness of the entropic solution to \eqref{eq: 1.13} was shown T.~P.~Liu and
		M.~Pierre \cite{LP1}. The entropic solution of \eqref{eq: 1.13} is self-similar and for \eqn{a=-\frac{1}{q}} is of the form
		\begin{equation} \label{eq: 1.14}
			u_M(x,t)=\big(\frac{x}{t} \big)^{1/(q-1)}{\cal
				X}_{\big(0,r(t)\big)} 
		\end{equation}
		${\cal X}_{\left(0,r(t)\right)}$ being the characteristic function of the interval $\bigl(0,r(t)\bigr)$ and
		\begin{equation} \label{eq: 1.15}
			r(t)=cM^{(q-1)/q} t^{1/q},\qquad
			c=\biggl(\frac{q}{q-1}\biggr)^{(q-1)/q}. 
		\end{equation}
		In this case we see therefore that the effect of the term of diffusion of the equation \eqref{eq: 1.1} disappears when  $t\to\infty$. We are in the presence of a		{\em strongly non-linear behavior}.
		
		When $N\ge 2$ and $1<q<1+\frac{1}{N}$ there is an analogue result but this time the diffusion disappears only in the direction of the convection $a$. For example, if $a=(1,0,\ldots,0)$ the asymptotic behavior it is given by the self-similar entropy solutions of the hyperbolic-parabolic equation
		\begin{align*}
		\begin{cases}
			u_t-\Delta'u=\partial_{x_1}(|u|^{q-1}u|)&\hbox{en $\R^N
				\times(0,\infty)$,}\\
			u(0)=M\delta
		\end{cases}
		\end{align*}
		$\Delta'$ being the laplacian in the variables
		$x_{2},\ldots,x_N$ (see \cite{EVZ2}).
		
		Let's mention that, in this case $1<q<1+\frac{1}{N}$, One of the main difficulties is to test the following estimate on the rate of decay
in $L^\infty(\R^N)$ of the solutions to \eqref{eq: 1.1}:
		$$\|u(t)\|_\infty\le Ct^{-(N+1)/2q}$$
		which, since $1<q<1+\frac{1}{N}$, is faster than in other cases, both in the linear and in the self-similar or in the weakly non-linear.	
		In the general case, the system \eqref{eq: 1.1} presents these four types of asymptotic behaviors (self-similar, weakly non-linear, linear and strongly non-linear convection) depending on the behavior of nonlinearity $F$ at the origin.
			
		As we have seen, in the case of linear convection, through a simple change of variables, the equation is reduced to the linear heat equation. Therefore, we will deal primarily with the study of behavior
asymptotic self-similar and weakly non-linear. In Chapter 7 we will describe
briefly the results obtained in \cite{EVZ1, EVZ2, EVZ3}.
		
		In the case of weakly non-linear asymptotic behavior, after obtaining optimal estimates of the decay of the $L^p(\R^N)$ norms of the solutions, the result is easily obtained by working on the integral equation associated with the convection-diffusion equation.
In the self-similar case we will adopt the point of view introduced in \cite{Z1}. First, the self-similar variables are introduced since, in these, the profiles of the self-similar solutions are the stationary solutions of the new parabolic equation and the self-similar behavior of the general solution is equivalent to the convergence of the new trajectories to the stationary solutions. Working within the functional framework of the Sobolev spaces with weight introduced by M. Escobedo and O. Kavian \cite{EK1} , we will test the pre-compactness of the trajectories. Finally we will use a dynamic systems argument that allows us to simultaneously prove the existence of stationary solutions and the convergence of trajectories towards these stationary solutions. At this point, the strong contraction in $L^1 (\R^N)$ which is verified by the semigroup associated with  \eqref{eq: 1.1}, will be fundamental.
The notes are organized as follows. In Chapter 1 we will study the linear heat equation and in Chapter 2 the Burgers equation
in dimension $N = 1$ by means of the Hopf-Cole transformation. In Chapter
3 we will study the linear heat equation in the self-similar variables. In
Chapter 4 will address the Cauchy problem \eqref{eq: 1.1}--\eqref{eq: 1.2} and prove a
result of existence and uniqueness of solutions for data $u_0 \in L^1(\R^N)$, as well
as estimates of solution decay in $L^p(\R^N)$. In Chapter 5 we will prove weakly nonlinear behavior for nonlinearities $F$ that behave like $F(s)=|s|^{q-1}s$
with $q>1+\frac{1}{N}$. 
Chapter 6 we will show the self-similar asymptotic behavior when $q=1+\frac{1}{N}$.
Finally, in Chapter 7, we will describe the main results obtained in \cite{EVZ1, EVZ2, EVZ3} for the case
		$q<1+\frac{1}{N}$, we will mention some extensions of the results of these notes as well as some open problems.
	
	\subsection*{Subsequent works}
	Hereinafter, we present a non exhaustive list of developments in the field of large time asymptotics for convection-diffusion equations since these notes were first written.
		\begin{itemize}
			\item Subsequent studies on scalar convection-diffusion and Hamilton-Jacobi equations include \cite{carpio1996large2, escobedo1997long, duro2000large, vazquez2002complexity,  benachour2004asymptotic};
			\item Large time asymptotics for simplified fluid-structure interaction models are studied in		\cite{vazquez2003large, vazquez2006lack, munnier2005large};
			\item An exhaustive treatise of the large time asymptotics of scalar hyperbolic conservation laws is \cite{serre2002l1}, while viscous conservation laws are addressed in \cite{dalibard2010long};
			\item The large time asymptotics for numerical schemes of such equations is studied in \cite{pozo2014large, ignat2015large};
			\item The large time asymptotics for non-local diffusion-convection equations are studied in \cite{chasseigne2006asymptotic, ignat2007nonlocal, ignat2008refined, cazacu2017asymptotic, ignat2018asymptotic};
			\item The case of degenerate diffusion with or without convective terms is addressed in \cite{feireisl19991, laurenccot1998long, carrillo2004wasserstein}	
			\item For the large time asymptotics of the Navier-Stokes equations, we refer to \cite{carpio1996large, gallay2002invariant}
			\item Finally, for scaling properties and large time asymptotics for damped or dissipative wave equations we refer to  \cite{gallay1998scaling, gallay2000scaling, zuazua2003large}.	
		\end{itemize}

		\chapter[The linear heat equation]{The linear heat equation}
		
		\section{Introduction}
		
		Let us consider the linear heat equation
		\begin{align} \label{eq: 2.1}
		\begin{cases}
		u_t-\Delta u=0 &\hbox{en $\R^N \times(0,\infty)$,}\\
		u(0)=u_0.
		\end{cases}
		\end{align}
		As is well known, the fundamental solution or heat kernel (solution of  \eqref{eq: 2.1} with $u_0=\delta_{}=$ Dirac mass at the origin) is given by the formula
		\begin{equation} \label{eq: 2.2}
			G(x,t)=(4\pi t)^{-N/2} \exp\biggl(\frac{-|x|^2}{4t}\biggr).
		\end{equation}
	Given $u_0\in L^1(\R^N)$ the solution to \eqref{eq: 2.1} is given by convolution with the heat kernel:
		\begin{align} \label{eq: 2.3}
			u(x,t)&=[G(\cdot,t)*u_0(\cdot)](x)\\
			&=(4\pi t)^{-N/2}\int_{\R^N}\exp\biggl(\frac{-|x-y|^2}{4t}\biggr)
			u_0(y)\,dy. \nonumber 
		\end{align}
		It is the unique solution of
		\eqref{eq: 2.1} in the class $u\in C\bigl((0,\infty);L^1(\R^N)\bigr)$. Thanks to
		regularizing effect of the heat kernel it holds $u\in C^\infty
		\bigl(\R^N\times(0,\infty)\bigr)$.
		
		Integrating the equation \eqref{eq: 2.1} in $\R^N$ we obtain
		\begin{equation} \label{eq: 2.4}
			\frac{d}{dt}\int u(x,t)\,dx=0.
		\end{equation}
		(Henceforth  $\int$ denotes the integral in all $\R^N$.) From \eqref{eq: 2.4} it follows that the mass of the solution is conserved over all time.
		
		The question that interests us is to study the way in which the mass of the
solution is distributed in space when $t\to\infty$.

		Calculating the $L^p(\R^N)$ norm of the heat kernel we obtain
		\begin{equation} \label{eq: 2.5}
		\|G(t)\|_p\le C_pt^{-\np},\qquad\forall t>0
		\end{equation}
		for all $p\in[1,\infty]$.
		Therefore, thanks to  the Young inequality we obtain
		\begin{align} \label{eq: 2.6}
			\|u(t)\|_p=\|G(t)*u_0\|_p&\le \|G(t)\|_p\|u_0\|_1 \\
			&\le C_pt^{-\np}\|u_0\|_1,\qquad\forall t>0 \nonumber 
		\end{align}
		for all $p\in[1,\infty]$.
		
		The question at hand can be formulated a little more precisely: we wish to study the behavior of $t^\np u(t)$ in
		$L^p(\R^N)$  when  $t\to\infty$.
		
		As we mentioned in the introduction and as we will prove in the
next section, asymptotic behavior is given by $u_M=MG$, $M$ being the mass of the solution, i.e.,
		\begin{equation} \label{eq: 2.7}
			M=\int u_0(x)\,dx.
		\end{equation}
		This result can be seen by a scaling argument.
		
		Let $u$ be the solution to \eqref{eq: 2.1} with initial data $u_0\in L^1(\R^N)$ of mass
		$M$. It is easy to check that the rescaled function
		$$u_\lambda(x,t)=\lambda^Nu(\lambda x,\lambda^2t)$$
		is a solution to
		\begin{align*}
		\begin{cases}
		u_{\lambda ,t}- \Delta u_\lambda =0 &\rlap{\qquad in
			$\R^N\times(0,\infty)$,}\\
		u_\lambda (x,0)=u_{0,\lambda}(x)=\lambda^Nu_0(\lambda x).
		\end{cases}
		\end{align*}
		
		On the other hand
		$$u_{0,\lambda}\to M\delta \quad \text{  when  } \quad \lambda\to\infty $$
		en the following sentido
		$$\int u_{0,\lambda}(x)\varphi(x)\,dx\to M\varphi(0),
		\quad \text{  when  } \quad \lambda\to\infty, \forall\varphi\in \BC(\R^N). $$
		(Henceforth $\BC(\R^N)$ designates the space continuous and bounded functions in $\R^N$.)
		
		Therefore, it is natural  natural to think that $u_\lambda(t)$ converges to $MG(t)$
		 when  $\lambda\to\infty$ for all $t>0$ as $u_M(x,t)=MG(x,t)$ is the solution to the heat equation with with  initial data $M\delta$.
		
	If this convergence takes place in $L^1(\R^N)$ for $t=1$ it holds
		$$ 
		u_\lambda(\cdot,1)\to MG(\cdot,1) \quad \text{ in } L^1(\R^N), \quad \text{  when  } \lambda\to\infty 
		$$
		and this is equivalent to
		$$
		\|u(t)-MG(t)\|_1\to 0 \quad \text {  when  } t\to\infty. 
		$$
		n the next section we will demonstrate this result rigorously. We will see that
		$$
		u_\lambda(\cdot,1)\to MG(\cdot,1) \quad \text{ in } L^p(\R^N), \quad \text{  when  }
		t\to\infty 
		$$
		for all $1\leq p\leq\infty$.
		This answers the question we had
raised.

		\section{Asymptotic behavior}
		
		The following result holds. 
				
		\begin{thm}
			Let $u_0\in L^1(\R^N)$ be such that $M=\int u_0(x)\,dx$.
			Then, the solution $u=G(t)*u_0$ de
			\eqref{eq: 2.1} satisfies
			\begin{equation} \label{eq: 2.8}
				t^\np\|u(t)-MG(t)\|_p\to 0 \quad \text{  when  }  t\to\infty
			\end{equation}
			for all $p\in[1,\infty]$.
		\end{thm}
		
		The proof of this theorem is based on the following lemma from
		\cite{EZ2}.
		
		\begin{lem}
			(a) 
			For all $p\in[1,\infty]$ there exists a constant $C_p>0$ such that
			\begin{equation} \label{eq: 2.9}
				\|G(t)*\varphi\|_p \le
				C_pt^{-\np-\frac{1}{2}}\|\varphi\|_{L^1(\R^N;|x|)},
				\qquad\forall t>0
			\end{equation}
			for all $\varphi\in L^1(\R^N;1+|x|)$ such that
			$\int\varphi(x)\,dx=0$. \\

			(b) For all $p\in[1,\infty]$ there exists a constant
			$C'_p>0$ such that if $\varphi\in L^1(\R^N;1+|x|)$ satisfies
			$\int\varphi(x)\,dx=M$
			then it holds
			\begin{equation} \label{eq: 2.10}
				\|G(t)*\varphi -MG(t)\|_p\le
				C'_pt^{-\np-\frac{1}{2}}\|\varphi\|_{L^1(\R^N;|x|)},
				\qquad\forall t>0.
			\end{equation}
		\end{lem}
		
		\begin{obs}
			We used the space $L^1$ with weights:
			$$L^1(\R^N;1+|x|)=\{\,\varphi\in L^1(\R^N): \int |\varphi(x)|(1+|x|)\,dx
			<\infty\,\}.$$
			The $\|\cdot\|_{L^1(\R^N;|x|)}$ norm is, by definition,
			$\|\varphi\|_{L^1(\R^N;|x|)}=\int|\varphi(x)||x|\,dx$.
		\end{obs}
		
		\begin{proof}[Proof of Lemma 1.2] 

			(a) It holds
				$$\bigl(G(t)*\varphi)(x)=(4\pi t)^{-N/2} \int\exp\biggl(
				-\frac{|x-y|^2}{4t}\biggr)\varphi(y)\,dy$$
				and since $\int\varphi(x)\,dx=0$, we obtain
				$$\bigl(G(t)*\varphi\bigr)(x)=(4\pi t)^{-N/2}
				\int\biggl[\exp\biggl(-\frac{|x-y|^2}{4t}\biggr)
				-\exp\biggl(-\frac{|x|^2}{4t}\biggr)\biggr]\varphi(y)\,dy.$$
				By Taylor's theorem we know that 
								$$\exp\biggl(-\frac{|x-y|^2}{4t}\biggr)
				-\exp\biggl(-\frac{|x|^2}{4t}\biggr)=\frac{1}{2t}
				\int_0^1 y\cdot(x-\theta y)
				\exp\biggl[-\frac{(x-\theta y)^2}{4t}\biggr]\,d\theta$$
				and therefore
				\begin{align} \label{eq: 2.11}
					&\bigl(G(t)*\varphi\bigr)(x) \\
					&=\frac{(4\pi  t)^{-N/2}}{\sqrt{t}}\int_0^1\!\!\!\int
						\frac{y\cdot(x-\theta y)}{2\sqrt{t}}\exp\biggl(-\frac{|x-\theta
							y|^2}{4t}\biggr)\varphi(y)\,dy\,d\theta. \nonumber 
				\end{align}
				Taking $L^\infty$ norms it holds:
				\begin{align*}
					&\|G(t)*\varphi\|_\infty \\
					&\le (4\pi)^{-N/2}
					t^{-N/2-1/2}\|\varphi\|_{L^1(\R^N;|x|)}{\sup_{x,y\in \R^N}
						\int_0^1\biggl\{\frac{|x-\theta y|}{2\sqrt{t}}
						\exp \biggl(-\frac{|x-\theta y|^2}{4t}\biggr)\,d\theta\biggr\}}\\
					&\le(4\pi)^{-N/2} \| |z|\exp (-|z|^2)\|_\infty t^{-N/2-1/2}
					\|\varphi\|_{L^1(\R^N; |x|)}.
				\end{align*}
				Therefore we obtain \eqref{eq: 2.9} for $p=\infty$ with $C_\infty=(4\pi)^{-N/2} \|
				|z| \exp (-|z|^2) \|_\infty$.

				Taking $L^1$ norms in \eqref{eq: 2.11} and using the Fubini theorem we obtain:
				\begin{align*}
					&\|G(t)*\varphi\|_\infty \\
					&\le(4\pi)^{-N/2} t^{-N/2-1/2} {\int_0^1\!\!\!\int 2\frac{|x-\theta y|}{2\sqrt{t}} \exp \biggl(-\frac{|x-\theta y|^2}{4t}\biggr)
						|y\varphi(y)|\,dx\,dy\,d\theta} \\
					&\le\pi^{-N/2}t^{-1/2}\| |z|\exp(-|z|^2)\|_1 \int_0^1\!\!\!\int |y|
					|\varphi(y)|\,dy\,d\theta\\
					&\le\pi^{-N/2}\| |z|\exp(-|z|^2)\|_1  \|\varphi\|_{L^1(\R^N;
						|x|)}t^{-1/2}, \qquad\forall t>0.
				\end{align*}
				therefore, we obtain \eqref{eq: 2.9} for $p=1$ with $C_1=\pi^{-N/2}
				\| |z|\exp(-|z|^2)\|_1$.
				
				The general case is obtained by interpolation:
				\begin{align*}
					&\|G(t)*\varphi\|_p \le \|G(t)*\varphi\|_\infty^{(p-1)/p}
					\|G(t)*\varphi\|_1^{1/p}\\
					&\le 2^{-N(p-1)/p}\pi^{-N/2} \| |z|\exp(-|z|^2)\|_\infty^{(p-1)/p}
				{\| |z|\exp(-|z|^2)\|_1^{1/p} \|\varphi\|_{L^1(\R^N;|x|)}
					t^{-\np-\frac{1}{2}}.}
				\end{align*}
				That is to say, we obtain \eqref{eq: 2.9} with
				$$C_p=2^{-N(p-1)/p}\pi^{-N/2} \| |z|\exp(-|z|^2)\|_\infty^{(p-1)/p}
				\| |z|\exp(-|z|^2)\|_1^{1/p}.$$
				This concludes the proof of part (a).
				
				(b) The proof of this part is analogous to the previous one. Indeed, since $\int\varphi(x)\,dx=M$,
				\begin{align*}
					&\bigl(G(t) * \varphi\bigr) (x) -MG(x,t)\\
					&=(4\pi t)^{-N/2}
					\biggl[\int\exp\biggl(-\frac{|x-y|^2}{4t}\biggr) \varphi (y) \,dy
					{-M\exp\biggl(-\frac{|x|^2}{4t}\biggr)\biggr]}\\
					&=(4\pi t)^{-N/2}
					\int\biggl[\exp\biggl(-\frac{|x-y|^2}{4t}\biggr)
					{-\exp\biggl(-\frac{|x|^2}{4t}\biggr)\biggr]\varphi(y)\,dy}
				\end{align*}
				and therefore we can just redo the estimates from part (a).
		\end{proof}
		
		\begin{obs}
			Lemma 1.2 can be proven easily using results from
			J.~Duoandikoetxea and E.~Zuazua \cite{DZ} which ensures that for all
			$u_0\in L^1 (\R^N; 1+|x|)$ there exists $f\in\bigl(L^1(\R^N)\bigr)^N$ such that
			$$u_0=M\delta+\div f$$
			with $M=\int u_0\,dx$. The vector-valued function $f$ can be computed explicitly. It holds
			$$f_j(x)=\int_0^1 x_j u_0(tx)t^{N-1}\,dt.$$
			From this expression one deduces that  $\|f\|_{(L^1(\R^N))^N}\le
			C\|u_0\|_{L^1(\R^N;|x|)}$.
		\end{obs}
		
		\begin{proof}[Proof of Theorem 1.1]
			
			Lemma 1.2 provides \eqref{eq: 2.8} in the case when $u_0\in L^1(\R^N; 1+|x|)$. We obtain the general case by density. Given arbitrary $u_0\in L^1(\R^N)$ with $\int u_0(x)\,dx=M$ we construct a sequence $\varphi_n\in
			L^1(\R^N;1+|x|)$ such that $\int\varphi_n(x)\,dx=M$ For all
			$n\in \N$ y
			\begin{equation} \label{eq: 2.12}
				\varphi_n\to u_0 \quad \text{ in } L^1(\R^N) \text{  when  } n\to\infty. 
			\end{equation}
			We have
			\begin{align} \label{eq: 2.13}
			&{t^\np \|G(t) * u_0 - MG(t)\|_p}\\
			&\le t^\np \{ \|G(t) *
				\varphi_n - MG(t)\|_p + \|G(t) * (\varphi_n -u_0)\|_p\}. \nonumber 
			\end{align}
			By virtue of \eqref{eq: 2.6} we have
			$$t^\np\|G(t) * (\varphi_n -u_0)\|_p\le C_p\|\varphi_n -u_0\|_1$$
			and therefore, thanks to  \eqref{eq: 2.12}, given arbitrary $\epsilon>0$, there exists
			$k\in\N$ such that 
			\begin{equation} \label{eq: 2.14}
			\sup_{t>0}\Bigl\{t^\np \|G(t) * (\varphi_k -u_0)\|_p\Bigr\}\le
			\epsilon/2. 
			\end{equation}
			On the other hand, thanks to Lemma 1.2, for fixed $k$,
			\begin{equation*} 
				t^\np \|G(t) * \varphi_k -MG(t)\|_p\to 0 \quad \text{  when  } t\to\infty 
			\end{equation*}
			and therefore, for $t_0>0$ sufficiently large it holds
			\begin{equation} \label{eq: 2.15}
				t^\np \|G(t) * \varphi_k -MG(t)\|_p\le \epsilon, \qquad\forall t\ge t_0
			\end{equation}
			which implies \eqref{eq: 2.8}.
		\end{proof}
		
		\section{The heat equation with linear convection}
		
		Let $a$ be a constant vector of $\R^N$, and consider the following linear convection-diffusion equation:
		\begin{align} \label{eq: 2.16}
		\begin{dcases}
			u_t-\Delta u = a\cdot \nabla u &\text{in } \R^N \times(0,\infty) \\
			u(0)= u_0\in L^1(\R^N).
		\end{dcases}
		\end{align}
		It can easily be shown that if $u$ satisfies \eqref{eq: 2.16} then
		\begin{equation} \label{eq: 2.17}
			v(x,t)=u (x-at,t)
		\end{equation}
		satisfies
		\begin{align} \label{eq: 2.18}
		\begin{cases}
				v_t-\Delta v =0  &\hbox{in $\R^N\times(o,\infty)$}\\
				v(0)=u_0.
		\end{cases}
		\end{align}
		As the unique solution to \eqref{eq: 2.18} in $C\bigl([0,\infty); L^1(\R^N)\bigr)$
		is given by
		$$v(x,t)=[G(t)*u_0] (x)$$
		undoing the change of variables \eqref{eq: 2.17} we obtain that the unique solution
		of \eqref{eq: 2.16} in $C\bigl([0,\infty); L^1(\R^N)\bigr)$ is the following:
		\begin{equation} \label{eq: 2.19}
			u(x,t)=[G(t)*u_0] (x+at).
		\end{equation}
		From this expression for $u$ we deduce  that the solutions to \eqref{eq: 2.16} have the same decay properties in $L^p$ as the solutions to the
		linear heat equation, i.e., the solutions to \eqref{eq: 2.16} satisfy \eqref{eq: 2.6}.
		
		From Theorem 1.1 we obtain the following result which provide
		the asymptotic behavior of the solutions to \eqref{eq: 2.16}.
		
		\begin{thm}
			For all $u_0\in L^1(\R^N)$ such that $\int u_0(x)\,dx=M$ la
			solution $u$ to \eqref{eq: 2.16} satisfies
			\begin{equation}
				t^\np\|u(x,t)-MG(x+at,t)\|_p\to 0, \quad \text{  when  } t\to\infty
			\end{equation}
			for all $p\in[1,\infty]$.
		\end{thm}
		
		\chapter{The Burgers equation}
		
		\section{The Hopf-Cole transform}
		
		Let us consider the viscous Burgers equation:
		\begin{align} \label{eq: 3.1}
		\begin{cases}
			u_t-u_{xx}= (u^2)_x &\hbox{en $\R\times(0,\infty)$}\\
			u(x,0)=u_0(x)&\hbox{en $\R$.}
		\end{cases}
		\end{align}
		At the beginning of the 1950s, E.~Hopf \cite{Ho} and J.~D.~Cole \cite{Co}
		independently discovered that the equation \eqref{eq: 3.1} can be
reduced to the linear heat equation by means of a simple change of
variables This change of variables is known since the {\em Hopf-Cole transform}.
		
		Given $u$ solution to \eqref{eq: 3.1} we define 
		$$v(x,t)=\int_{-\infty}^x u(z,t)\,dz.$$
		
		It can easily be shown that  $v$ satisfies
		\begin{align*}
		\begin{dcases}
			v_t-v_{xx}=(v_x)^2 &\hbox{en $\R\times(0,\infty)$}\\
			v(x,0)=v_0(x)=\int_{-\infty}^x u_0(z)\,dz
		\end{dcases}
		\end{align*}
		We define $w(x,t)=\exp\bigl(v(x,t)\bigr)$. So, $w$ is a solution to
		the equation heat equation:
		\begin{align} \label{eq: 3.2}
		\begin{dcases}
			w_t-w_{xx}= 0 &\hbox{en $\R\times(0,\infty)$}\\
			w(x,0)=w_0(x)=\exp\biggl(\int_{-\infty}^x u_0(z)\,dz\biggr) 
		\end{dcases}
		\end{align}
		We see therefore that if $u$ is solution to \eqref{eq: 3.1}, so
		\begin{equation} \label{eq: 3.3}
			w(x,t)=\exp\biggl(\int_{-\infty}^x u(z,t)\,dz\biggr) 
		\end{equation}
		satisfies \eqref{eq: 3.2}.
		
		On the other hand we see that if $u_0\in L^1(\R)$ with
		${\int_{-\infty}^\infty u_0(x)\,dx=M}$ and $\|u_0\|_1=L$
		then
		\begin{subequations}
			\label{eq: 3.4}
			\begin{align} 
			&w_0\in C(\R)\cap L^\infty(\R) \label{eq: 3.4a}  \\
			&e^{-|L|}\le w_0(x)\le e^{|L|},\qquad\forall x\in\R \label{eq: 3.4b}  \\ 
			&\lim_{x\to-\infty} w_0(x)=1 \label{eq: 3.4c}  \\
			&\lim_{x\to+\infty} w_0(x)=e^M \label{eq: 3.4d}  
			\end{align}
		\end{subequations}
		The solution $w$ to \eqref{eq: 3.2} is given by convolution with the one-dimensional heat kernel
		\begin{equation} \label{eq: 3.5}
			w(x,t)=[G(\cdot,t)*w_0(\cdot)](x) 
		\end{equation}
		where
		$$G(x,t)=(4\pi t)^{-1/2}\exp\biggl(-\frac{|x|^2}{4t}\biggr).$$
		
		Note that  $w_0$ does not belong to  $L^1(\R)$. Therefore, the
		$L^p$ estimates obtained in the previous chapter for the solutions
		of the linear heat equation with initial data in $L^1$ do not apply to $w$.
		
		As $w_0\in L^\infty(\R)$, applying the Young inequality we obtain
		\begin{equation} \label{eq: 3.6}
			\|w(t)\|_\infty \le \|G(t)\|_1 \|w_0\|_\infty = \|w_0\|_\infty,
			\forall t>0
		\end{equation}
		On the other hand, combining \eqref{eq: 3.4a}--\eqref{eq: 3.4d} with the explicit expression of
		$w$:
		$$w(x,t)=(4\pi t)^{-1/2} \int_{-\infty}^\infty
		\exp\biggl(-\frac{|x-y|^2}{4t}\biggr)w_0(y)\,dy$$
		we easily obtain
		
		\begin{subequations}
			\label{eq: 3.7}
			\begin{align}
				&e^{-|L|}\le w(x,t) \le e^{|L|}, \quad \forall (x,t)\in
						\R\times(0,\infty) \label{eq: 3.7a}\\
				&\lim_{x\to-\infty} w(x,t)=1,\quad \lim_{x\to+\infty}
						w(x,t)=e^M, \quad \forall t>0 \label{eq: 3.7b}
			\end{align}
		\end{subequations}
		Inverting the transformation \eqref{eq: 3.3} we obtain
		\begin{equation} \label{eq: 3.8}
			u(x,t)=\Bigl(\log\bigl(w(x,t)\bigr)\Bigr)_x=
			\frac{w_x(x,t)}{w(x,t)}.
		\end{equation}
		Differentiating the equation \eqref{eq: 3.2} with respect to $x$ we see that $\omega=w_x$ is
		solution to
		\begin{align} \label{eq: 3.9} 
		\begin{dcases}
		\omega_t-\omega_{xx}=0 &\text{ in } \R\times(0,\infty)\\
		\omega(x,0)=\omega_0(x)=w_0(x) u_0(x). 
		\end{dcases}
		\end{align}
		As $w_0\in L^\infty(\R)$ and $u_0\in L^1(\R)$ We see that $\omega_0\in
		L^1(\R)$. Thus, the $L^p$ estimates obtained in Chapter 1 for the solutions to the heat equation with initial data in $L^1(\R)$ can be applied to $\omega$ and we obtain
		\begin{equation} \label{eq: 3.10}
			\|\omega(t)\|_p \le C_p t^{-\op}, \qquad\forall t>0 
		\end{equation}
		for all $p\in[1,\infty]$.
		
		Combining \eqref{eq: 3.7a} and \eqref{eq: 3.10} we deduce that the solution $u=u(x,t)$ to (i)
		satisfies
		\begin{equation} \label{eq: 3.11}
			\|u(t)\|_p\le C'_p t^{-\op},\qquad\forall t>0
		\end{equation}
		for all $p\in[1,\infty]$.
		
		On the other hand, integrating the equation \eqref{eq: 3.1} in $\R$ with respect to $x$ one obtains
				$$\frac{d}{dt}\int_{-\infty}^\infty u(x,t)\,dx=0.$$
		Thus the mass of the solution to \eqref{eq: 3.1} is conserved with time.
		
		The Hopf-Cole transform allowed us to obtain an explicit form for the solutions to the Burgers equation, as well as $L^p (\R)$ estimates which hold that for the heat equation. 
		In the next sections we will see that this transformation also allows obtaining self-similar solutions to \eqref{eq: 3.1}, as well as the asymptotic behavior of the general solution.
		
		\section{Self-similar solutions}
		
		The equation \eqref{eq: 3.1} is invariant under the rescaling
		\begin{equation} \label{eq: 3.12}
			u_\lambda(x,t)=\lambda u(\lambda x,\lambda^2 t).		
		\end{equation}
		
		Therefore, the problem arises of the existence of self-similar solutions that remain invariant, i.e., that satisfy
		\begin{equation} \label{eq: 3.13}
			u_\lambda=u,\qquad \forall \lambda>0.
		\end{equation}
		It can easily be shown  that $u$ satisfies \eqref{eq: 3.2} if and only if 
		\begin{equation} \label{eq: 3.14}
			u(x,t)=t^{-1/2} f\delfrac({x}{\sqrt{t}})   
		\end{equation}
		with $f(x)=u(x,1)$.
		
		Any solution to \eqref{eq: 3.1} of the form \eqref{eq: 3.1} with profile $f\in L^1(\R)$ es
		solution the Cauchy problem:
		\begin{align} \label{eq: 3.15}
		\begin{dcases}
				u_t-u_{xx}=(u^2)_x &\text{ in } \R\times(0,\infty)\\
				u(x,0)= M\delta 
		\end{dcases}
		\end{align}
		with ${M=\int_{-\infty}^\infty f(x)\,dx}$ and $\delta$ being the mass Dirac at the origin.
				
		For every $M\in\R$, the existence and uniqueness of the solution to \eqref{eq: 3.4} can be expected, with the solution being self-similar.
		
		Therefore, for every $M\in\R$ we look for $f_M\in L^1(\R)$ with
		${\int_{-\infty}^\infty f(x)\,dx=M}$ and such that the function
		$u_M(x,t)=t^{-1/2} f_M(x/\sqrt{t})$ satisfies \eqref{eq: 3.4}.
		
		It can easily be shown that  $u_M=t^{-1/2} f_M(x/\sqrt{t})$ satisfies \eqref{eq: 3.4} if and only if	
		\begin{align} \label{eq: 3.16}
		\begin{dcases}
			-f''_M-\frac{xf'_M}{2}-\frac12 f_M=(f^2_M)'&\hbox{en $\R$}\\
			\int_{-\infty}^\infty f_M=M. 
		\end{dcases}
		\end{align}
		This differential equation is easily solved using the Hopf-Cole transform
		We define 
		\begin{equation} \label{eq: 3.17}
			g_M(x)=\exp \biggl(\int_{-\infty}^x f_M(s)\,ds\biggr).
		\end{equation}
		Then \eqref{eq: 3.5} transforms to
		\begin{align} \label{eq: 3.18}
		\begin{dcases}
			-g''_M-\frac{xg'_M}{2}=0 &\text{ in } \R\\
			\lim_{x\to-\infty} g_M(x)=1, \quad \lim_{x\to-\infty} g(x)=e^M.  
		\end{dcases}
		\end{align}
		If we write
		\begin{equation} \label{eq: 3.19}
			g_M(x)=(e^M-1)\int_{-\infty}^x \int_{-\infty}^x h(s)\,ds+1
		\end{equation}
		we see that \eqref{eq: 3.7} is equivalent to
		\begin{align} \label{eq: 3.20}
		\begin{dcases} 
				-h_{xx}-\frac{xh_x}{2}-\frac12 h=0 &\text{ in } \R \\
				\int_{-\infty}^\infty h\,dx=1.
		\end{dcases}
		\end{align}
		Hence, \eqref{eq: 3.19} is the equation for the profile of the self-similar solution of the linear heat equation and thus :
		\begin{equation} \label{eq: 3.21}
			h(x)=(4\pi)^{-1/2} \exp(-x^2/4).
		\end{equation}
		On the other hand, undoing the change \eqref{eq: 3.16} we obtain the
		self-similar solution to \eqref{eq: 3.15}:
		\begin{equation} \label{eq: 3.22}
		f_M(x)=\Bigl(\log\bigl(g_M(x)\bigl)\Bigl)_x = \frac{g'_M(x)}{g_M(x)}
		=\frac{(e^M-1)h(x)}{\displaystyle(e^M-1) \int_{-\infty}^x h(s)\,ds+1}.  
		\end{equation}
		The self-similar solutions to \eqref{eq: 3.1} are thus of the form:
		\begin{equation} \label{eq: 3.23}
			u_M(x,t)=t^{-1/2}\frac{(e^M-1)h(x/\sqrt{t})}{\displaystyle(e^M-1)
				\int_{-\infty}^{x/\sqrt{t}} h(s)\,ds+1}.
		\end{equation}	
The main difference observed in the self-similar profiles of the Burgers equation with respect to the self-similar profiles of the heat equation is its asymmetry with respect to the origin. This is due to the convection term.
		
		When $M$ increases, $f_M(x)$ increases and  when  $M$ goes to infinity $f_M$
		converges to 
		\begin{equation}
			f_\infty(x)=\frac{h(x)}{\displaystyle\int_{-\infty}^x h(s)\,ds}.
		\end{equation}
		The function $f_\infty$ satisfies: $\lim_{x\to\infty} f_\infty (x)=0$ y
		$\lim_{x\to-\infty} f_\infty(x)=\infty$.
		
		The function $u_\infty(x,t)=t^{-1/2} f_\infty(x/\sqrt{t})$ is a self-similar solution of the Burgers equation with profile $f_\infty$ which does not belong to $L^1$.
		
		\section{Asymptotic behavior}
		
		Given $u_0\in L^1(\R)$, we have seen that the Burgers equation \eqref{eq: 3.1} admits a unique solution that we have explicitly obtained by means of the Hopf-Cole transformation.
	On the other hand we have seen  that $t^\op u(t)
		\in L^\infty \bigl(0,\infty; L^p(\R)\bigr)$ for all $p\in[1,\infty]$. Therefore, the problem of the asymptotic behavior of solutions arises or, more precisely, the behavior of $t^\op u(t)$ in
		$L^p(\R)$  when  $t\to\infty$.
		
		The following theorem shows that the asymptotic behavior is given by the self-similar solutions constructed in the previous section.
		
		\begin{thm}
	Let $u_0\in L^1(\R)$ with ${\int_{-\infty}^\infty u_0(x)\,dx=M}$. Let
			$u=u(x,t)$ be the solution to \eqref{eq: 3.1} and $u_M=u_M(x,t)$ the self-similar solution of mass $M$ obtained in \eqref{eq: 3.22}.
Then
			\begin{equation} \label{eq: 3.24}
				t^\op\|u(t)-u_M(t)\|_p\to 0, \quad \text{  when  } t\to\infty
			\end{equation}
			for all $p\in[1,\infty]$.
		\end{thm}
		
		\begin{proof}
			As we saw in \eqref{eq: 3.8}, ${u=\frac{w_x}{w}}$ where $w$ is the solution to the heat equation \eqref{eq: 3.2}.
			
	On the other hand $u_M$ is of the form
			${u_M=\frac{w_{M,x}}{w_M}}$ with 
			\begin{equation*} 
			w_M(x,t)=g_M(x/\sqrt{t})=(e^M-1)\int_{-\infty}^{x/\sqrt{t}} h(s)\,ds+1,
			\end{equation*}
			$h$ being the gaussian defined in \eqref{eq: 3.20}.
			
			Thanks to\eqref{eq: 3.7} we know that
			\begin{equation} \label{eq: 3.25}
				e^{-|L|}\le w(x,t)\le e^{|L|},\qquad \forall (x,t)\in
				\R\times(0,\infty) 
			\end{equation}
			and on the other hand, from the expression of $w_M$ we can deduce
			\begin{equation} \label{eq: 3.26}
			e^{-M^-} \le w_M(x,t) \le e^{M^+} 
			\end{equation}
			with $M^+=\max(M,0)$ and $M^- = -\min(M,0)$.
			
			Combining the forms of $u$ and $u_x$ with \eqref{eq: 3.25}--\eqref{eq: 3.26} we see that \eqref{eq: 3.24} is equivalent to
			\begin{equation} \label{eq: 3.27}
				t^\op \|w_x(t) w_M(t)-w(t) w_{M,x} (t)\|_p\to 0 \quad \text{  when  } t\to\infty
			\end{equation}
			which is the same as:
			\begin{align*}
				{t^\op \bigl\|\bigl(w_x(t)-w_{M,x} (t)\bigr) w_M(t)-\bigl(w(t) -
					w_M(t)\bigr) w_{M,x} (t)\bigr\|_p\to 0} \quad {\hbox{ }}
			\end{align*}
			when  $t\to\infty$.
			Thanks to \eqref{eq: 3.26} we know that $w_M \in
			L^\infty\bigl(\R\times(0,\infty)\bigr)$. On the other hand
			$$w_{M,x}=t^{-1/2} (e^M-1) h\delfrac({x}{\sqrt{t}})$$
			satisfies $t^\op w_{M,x} (t) \in L^\infty\bigl(0,\infty; L^p(\R)\bigr)$.
			
			It is therefore enough to prove that
			\begin{align} \label{eq: 3.28}
					t^\op \|w_x(t)-w_{M,x}(t)\|_p &\to0\\
					\|w(t)-w_M(t)\|_\infty &\to  0.
			\end{align}
			
			We see that $w_x$ is solution to the linear heat equation with
			mass $e^M-1$ as, for  \eqref{eq: 3.7b}, we have
			$$\int_{-\infty}^\infty w_x(x,t)\,dx=\lim_{x\to+\infty} w(x,t) -
			\lim_{x\to-\infty} w(x,t)=e^M-1.$$
			
			On the other hand, as $w_{M,x}$ is the fundamental solution of the heat equation with mass $e^M-1$, \eqref{eq: 3.27} is a consequence of Theorem 1.1.
			
			On the other hand,
			$$w(x,t)-w_M(x,t)=\int_{-\infty}^x \bigl(w_x(s,t) - w_{M,x}
			(s,t)\bigr)\,ds$$
			and
			$$\lim_{x\to-\infty} w(x,t)=\lim_{x\to-\infty} w_M(x,t)=1.$$
			Thus  \eqref{eq: 3.28} is a consequence of \eqref{eq: 3.27} for $p=1$.
		\end{proof}
		
		\begin{obs}
			The equation
			\begin{equation} \label{eq: 3.30}
				u_t-u_{xx} = a(u^2)_x \quad \text{ in } \R\times(0,\infty)
			\end{equation}
			with $a\ne 1$ can easily be reduced to the case $a=1$.
			
			Indeed, if $u$ is solution to \eqref{eq: 3.30} with initial data $u_0$, $v=au$ is solution to			\begin{equation} \label{eq: 3.31}
				v_t-v_{xx} = (v^2)_x \quad  \text{ in } \R\times(0,\infty)
			\end{equation}
			con initial data $au_0$.
			
			This allows us to calculate the self-similar solutions to \eqref{eq: 3.30}. Let
			$f_{M,a}$ be the self-similar profile of the equation \eqref{eq: 3.30} with mass $M$. 
		It holds $f_{M,a}=(1/a) f_{aM}$ where $f_{aM}$ is the self-similar profile of mass $aM$ of the equation \eqref{eq: 3.31}.
			
			By virtue of \eqref{eq: 3.22} we have 
			\begin{equation*}
				f_{M,a}(x)=\frac{(e^{aM}-1) h(x)}{\displaystyle a\biggl[(e^{aM}-1)
					\int_{-\infty}^x h(s)\,ds+1\biggr]}.
			\end{equation*}
		\end{obs}
		
		\chapter{The heat equation in similarity variables}
		
		\section{Motivation: the self-similar variables}
		
		As we saw in the introduction, the convection-diffusion equation
		$$u_t-\Delta u=a\cdot\nabla(|u|^{q-1}u) \quad \text{ in } \R^N\times(0,\infty). $$
		can only admit self-similar solutions with profile in $L^1(\R^N)$ si
		${q=1+\frac{1}{N}}$. In this case we obtain the equation
		\begin{equation} \label{eq: 4.1}
			u_t-\Delta u=a\cdot\nabla(|u|^{1/N}u) \quad \text{ in } \R^N\times(0,\infty)
		\end{equation}
	and the self-similar solutions are of the form
		\begin{equation} \label{eq: 4.2}
			u(x,t)= t^{-N/2}f\biggl(\frac{x}{\sqrt{t}}\biggr) 
		\end{equation}
		with a profile $f$ that satisfies the elliptic equation
		\begin{equation} \label{eq: 4.3}
			-\Delta f-\frac{x\cdot\nabla f}{2}-\frac{N}{2} f=a\cdot\nabla
			(|f|^{1/N}f) \quad \text{ in } \R^N.
		\end{equation}

In the study of self-similar solutions and asymptotic behavior, it is convenient to introduce new space-time variables so that the self-similar profiles become stationary solutions of the new evolution equation, and so that the self-similar asymptotic behavior of the general solution is equivalent to the convergence of the new trajectories to the self-similar profiles.

This point of view has been adopted by various authors in the study of different problems. Let us cite, among others, the works of M. ~ Escobedo and O. ~ Kavian \cite{EK1, EK2} on the asymptotic behavior of parabolic equations with terms of absorption and the work of O. ~ Kavian \cite{K} on the heat equation with a reaction term.
		
		In the case of the equation \eqref{eq: 4.1} that concerns us, the similarity variables are
		\begin{equation} \label{eq: 4.4}
			y=\frac{x}{\sqrt{t+1}} \quad	s=\log(t+1).
		\end{equation}
		If $u$ is solution to \eqref{eq: 4.1} with initial data $u(0)=u_0$, then
		\begin{equation} \label{eq: 4.5}
		v(y,s)=e^{Ns/2}u(e^{s/2}y,e^s-1)
		\end{equation}
		is solution to
		\begin{equation} \label{eq: 4.6}
			\begin{cases}
				v_s-\Delta v-\frac{y\cdot\nabla v}{2}-\frac{N}{2}=a\cdot\nabla(|v|^{1/N} v) &\hbox{en $\R^N\times(0,\infty)$}\\
			 v(y,0)=u_0(y).
			\end{cases}
		\end{equation}
		In \eqref{eq: 4.6}, $\nabla$ y
		$\Delta$ respectively represent the gradient and the Laplacian in the new spatial variable $y$.
		
		Reciprocally, if $v=v(y,s)$ is solution to \eqref{eq: 4.6}, then
		\begin{equation} \label{eq: 4.7}
			u(x,t)=(t+1)^{-\frac{N}{2}}v\bigl(x/\sqrt{t+1},\log(t+1)\bigr)
		\end{equation}
		es solution to \eqref{eq: 4.1} y
		\begin{equation} \label{eq: 4.8}
		\biggl\|u(x,t)-t^{-N/2}f\delfrac({x}{\sqrt{t}})\biggr\|_p\to 0 \quad \text{  when  }
		t\to\infty
		\end{equation}
		if and only if
		\begin{equation} \label{eq: 4.9}
			v(s)\to f \quad \text{ in } L^p(\R^N) \quad \text{  when  } s\to\infty.
		\end{equation}
		Thus, the existence of self-similar solutions for \eqref{eq: 4.1}
		and the asymptotic self-similar behavior of its solutions transform into, respectively, the existence of stationary solutions for \eqref{eq: 4.6} and the convergence of the latter's trajectories to equilibrium points.
		
		The equation \eqref{eq: 4.6} is of parabolic type and the elliptic operator involved is
		\begin{equation} \label{eq: 4.10}
			Lv:=-\Delta-\frac{y\cdot\nabla v}{2}.
		\end{equation}
		We observe that
		\begin{equation} \label{eq: 4.11}
			Lv=-\frac{1}{K(y)}\div\bigl(K(y)\nabla v(y)\bigr)
		\end{equation}
		where
		\begin{equation} \label{eq: 4.12}
				K(y)=\exp\delfrac({|y|^2}{4}).
		\end{equation}
		We see therefore that $L$ is symmetric with respect to the scalar product
		$$(v,w)_K=\int v(y)w(y)K(y)\,dy$$
		es decir, it holds
		$$(Lv,w)_K=(v,Lw)_K.$$
		therefore it is natural to introduce the weighted Sobolev spaces
		\begin{align*} 
			L^2(K)&=\biggl\{\,v\in L^2(\R^N):\int|v(y)|^2K(y)\,dy<\infty\,\biggr\}, \\
			H^m(K)&=\{\,v\in H^m(\R^N):D^\alpha v\in L^2(K),\forall\alpha
			\in{\mathbf{N}}^N:|\alpha|\le m\,\},\qquad m\in{\N} \nonumber
		\end{align*}
		endowed with the norms
		\begin{align*} 
			\|v\|^2_{m,K}&=\sum_{|\alpha|\le m}\|D^{\alpha}v\|^2_K, \qquad m=1,2\ldots \nonumber\\
			\|v\|^2_K&=\int |v(y)|^2 K(y)\,dy.
		\end{align*}
		These are Hilbert spaces.

		\section{Analysis in weighted Sobolev spaces}
		
		In this part we prove the fundamental properties of the operator $L$ in the Sobolev spaces $H^m(K)$. These results were shown by M.~Escobedo and O.~Kavian \cite{EK1} and O.~Kavian \cite{K}.
		
		\begin{prop}[\cite{EK1}]
			\begin{enumerate}
				\item[(i)]
				It holds
				\begin{equation} \label{eq: 4.13}
				\int|v(y)|^2|y|^2K(y)\,dy\le 16\int |\nabla v(y)|^2 K(y)\,dy, 
				\qquad\forall v\in H^1(K).
				\end{equation}
				
				\item[(ii)]
				The inclusion $H^1(K)\subset L^2(K)$ is compact.
				
				\item[(iii)]
				$L:H^1(K)\to \bigl(H^1(K)\bigr)^*$
				is an isomorphism.
				
				\item[(iv)]
				$L^{-1}:L^2(K)\to L^2(K)$ is compact.
				
				\item[(v)] The eigenvalues of $L$ are positive real numbers.
				\begin{equation} \label{eq: 4.14}
				\lambda_k=\frac{N+k-1}{2},\qquad k=1,2,\ldots  
				\end{equation}
				
				The corresponding eigenspaces are				\begin{equation} \label{eq: 4.15}
					N(L-\lambda_kI)=\Span\{\,D^\beta\varphi_1:|\beta|=k-1\,\} 
				\end{equation}
				with ${\varphi_1(y)=K^{-1}(y)=\exp\biggl(-\frac{|y|^2}{4}\biggr)}$.
				
				\item[(vi)]
				For all $\epsilon>0$ and $q\ge1$ there exists $R>0$ and a constant
				$C=C(R,q,\epsilon)>0$ such that
				\begin{equation} \label{eq: 4.16}
				\|v\|^2_K\le\epsilon \|v\|^2_{1,K}+C\|v\|^2_{L^q(B_R)},
				\qquad\forall v\in H^1(K)\cap L^q_{\rm loc}(\R^N). 
				\end{equation}
			\end{enumerate}
		\end{prop}
		
		\begin{proof}
			\begin{enumerate}
				\item[(i)]
				Given $v\in H^1(K)$ set
				$w=K^{1/2}v$. It holds ${\nabla w-\frac{y}{4}
					w=K^{1/2}\nabla v}$. As $K^{1/2}\nabla v\in
				\bigl(L^2(\R^N)\bigr)^N$ We have ${\nabla w- \frac{y}{4} w\in
					\bigl(L^2(\R^N)\bigr)^N}$. 
				Besides,
				\begin{align} \label{eq: 4.17}
					\int\biggl|\nabla w-\frac{y}{4} w\biggr|^2&=\int\biggl(|\nabla w|^2+
					\frac{|y|^2}{16}|w|^2-w \frac{y\cdot\nabla w}{2}\biggr)\\
					&=\int\biggl(|\nabla w|^2+\frac{|y|^2}{16}|w|^2+\frac{N}{4}
					|w|^2\biggr). \nonumber 
				\end{align}
				In particular $\int|y|^2|w|^2=\int|y|^2|v|^2K$, and therefore:
				$$\int |y|^2|v(y)|^2K(y)\,dy \le 16\int |\nabla v(y)|^2K(y)\,dy.$$
				
				\item[(ii)] Let $v_n\in H^1(K)$ be such that $v_n\to v$ weakly in $H^1(K)$
				and we see that $\|v_n-v\|_K\to 0$. By virtue of \eqref{eq: 4.13} it holds
				\begin{align} \label{eq: 4.18}
					\|v_n-v\|^2_K&\le\exp\delfrac({R^2}{4})\|v_n-v\|^2_{L^2(B_R)}
					+\frac{1}{R^2}\int|y|^2|v(y)|^2 K(y)\\
					&\le\exp\delfrac({R^2}{4})\|v_n-v\|^2_{L^2(B_R)}
					+\frac{16}{R^2}\|v_n-v\|^2_{1,K}.\nonumber 
				\end{align}
				
				Given $\epsilon>0$, since $\{v_n-v\}$ is bounded in $H^1(K)$, for $R=R_0$
				sufficiently large  it holds
				\begin{equation} \label{eq: 4.19}
				\frac{16}{R^2_0}\|v_n-v\|^2_{1,K}\le\frac{\epsilon}{2},
				\qquad\forall n\ge 1.  
				\end{equation}
				On the other hand, since the inclusion $H^1(B_{R_0})\subset L^2(B_{R_0})$
				is compact, it holds
				\begin{equation} \label{eq: 4.20}
				\|v_n-v\|_{L^2(B_{R_0})}\le\frac{\epsilon}{2}, \text{ for all } n\ge
				n_0
				\end{equation}
				if $n_0$ is sufficiently large.
			
				Combining \eqref{eq: 4.18} with $R=R_0$, \eqref{eq: 4.19} and \eqref{eq: 4.20} we deduce  that $\|v_n-v\|_K\to 0$
				 when  $n\to\infty$.
				
				\item[(iii)]
				From the compactness of the injection $H^1(K)\subset L^2(K)$ we deduce  that 
				$$\lambda_1=\min_{\{  v\in H^1(K) \\ v\ne0 \}}
				\left\{\frac{\displaystyle\int |\nabla v|^2K}{\displaystyle\int
					|v|^2K}\right\}>0$$
				and therefore $\|\nabla v\|_K$ and $\|v\|_{1,K}$ are equivalent norms in $H^1(K)$.
				
				As the bilinear form
				$$a(v,w)=(Lv,w)_K= (v,Lw)_K=\int\nabla v\nabla wK$$
				is coercive in $H^1(K)$, by the Lax-Milgram Theorem we deduce that
				$L:H^1(K)\to \bigl(H^1(K)\bigr)^*$ is an isomorphism.
				
				\item[(iv)]
				We firstly note that
				$$D(L)=\{\,v\in L^2(K): Lv\in L^2(K)\,\}=H^2(K).$$
				As
				$$(Lv,v)_K=\int|\nabla v|^2K$$
				we deduce that $D(L)\subset H^1(K)$.
				
			Let $v\in D(L)$ be such that $Lv=f\in L^2(K)$ and define $w=K^{1/2}v$. As
				$v\in H^1(K)$ it holds $w\in H^1(\R^N)$ and also:
				$$-\Delta w+\biggl(\frac{N}{4}+\frac{|y|^2}{16}\biggr)w=K^{1/2}
				f\in L^2(\R^N).$$
				On the other hand
				\begin{align*}
					&\int\biggl|-\Delta w+\biggl(\frac{N}{4} +
						\frac{|y|^2}{16}\biggr) w\biggr|^2\\
					&=\int\biggl(|\Delta w|^2 + \frac{N^2}{16}
						w^2+\frac{|y|^4w^2}{256}+\frac{N|y|^2}{32} w^2+\frac{N}{2} |\nabla w|^2
						+\frac{|y|^2}{8}|\nabla w|^2-\frac{N}{8}|w|^2\biggr)
				\end{align*}
				from where we deduce  that $\Delta w\in L^2(\R^N)$ and therefore
				\begin{equation} \label{eq: 4.21}
					w\in H^2 (\R^N).
				\end{equation}
				Besides, $|y|^2 w$, $|y||\nabla w|\in L^2(\R^N)$, which implies
				\begin{equation} \label{eq: 4.22}
					|y|^2v\in L^2(K),\qquad |y||\nabla v|\in L^2(K).
	\end{equation}
				
				Combining \eqref{eq: 4.21}--\eqref{eq: 4.22} it can be shown that $\partial_{ij}^2v\in
				L^2(K)$,$\forall i,j \in \{1,\ldots,N\}$ which ensures that $v\in
				H^2(K)$. It can easily be shown that  $L$ is selfadjoint. In the part (iii)
				we saw that $L$ is positive.
				
				On the other hand, $L^{-1}:\bigl(H^1(K)\bigr)^*\to H^1(K)$ is continuous and
				since the injections $\bigl(H^1(K)\bigr)^* \subset L^2(K)$ and
				$H^1(K)\subset L^2(K)$ are compact, we deduce that $L^{-1}:L^2(K)\to
				L^2(K)$ is compact.
				
				\item[(v)] As seen in the previous parts, $L$ has a non-decreasing sequence of positive eigenvalues$\{\lambda_k\}_{k\ge 1}$ that goes to infinity  when  $k\to\infty$.
				
				Let $\lambda>0$ and $v\in L^2(K)$ be such that
				\begin{equation} \label{eq: 4.23}
					Lv=\lambda v.
				\end{equation}
				As $|y|^m v\in L^2(\R^N)$ for all $m\in \N$, we see that $\hat v$ -- the Fourier transform of $v$ -- satisfies $\hat v\in H^m(\R^N)$ for all
				$m\in \N$ and therefore $\hat v\in C^\infty(\R^N)$.
				
	Applying the Fourier transform in the identity \eqref{eq: 4.23} and applying the formula ${\widehat{(y\nabla v)}(z) = -\hat v(z) 
					+\frac12 z\cdot\nabla\hat v(z)}$ we obtain
				$$|z|^2 \hat v(z)+\frac{N}{2} \hat v(z) +\frac12 z\cdot\nabla\hat v(z)
				=\lambda\hat v(z)$$
				and therefore, $w=e^{|z|^2}\hat v$ satisfies
				\begin{equation} \label{eq: 4.24}
					z\cdot\nabla w(z)=(2\lambda-N) w(z). 
				\end{equation}
				It is the Euler's equation that is satisfied by the homogeneous functions of degree $2\lambda-N$. As $w\in C^\infty(\R^N)$ necessarily
				$2\lambda-N\in\N$. So we obtain the eigenvalues
				$$\lambda_k=\frac{N+k-1}{2},\qquad k=1,2,3,\ldots$$
				
				The functions $w_k$ solution to \eqref{eq: 4.24} for $\lambda=\lambda_k$ are 
				$$w_k(z)=P_{k-1}(z)$$
				where $P_{k-1}$ is a homogeneous polynomial of degree $k-1$ arbitrary,
				i.e., $P_{k-1}(z)=z^\beta=z^{\beta_1}\cdots z_N^{\beta_N}$ for some
				$\beta=(\beta_1,\ldots,\beta_N)\in\N^N$ such that
				$|\beta|=\beta_1+\ldots+\beta_N=k-1$. therefore
				$$\hat v_k(z)=e^{-|z|^2}P_{k-1}(z)$$
				and applying the inverse Fourier transform we obtain
				$$v_k(y)=P_{k-1}(D)[e^{-|y|^2/4}]=D^\beta[e^{-|y|^2/4}]$$
				with $|\beta|=k-1$.
				
				\item[(vi)]
				Thanks to\eqref{eq: 4.13} we have
				\begin{align*}
					\int|v(y)|^2K(y) &\le\exp\delfrac({R^2}{4}) \int_{B_R} |v(y)|^2
					+\frac{1}{R^2}\int_{B_R^c} |v(y)|^2 K(y)\\
					&\le\exp\delfrac({R^2}{4}) \int_{B_R} |v(y)|^2
					+\frac{16}{R^2}\int|\nabla v(y)|^2 K(y).
				\end{align*}
				
				Le t$R>0$ be such that ${\frac{16}{R^2}=\epsilon}$, So,
				\begin{equation*} 
				\|v\|_K^2\le\epsilon\|\nabla v\|^2_K + e^{R^2/4}\|v\|_{L^2(B_R)}^2
				\end{equation*}
				and as
				$$\|v\|_{L^2(B_R)}^2 \le C(R,q)\|v\|_{L^q(B_R)}^2$$
				for all $q\ge2$, we deduce  \eqref{eq: 4.16}  when $q\ge 2$.
				
				if $1\le q<2$, simply note that
				$$\|v\|_{L^2(B_R)}\le C\|v\|_{L^q(B_R)}^\alpha
				\|v\|_{H^1(B_R)}^{1-\alpha}$$
				for some $0<\alpha<1$. Therefore,
				\begin{equation*} 
					\|v\|_{L^2(B_R)}^2\le \epsilon\|v\|_{H^1(K)}^2 + C(\epsilon)
					\|v\|_{L^q(B_R)}^2.
				\end{equation*}
				We deduce \eqref{eq: 4.16} for all $1\le q<2$.
			\end{enumerate}
		\end{proof}
		
		\section{The heat equation in similarity variables}
		
		The linear heat equation
		\begin{equation} \label{eq: 4.25}
			\begin{cases}
			u_t-\Delta u= 0 &\hbox{en $\R^N\times(0,\infty)$}\\
			u(x,0)= u_0(x)&\hbox{en $\R^N$}
			\end{cases}
		\end{equation}
		with the change of variables \eqref{eq: 4.5} transformes into
		\begin{equation} \label{eq: 4.26}
			\begin{dcases}
			v_s-\Delta v-\frac{y\cdot\nabla v}{2}-\frac{N}{2}v=0  &\hbox{en
				$\R^N\times(0,\infty)$}\\
			v(y,0)=u_0(y)
			\end{dcases}
		\end{equation}
		The parabolic operator of this system		$$\frac{\partial}{\partial s}+L-\frac{N}{2} I$$
		coincides with the linear part of the nonlinear equation \eqref{eq: 4.6}.
		
		It is thence interesting to study the asymptotic behavior of the solutions to \eqref{eq: 4.26}.
		
		Let $\{\varphi_l\}_{l\ge 1}$ be an orthonormal basis of $L^2(K)$ consisting of eigenfunctions of $L$. Every eigenfunction $\varphi_l$
		corresponds to the eigenvalue $\mu_l$ of the operator $L$, such that
		$\{\mu_l\}_{l\ge 1}$ is a non-decreasing sequence of eigenvalues in which each eigenvalue ${\frac{N+k-1}{2}}$ appears as many times as its multiplicity indicates. In particular
		\begin{equation} \label{eq: 4.27}
			\mu_1=\frac{N}{2},\qquad \mu_2=\frac{N+1}{2} 
		\end{equation}
		and the first eigenfunction $\varphi_1$ is
		\begin{equation} \label{eq: 4.28}
			\varphi_1(y)=(4\pi)^{-N/4}\exp(-|y|^2/4).
		\end{equation}
		Let $u_0\in L^2(K)$. Let $\{\alpha_l\}_{l\ge 1}$ be its Fourier coefficients	
		\begin{equation} \label{eq: 4.29}
			\alpha_l=(u_0,\varphi_l)_K=\int u_0(y)\varphi_l(y)K(y)\,dy
		\end{equation}
		so that 
		\begin{equation} \label{eq: 4.30}
			u_0=\sum_{l\ge1} \alpha_l \varphi_l.
		\end{equation}
		
		We observe that
		\begin{equation} \label{eq: 4.31}
		\alpha_1=(u_0,\varphi_1)_K=(4\pi)^{-N/4} \int u_0(y)\,dy.
		\end{equation}
		
		The solution $v=v(y,s)$ to \eqref{eq: 4.26} is given by the formula
		\begin{align} 
			v(y,s)&=\sum_{l\ge1} e^{-(\mu_l-N/2)s}\alpha_l\varphi_l(y)
			=\alpha_1\varphi_1(y) +\tilde v(y,s) \label{eq: 4.32}\\
			\text{ with } \nonumber \\
			\tilde v(y,s)&=\sum_{l\ge2} e^{-(\mu_l-N/2)s}\alpha_l\varphi_l(y)  \label{eq: 4.33}
		\end{align}
		Taking into account that if $\psi=\sum_{l\ge1}\beta_l\varphi_l$,
		\begin{equation} \label{eq: 4.34}
			\|\psi\|_{m,K}^2=\sum_{l\ge1}(\mu_l)^m \beta_l^2
		\end{equation}
		from \eqref{eq: 4.32} we deduce that
		\begin{equation} \label{eq: 4.35}
		v\in C\bigl([0,\infty); L^2(K)\bigr)\cap C^\infty \bigl((0,\infty);
		H^m(K)\bigr) 
		\end{equation}
		for all $m\in\N$. (This is the usual regularizing effect in semigroups generated by selfadjoint operators.)
		
		In the decomposition  \eqref{eq: 4.32} also observe that the projection of the solution $v(s)$ on the first eigenspace $E_1=\Span\{\varphi\}$ does not depend on $s$, i.e.,
		\begin{equation} \label{eq: 4.36}
			\int v(y,s)\,dy=\int u_0(y)\,dy,\qquad\forall s\ge 0.
		\end{equation}
		Hence, the mass of the solution is conserved.
		
		On the other hand
		$$\mu_l-\frac{N}{2} \ge \frac12, \qquad\forall l\ge 2$$
		and therefore for the component $\tilde v(y,s)$ we obtain an exponential decay:
		$$\|\tilde v(s)\|_K\le e^{-s/2} \|\tilde u_0\|_K,\qquad \forall s\ge 0$$
	 $\tilde u_0$ being the projection of $u_0$ on $E_1^\perp$:
		$$\tilde u_0(y)=\sum_{l\ge2} \alpha_l\varphi_l(y).$$
		
		Therefore,
		\begin{equation} \label{eq: 4.37}
			\|v(s)-\alpha_1\varphi_1\|_K \le e^{-s/2}\|\tilde u_0\|_K, \qquad
			\forall s\ge1
		\end{equation}
		for all $m\in\N$.
		
		It is also easily checked that
		\begin{equation} \label{eq: 4.38}
			\| v(s)-\alpha_1 \varphi_1\|_{m, K} \leq C_m e^{-s/2} \| \tilde u_0\|_K, \qquad s \geq 1
		\end{equation}
		for all $m \in \N$.
		
		The operator \eqn{L-\frac{N}{2}I} generates an analytic semigroup of contractions in $L^2(K)$ that we henceforth denote by $\{S(s)\}_{s\ge0}$
		which is such that
		\begin{equation} \label{eq: 4.39}
		S(s) u_0=v(s),\qquad\forall u_0\in L^2(K),\forall s>0
		\end{equation}
		 $v=v(s)$ being the solution to \eqref{eq: 4.26}.
		
		The following result then holds.
		
		\begin{thm}
			\begin{enumerate}
				\item[(i)]
				The linear parabolic equation \eqref{eq: 4.26} admits a unique solution
				$v(s)=S(s)u_0\in C\bigl([0,\infty); L^2(K)\bigr)$ for every
				initial data $u_0\in L^2(K)$. This solution satisfies
				$$v\in C^\infty\bigl((0,\infty);H^m(K)\bigr),\qquad\forall m\in\N.$$
				
				\item[(ii)]
				The projection of the solution $v$ on the first eigenspace $E_1$ remains invariant				\begin{equation} \label{eq: 4.40}
					\bigl(v(s),\varphi_1\bigr)_K=(u_0,\varphi_1)_K=\alpha_1,\qquad\forall
					s\ge 0.
				\end{equation}
				
				\item[(iii)]
				It holds
				$$\|v(s)\|_K\le Ce^{-s/2}\|u_0\|_K,\qquad\forall s\ge0,\forall u_0\in
				E_1^\perp$$
				and for all $m\in\N$ there exists $C_m>0$ independent of initial data such that 
				\begin{equation} \label{eq: 4.41}
					\|v(s)\|_{m,K}\le C_m e^{-s/2}\|u_0\|_K,\qquad\forall s\ge1,\forall
					u_0\in E_1^\perp. 
				\end{equation}
				
				\item[(iv)]
				On the other hand
				\begin{equation} \label{eq: 4.42}
					\biggl\|\biggl(L-\frac{N}{2}\biggr) v(s)\biggr\|_K \le \frac{1}{\sqrt2s}
					\|u_0\|_K,\qquad\forall s>0, \forall u_0\in E_1^\perp
				\end{equation}
				y
				\begin{equation} \label{eq: 4.43}
					\|\nabla v(s)\|_K^2-\frac{N}{2}\|v(s)\|_K^2 \le
					\frac{1}{2s}\|u_0\|_K^2,\qquad\forall s>0,
					\forall u_0\in E_1^\perp
				\end{equation}
			\end{enumerate}
		\end{thm}
		
		\begin{proof}
			The propositions (i)--(iii) have previosly been proven. Since
			${L-\frac{N}{2}I}$ is a maximal monotone and selfadjoint operator on 
			$L^2(K)$ with domain $H^2(K)$, \eqref{eq: 4.42} and \eqref{eq: 4.43} are a consequence of the regularizing effect of the equation (cf.\ T.~Cazenave and A.~Haraux \cite{CaHa}, Th.
			3.2.1.)
		\end{proof}
		
		As $H^m(K)\subset W^{m,p}(\R^N)$ with continuous injection for all
		$p\in[1,2]$, from \eqref{eq: 4.38} we deduce  that
		\begin{equation} \label{eq: 4.44}
		\|v(s)-\alpha_1\varphi_1\|_{W^{m,p}(\R^N)} \le C_m e^{-s/2}\|u_0\|_K,
		\qquad\forall s\ge1.
		\end{equation}
		
		On the other hand, for all $m\in\N$ and $p\in[2,\infty]$ there exists $m'\in\N$
		such that $H^{m'}(\R^N)\subset W^{m,p}(\R^N)$ with continuous injection and thus, we deduce that \eqref{eq: 4.44} es cierto for all $m\in\N$ y
		$p\in[1,\infty]$.
		
		In terms of the solution $u=u(x,t)$ of the heat equation \eqref{eq: 4.27}, the estimate \eqref{eq: 4.44} implies \begin{equation} \label{eq: 4.45}
		\|D^\alpha u(x,t)-MD^\alpha G(x,t)\|_p \le C_{p,m}
		t^{-\np-(|\alpha|+1)/2},\qquad\forall t\ge1
		\end{equation}
		for all $\alpha\in\N$ and $p\in[1,\infty]$, $M$  being the mass of the
		solution: ${M=\int u_0(x)\,dx}$. In particular we obtain
		\begin{equation} \label{eq: 4.46}
		\|u(t)-MG(t)\|_p \le C_p t^{-\np-(|\alpha|+1)/2},
		\qquad\forall t\ge1 
		\end{equation}
		for all $p\in[1,\infty]$.
		
		We obtain therefore a result analogous to the proof of Lemma 1.2 of Chapter 1, but this time only for initial data $u_0\in L^2(K)$. Nevertheless, since $L^2(K)$ is dense in $L^2(\R^N)$, the density argument
		used in the proof of Theorem 1.1 of Chapter I now allows us to conclude that
		\begin{equation*}
			t^\np\|u(t)-MG(t)\|_p\to0  \quad \text{ when } \quad t\to\infty
		\end{equation*}
		for all $p\in[1,\infty]$ and $u_0\in L^1(\R^N)$ such that ${\int
			u_0(x)\,dx=M}$.
		
		Thus, by means of the self-similar variables, we obtain optimal results about the asymptotic behavior of the solutions to the heat equation.

		\chapter{The Cauchy problem}
		
		\section{The Cauchy problem in $L^1(\R^N)\cap L^\infty(\R^N)$}
		
		In this part, we study the equation
		\begin{equation} \label{eq: 5.1}
			\begin{cases}
			u_t-\Delta u=a\cdot\nabla\bigl(F(u)\bigr) & \text{ in } \R^N\times(0,\infty),\\
			u(0)=u_0
			\end{cases}
		\end{equation}
	with initial data $u_0\in L^1(\R^{N})\cap L^\infty(\R^N)$.
		
		We have the following results of existencia, uniqueness and regularity of
		solutions.
		\begin{thm}
			Suppose that $F\in C^1(\R)$ and $F(0)=0$. Then, for all initial data $u_0\in L^1(\R^N)\cap L^\infty (\R^N)$ there exists a unique
			solution $u\in C([0,\infty);L^1(\R^N))\cap
			L^\infty\bigl(\R^N\times(0,\infty )\bigr)$. This solution
			satisfies $u\in C\bigl((0,\infty);W^{2,p}(\R^N)\bigr)\cap
			C^1\bigr((0,\infty );L^p(\R^N)\bigr)$ for all $p\in(1,\infty)$.
			
			Besides, if $u$ and $v$ are solutions to \eqref{eq: 5.1} corresponding to initial data $u_0$, $v_0 \in L^1(\R^N)\cap
			L^\infty(\R^N)$, it holds
			$$\|u(t)-v(t)\|_1\le
			\|u_0-v_0\|_1,\qquad\forall t\ge 0.$$
			
			In other words, \eqref{eq: 5.1} a contraction semigroup in $L^1(\R^N)$.
		\end{thm}
		
		\begin{proof}
			
			\begin{enumerate}
				\item[(1)] {\em Existence}.
				
				We consider the integral equation associated to \eqref{eq: 5.1}:
				
				\begin{align} \label{eq: 5.2}
					u(t)&=G(t)*u_0+\int^t_0 G(t-s)*a\cdot\nabla
					\Bigl(F\bigl(u(s)\bigr)\Bigr)\,ds \\
					&=G(t)*u_0+\int^t_0 a\cdot\nabla G(t-s)*F\bigl(u(s)\bigr)\,ds \nonumber
				\end{align}
				with $G$ being the heat kernel
				$$G(x,t)=(4\pi t)^{-N/2}\exp\biggl(-\frac{|x|^2}{4t}\biggr)$$
			and we henceforth denote by the s\'imbolo $*$ the convolution in the spatial variables.
				
				The solutions to \eqref{eq: 5.2} are weak solutions of \eqref{eq: 5.1}. We seek  the  local solution to \eqref{eq: 5.2} (for $0\le t\le T$, with $T>0$ sufficiently small) as a fixed point of the operator $$[\phi(u)](t)=G(t)*u_0
				+\int^t_0a\cdot\nabla G(t-s)*F\bigl(u(s)\bigr)\,ds$$
				in the Banach space
				$$X_T=C\bigl([0,T];L^1(\R^N)\bigr)\cap
				L^\infty\bigl(\R^N\times(0,T)\bigr)$$
				with $T>0$ sufficiently small that we will choose later. We will apply the Banach fixed point theorem. It will be enough to prove that $\phi$ is a strict contraction in $B_R$, the ball of radius $R$ in $X_T$:
				$$B_R=\{\,u\in X_T:\sup_{t\in
					[0,T]}\{\|u(t)\|_1+\|u(t)\|_\infty\}<R\,\}$$
				with a radius $R$ that we will choose later.				
				Let $M=\max\{\|u_0\|_1,\|u_0\|_\infty\}$. Let
				$r=1,\infty$ and let us estimate the $L^\infty\bigl(0,T;L^r(\R^N)\bigr)$ norm of $\phi(u)$ for $u\in X_T$.
				Thanks to the Young inequality it holds
				$$\|[\phi(u)](t)\|_r\le \|G(t)\|_1 \|u_0\|_r+\int^t_0 \|a\cdot\nabla
				G(t-s)\|_1 \bigl\|F\bigl(u(s)\bigr)\bigr\|_r\,ds.$$
				
				By means of an explicit computation it can be seen that				\begin{align*}
						\|G(t)\|_1&= 1,\forall t>0\\
						\|\nabla G(t)\|_1&\le Ct^{-1/2},\forall t>0.
				\end{align*}
				therefore
				$$\|[\phi(u)](t)\|_r\le M+C|a|\int^T_0(t-s)^{-1/2}
				\bigl\|F\bigl(u(s)\bigr)\bigr\|_r\,ds.$$
				On the other hand
				$$\|F\bigl(u(t)\bigr)\|_r\le m(R)\|u(t)\|_r,\qquad\forall t\in [0,T]$$
				con
				$$m(R)=\max_{z\in [0,R]}\biggl|\frac{F(z)}{z}\biggr|$$
				y
				$$\|u(t)\|_r<R,\qquad\forall t\in [0,T]$$
				for $r=1,\infty$, for all $u\in B_R$.
				
				Thus
				\begin{align*}
					\|[\phi(u)](t)\|_r&\le M+C|a|m(R)R\int^t_0(t-s)^{-1/2}\,ds\\
					&=M+2C|a|m(R)Rt^{1/2}\\
					&\le M+2C|a|m(R)RT^{1/2},\qquad\forall t\in [0,T]
				\end{align*}
				for $r=1,\infty$.
				
				If $R$, $T>0$ are such that				\begin{equation} \label{eq: 5.3}
					M+2C|a|m(R)RT^{1/2}<R/2
				\end{equation}
				we will have guaranteed that $\phi(B_R)\subset B_R$.
				
				For \eqref{eq: 5.3} to be fulfilled, we choose $R$ such that
				\begin{equation} \label{eq: 5.4}
					R\ge 4M
				\end{equation}
				and then $T>0$ sufficiently small as to have
				\begin{equation} \label{eq: 5.5}
					2C|a|m(R)RT^{1/2}<R/4.
				\end{equation}
				We now estimate $\phi(u)-\phi(v)$ in $L^\infty\bigl(0,T;L^r(\R^N)\bigr)$
				for all $u,v\in B_R$. It holds
				\begin{align*}
					\|[\phi(u)](t)-[\phi(v)](t)\|_r &\le \int^t_0 \|a\cdot\nabla
					G(t-s)\|_1 \bigl\|F\bigl(u(s)\bigr)-F\bigl(v(s)\bigr)\bigr\|_r\,ds\\
					&\le C|a|\int^t_0(t-s)^{-1/2}
					\bigl\|F\bigl(u(s)\bigr)-F\bigl(v(s)\bigl)\bigr\|_r\,ds
				\end{align*}
				
				Set
				$$L(R)= \max_{z\in [0,R]} |F'(z)|.$$
				
				Then
				\begin{align*}
					\|[\phi(u)](t)-[\phi(v)](t)\|_r&\le C|a|L(R)\int^t_0(t-s)^{-1/2}
					\|u(s)-v(s)\|_r\,ds\\
					&\le 2C|a|L(R)T^{1/2}\|u-v\|_{L^\infty(0,T;L^r(\R^N))},\\
					{\forall t\in [0,T],\forall u,v\in B_R.}
				\end{align*}
				
				To guarantee that
				$\phi$ is a strict contraction in $B_R$ it will be enough to choose
				$T$ sufficiently small as to, in addition to \eqref{eq: 5.5}, 
				$$2C|a|L(R)T^{1/2}<1$$
				holds as well.				
				In this way we deduce that $\phi$ admits a unique fixed point in $B_R$,
				which ensures the existence of a local solution of the integral equation \eqref{eq: 5.2} for $t\in[0,T]$. This solution satisfies \eqref{eq: 5.1} for $t\in
				[0,T]$.
				
				By means of classical extension methods (see for instance
				\cite{CaHa}), we deduce  the existence of a maximal solution
				$u\in C\bigl([0,T_m);L^1(\R^N)\bigr)
				\cap L^\infty\bigl(\R^N\times(0,T_m)\bigr)$
				where $T_m>0$ is the maximal time of existence that satisfies the
				following alternative:
				$$T_m=\infty$$
				or
				$$\hbox{Si }T_m<\infty,\quad \limsup_{t\to T^-_m}
				\{\|u(t)\|_1+\|u(t)\|_\infty\}=\infty. $$
				In order to ensure that the solution is global, that is to say that
				$u\in C\bigl([0,\infty); L^1(\R^N)\bigr)\cap
				L^\infty\bigl(\R^N\times(0,\infty)\bigr)$, it is enough to prove a priori estimate for the solution, meaning
				$$\sup_{t\in[0,T]}\{\|u(t)\|_1+\|u(t)\|_\infty\}<\infty. $$
				
				As mentioned in what precedes, the solution to \eqref{eq: 5.2} is a solution of
				the differential equation  \eqref{eq: 5.1} in $\R^N\times[0,T_m)$.
				Classical results on the
				regularity for the linear heat equation and a	{\em boot-strap} argument allows for showing that the constructed solution satisfies the following regularity:
				\begin{equation} \label{eq: 5.6}
					u\in C\bigl((0,T_m);W^{2,p}(\R^N)\bigr)\cap
					C^1\bigl((0,T_m);L^p(\R^N)\bigr).
				\end{equation}
				
				Multiplying formally the equation \eqref{eq: 5.1} by
				\begin{align*}
				\sgn(u) = \begin{cases}
				1&u>0  \\  0&u=0  \\  -1&u<0
				\end{cases}
				\end{align*}
				and integrating in $\R^N$ we obtain
				\begin{align*}
					{\frac{d}{dt} \int |u(x,t)|\,dx
						-\int\Delta u(x,t)\sgn\bigl(u(x,t)\bigr)\,dx}\\
						=\int a\cdot\nabla
						\bigl(F(u(x,t)\bigr)\sgn\bigl(u(x,t)\bigr)\,dx\\
						=\int a\cdot\nabla
						\Bigl(H\bigl(u(x,t)\bigr)\Bigr)\,dx				
				\end{align*}
				with
				$$H(z)=\int^z_0 F'(s)\sgn(s)\,ds.$$
				Integrating by parts we deduce  that
				$$\int a\cdot\nabla \bigl(H(u)\bigr)\,dx=0$$
				and the Kato inequality ensures that
				$$-\int \Delta u\sgn(u)\,dx\ge 0$$
				so
				$$\frac{d}{dt}\int |u(x,t)|\,dx\le 0$$
				and therefore
				\begin{equation} \label{eq: 5.7}
					\|u(t)\|_1\le \|u_0\|_1,\qquad\forall t\in [0,T_m).  
				\end{equation}
				To justify the formal calculation that we have just carried out, we define Lipschitz regularization of $\sgn(u)$:
				\begin{equation*}
					\theta_\epsilon(u)=
					\begin{dcases}
					\frac{u}{\epsilon}&|u|<\epsilon\\
					\sgn(u)&\epsilon\le|u|\le\frac{1}{\epsilon}\\
					u-\biggl(\frac{1}{\epsilon}-1\biggr)&|u|\ge\frac{1}{\epsilon}.
					\end{dcases}
				\end{equation*}
				As $u\in L^1(\R^N)\cap L^\infty(\R^N)$ it can be seen that
				$\theta_\epsilon(u) \in L^1(\R^N)\cap L^\infty (\R^N)$ for all
				$\epsilon >0$, and 
				$$\theta_\epsilon(u)\to\sgn(u) \quad \text{ a.e. } \quad x\in \R^N, \forall
				t\in [0,T_m). $$
				
				Multiplying the equation \eqref{eq: 5.1} by $\theta_\epsilon(u)$ and integrating
				in $\R^N$ we now obtain in a rigorous way that
				\begin{align*}
					\int a\cdot\nabla \bigl(F(u)\bigr)\theta_\epsilon(u)\,dx&=0\\
					\int\Delta u\theta_\epsilon(u)\,dx&\ge0
				\end{align*}
				for all $\epsilon >0$ and integrating in $[0,t]$ with respect to time, we obtain
				\begin{align*}
					\int^t_0\!\!\!\int u_t(x,s)\theta_\epsilon
					\bigl(u(x,s)\bigr)\,dx\,ds&=\int^t_0\biggl(\frac{d}{dt}\int\gamma
					_\epsilon\bigl(u(x,s)\bigr)\,dx\biggr)\,ds\\
					&=\int\gamma_\epsilon \bigl(u(x,t)\bigr)\,dx
					-\int\gamma_\epsilon\bigl(u_0(x)\bigr)\,dx \le 0
				\end{align*}
				where
				$$\gamma_\epsilon(z)=\int^z_0\theta_\epsilon(s)\,ds.$$
				
				Thus
				$$\int\gamma_\epsilon\bigl(u(x,t)\bigr)\,dx\le \int\gamma_\epsilon
				\bigl(u_0(x)\bigr)\,dx. $$
				Passing to the limit in this inequality  when  $\epsilon \to 0$, by the Lebesgue dominated convergence Theorem we obtain \eqref{eq: 5.7}.
				
	To estimate the $L^\infty$ norm of the solution we define 
				$\mu=\|u_0\|_\infty$. We multiply the equation \eqref{eq: 5.1} by
				$\sgn[(u-\mu)^+]$ where
				\begin{equation*}
					z^+=
					\begin{cases}
					z&z\ge0  \\  0&z\le0
					\end{cases}
				\end{equation*}
				Integrating in $\R^N$ and by arguments analogous to those used in what precedes, we deduce that 
				$$\frac{d}{dt}\int\bigl(u(x,t)-\mu\bigr)^+\,dx\le 0.$$
				As $(u_0-\mu)^+=0$ we obtain that
				$$u(x,t)\le \mu, p.c.t. x\in \R^N, \forall t\in [0,T_m). $$
				
				Multiplying the equation by $\sgn[(u+\mu)^-]$ with $z^-=-(-z)^+$,
				we obtain
				$$u(x,t)\ge -\mu , p.c.t. x\in \R^N, \forall t\in [0,T_m). $$
				we obtain
				\begin{equation} \label{eq: 5.8}
					\|u(t)\|_\infty\le\mu,\qquad\forall t\in [0,T_m).
				\end{equation}
			From \eqref{eq: 5.7} and \eqref{eq: 5.8} we deduce  that $T_m=\infty$, which concludes the proof of the existence of a global solution. From \eqref{eq: 5.6} we obtain the regularity of the solution stated in the Theorem.
				
				\item[(2)] {\em Uniqueness}.
				
				The uniqueness is an immediate consequence of the following $L^1$ contraction property enjoyed by the semigroup generated by the system \eqref{eq: 5.1}:
				
				\begin{equation} \label{eq: 5.9}
					\|u(t)-v(t)\|_1 \le \|u_0-v_0\|_1,\qquad\forall t>0
				\end{equation}
				for all solutions to \eqref{eq: 5.1} with initial data $u_0,v_0\in
				L^1(\R^N)\cap L^\infty (\R^N)$.
				
				To prove this property, simply multiply by $\sgn(u-v)$ the equation satisfied by $u-v$ and integrate in $\R^N$. we obtain
				\begin{equation} \label{eq: 5.10}
				\frac{d}{dt}\|(u-v)(t)\|_1\le 0.
				\end{equation}
				Again, to justify this calculation we multiply the equation by
a regularization $\theta_\epsilon(u-v)$ and we pass to the limit as
				$\epsilon\to0$. 
			\end{enumerate}
		\end{proof}
		
		\section{Decay in $L^p(\R^N)$}
		
		In this part we will obtain $L^p$ estimates for the solutions to \eqref{eq: 5.1}. We will show decay results s analogous to those obtained in Chapter I for the linear heat equation.		
		The following result holds.
		
		\begin{thm}
			For all $p\in [1,\infty]$ there exists a positive constant $C=C(p,N)>0$ such that
			\begin{equation} \label{eq: 5.11}
				\|u(t)\|_p\le C\|u_0\|_1 t^{-\np},\qquad \forall t>0
			\end{equation}
			for all solution to $\eqref{eq: 5.1}$ with initial data $u_0\in L^1(\R^N)\cap
			L^\infty (\R^N)$.
		\end{thm}
		
		\begin{obs}
			\begin{enumerate}
				\item[(1)]
				The estimate \eqref{eq: 5.11} analogous to the one we obtain for solutions to the linear heat equation. This estimate implies that
				$t^{\np}u(t)\in L^\infty \bigl(0,\infty ;L^p(\R^N)\bigr)$.Therefore, its behavior in $L^p(\R^N)$ should be considered when $t\to+\infty$.
				
				\item[(2)]
				The estimates will be demonstrated by means of the multiplier method.
				We multiply the equation by $|u|^{r-1}u$ for any $r>1$.
				Integrating on $\R^N$, we see that the convective term vanishes. That is why we obtain the same estimates that for the heat equation linear. An iterative argument allows obtaining the results for
				$p=\infty$. In this case \eqref{eq: 5.11} means
				$$\|u(t)\|_\infty\le C\|u_0\|_1t^{-N/2},\qquad \forall t>0. $$
				
				\item[(3)]
	These estimates are very useful when extending the results of existence and uniqueness of solutions to \eqref{eq: 5.1} to initial data $u_0\in L^1(\R^N)$.
			\end{enumerate}
		\end{obs} 
		
		\begin{proof}[Proof of Theorem 4.2]
			The case $p=1$ is an immediate consequence of \eqref{eq: 5.7}, which provides the estimate
			\begin{equation} \label{eq: 5.12}
				\|u(t)\|_1\le \|u_0\|_1,\qquad\forall t\ge 0.
			\end{equation}
			
			In the case $p\in (1,\infty)$, we multiply the equation \eqref{eq: 5.1} by
			$|u|^{p-2}u$ and integrate in $\R^N$. We obtain
			\begin{align*}
				{\frac{1}{p} \frac{d}{dt}\int |u(x,t)|^p\,dx+(p-1)\int|\nabla
					u(x,t)|^2|u(x,t)|^{p-2}\,dx}\\
					=\int a\cdot\nabla\Bigl(F\bigl(u(x,t)\bigr)\Bigr)
					|u(x,t)|^{p-2}u(x,t)\,dx\\
					=\int a\cdot\nabla \Bigl(H_p\bigl(u(x,t)\bigr)\Bigr)\,dx
			\end{align*}
			with
			$$H_p(z)=\int^z_0 F'(s)|s|^{p-2}s\,ds$$
			By the arguments of the proof of Theorem 4.1 we obtain
			$$\int a\cdot\nabla \Bigl(H_p\bigl(u(x,t)\bigr)\bigr)\,dx=0$$
			and therefore
			\begin{equation} \label{eq: 5.13}
				\frac{1}{p} \frac{d}{dt} \int |u(x,t)|^p\,dx+ \frac{4(p-1)}{p^2}\int
				\bigl|\nabla (|u(x,t)|^{p/2})\bigr|^2\,dx=0
			\end{equation}
			We need the following interpolation lemma.
			
			\begin{lem}
				For all $p\in(1,\infty)$ there exists a constant
				$C=C(p,N)>0$ such that
				\begin{equation} \label{eq: 5.14}
					\|v\|_p^{\frac{(N(p-1)+2)p}{N(p-1)}} \le
					C\|v\|^{\frac{2p}{N(p-1)}}_1\|\nabla (|v|^{p/2})\|^2_2  
				\end{equation}
				For all
				$v\in W^{2,p}(\R^N)\cap L^1(\R^N)$.
			\end{lem}
			
			\begin{obs}
				Note that since $v\in W^{2,p}(\R^N)$, entonces $\Delta v\in L^p(\R^N)$.
				therefore the integral
				$$\int \Delta v|v|^{p-2}v\,dx$$
				is well defined. By Green's formula it is observed that it coincides with				${\int |\nabla (|v|^{p/2})|^2\,dx}$, from where we deduce  that $\nabla
				(|v|^{p/2})\in L^2(\R^N)$. All the terms in \eqref{eq: 5.14} are thence well defined.
			\end{obs} 
			
			\begin{proof}[Proof of Lemma 4.4]
				We will just describe the main steps of the proof. For more details, see \cite{EZ2}.				
				Let us consider solely the case $N\ge 3.$ By the Sobolev inequality
				(cf.\ H. Brezis \cite{Br}, Th.\ IX.9) we know that there exists a constant
				$C_N>0$ (which depends only on the dimension $N$) such that
				$$\|v\|_{2N/(N-2)}\le C_N\|\nabla v\|_2,\forall v\in H^1(\R^N).$$
				therefore
				$$\|v\|^{p/2}_{Np/(N-2)}=\||v|^{p/2}\|_{2N/(N-2)}\le
				C_N\|\nabla (|v|^{p/2})\|_2.$$
				So, to obtain \eqref{eq: 5.14}, it will suffice to prove
				$$\|v\|^{\frac{(N(p-1+2)p}{N(p-1)}}_p\le
				C\|v\|^{\frac{2p}{N(p-1)}}_1\|v\|^p_{Np/(N-2)}$$
				This is an immediate consequence of the classical interpolation inequalities in $L^p$ (see, for instance, H.~Brezis
				\cite{Br}, Ch.\ IV).
			\end{proof} 
			
			\paragraph*{End of proof of Theorem 4.2:}
			Combining \eqref{eq: 5.13} and \eqref{eq: 5.14} we obtain
			\begin{equation} \label{eq: 5.15}
				\frac{d}{dt}(\|u(t)\|^p_p)+\frac{4(p-1)}{Cp}
				\frac{\|u(t)\|_p^{\frac{(N(p-1)+2)p}{N(p-1)}}}%
				{\|u(t)\|^{2p/N(p-1)}_1} \le 0.
			\end{equation}
			Combining \eqref{eq: 5.12} and \eqref{eq: 5.15} we obtain
			\begin{equation} \label{eq: 5.16}
			\frac{d}{dt}(\|u(t)\|^p_p)+\frac{4(p-1)}{Cp}\|u_0\|^{-2p/N(p-1)}_1
			\|u(t)\|^{(N(p-1)+p/N(p-1)}_p \le 0.
			\end{equation}
			
			It can easily be shown that the solutions $\varphi\ge0$ of the differential inequality
			\begin{equation} \label{eq: 5.17}
				\varphi'(t)+\alpha\bigl(\varphi(t)\bigr)^\beta\le 0,\qquad\forall t>0
			\end{equation}
			with $\alpha>0$, $\beta>1$ satisfy
			$$\varphi(t)\le\bigl(\alpha (\beta -1)t
			+\varphi^{1-\beta}(0)\bigr)^{-\frac{1}{\beta-1}},\qquad\forall t>0 $$
			and therefore, in particular,
			\begin{equation} \label{eq: 5.18}
				\varphi(t)\le \bigl(\alpha (\beta-1)t\bigr)^{\frac{1}{\beta-1}},
				\qquad\forall t>0.
			\end{equation}
			
			From \eqref{eq: 5.16}--\eqref{eq: 5.17} we deduce 
			$$\|u(t)\|^p_p\le
			\biggl[\frac{8\|u_0\|^{-2p/N(p-1)}_1}{NCp}t\biggr]^{-\frac{N}{2}
				\scriptscriptstyle (p-1)}$$
			and therefore
			$$\|u(t)\|_p\le\|u_0\|_1\delfrac({8}{NCp})^{-\np}t^{-\np}.$$
			This concludes the proof of \eqref{eq: 5.11} for $p\in [1,\infty )$.
			
			Note that the constant
			$$C_p=\delfrac({8}{NCp})^{-\np}$$
			from the estimate \eqref{eq: 5.11} does not remain bounded when  $p\to\infty$ . It is for this reason that in the case $p=\infty$ the estimate \eqref{eq: 5.11} is not obtained by passing to the limit when  $p\to\infty$. It is necessary to obtain $L^q-L^p$ estimates and use an iterative argument.
			
			The interpolation inequality \eqref{eq: 5.14} for $p=2$ imples
			\begin{equation} \label{eq: 5.19}
			\|v\|^{2(N+2)/N}_2\le C\|v\|^{4/N}_1\|\nabla v\|^2_2
			\end{equation}
			for all $v\in H^2(\R^N)\cap L^1(\R^N)$. By density, this inequality
			extends to all $v\in H^1(\R^N)\cap L^1(\R^N)$.
			
			Combining \eqref{eq: 5.13} and \eqref{eq: 5.19} with $p=2q$ for $q>1$ we obtain
			\begin{equation} \label{eq: 5.20}
				\frac{d}{dt}\bigl(\|u(t)\|^{2q}_{2q}\bigr)+
				\frac{2(2q-1)}{Cq} \frac{\|u(t)\|^{2q(N+2)/N}_{2q}}{\|u(t)\|^{4q/N}_q}
				\le 0. 
			\end{equation}
			On the other hand, multiplying the equation \eqref{eq: 5.1} por
			$|u|^{q-2}u$ we easily obtain that
			$$\frac{d}{dt}\int |u(x,t)|^q\,dx\le 0.$$
			and therefore
			\begin{equation} \label{eq: 5.21}
			\|u(t)\|_q\le \|u_0\|_q,\qquad\forall t>0
			\end{equation}
			for all $q\in[1,\infty]$.
			
			Combining \eqref{eq: 5.20}--\eqref{eq: 5.21} 	it follows that		\begin{equation} \label{eq: 5.22}
				\frac{d}{dt}\bigl(\|u(t)\|^{2q}_{2q}\bigr)+
				\frac{2(2q-1)}{Cq\|u_0\|^{4q/N}_q}\|u(t)\|^{2q(N+2)/N}_{2q} \le
				0. 
			\end{equation}
			
		Applying estimate (18) which is satisfied by solutions to (17), it follows that			$$\|u(t)\|^{2q}_{2q}\le \biggl[\frac{4(2q-1)}{NCq\|u_0\|^{4q/N}_q}
			t\biggr]^{-N/2}$$
			and thus
			$$\|u(t)\|_{2q}\le\delfrac({NCq}{4(2q-1)})^{N/4q}
			\|u_0\|_q t^{-N/4q},\qquad\forall t>0.$$
			As $q/(2q-1)\le 1$ we deduce  that
			\begin{equation} \label{eq: 5.23}
				\|u(t)\|_{2q}\le
				\delfrac({NC}{4})^{N/4q}\|u_0\|_q t^{-N/4q},\qquad\forall t>0.
			\end{equation}
			As the equation \eqref{eq: 5.1} is invariant with respect to translations in  $t$, from \eqref{eq: 5.23} we deduce that
			$$\|u(t+s)\|_{2q}\le
			\delfrac({NC}{4})^{N/4q}\|u(t)\|_q s^{-N/4q},\qquad\forall s,t>0.$$
			Given $\tau>0$, choosing $s=\tau 2^{-(n+1)}$ and $q=2^n$ we obtain
			$$\|u(t+\tau 2^{-(n+1)})\|_{2^{n+1}}\le
			\delfrac({NC}{4\tau})^{N2^{-(n+2)}}
			2^{-N(n+1) 2^{-(n+2)}}\|u(t)\|_{2^n}.$$
			
			By iterating we obtain
			$$\|u\bigl(\tau(2^{-1}+\ldots 
			+2^{-n})\bigr)\|_{2^{n+1}}\le C_n\|u_0\|_1$$
			with
			$$C_n=\delfrac({NC}{4\tau})^{N \sum\limits^{n-1}_{j=0}
				2^{-(j+2)}}(2^{-N/2})^{\sum\limits^{n-1}_{j=0}(j+1)2^{-(j+1)}}$$
		and since $2^{-1}+\ldots+2^{-n}\le 1$, por \eqref{eq: 5.21} it holds
			\begin{equation} \label{eq: 5.24}
			\|u(\tau)\|_{2^{n+1}}\le C_n\|u_0\|_1. 
			\end{equation}
			
			It can easily be shown that
			$$C_n\to 2^{-N}\delfrac({NC}{4\tau})^{N/2}  when  n\to \infty $$
			and therefore from \eqref{eq: 5.24} we deduce 
			$$\|u(\tau)\|_\infty \le \delfrac({NC}{4\tau})^{N/2}2^{-N}\|u_0\|_1.$$
			This concludes the proof of \eqref{eq: 5.11} for $p\in[1,\infty]$.
		\end{proof}
		
		\begin{obs}
			Estimates of the type \eqref{eq: 5.11} were proven by M.~Schonbeck
			\cite{S1, S2} for less general non-linearities $F$ and more regular initial data using the Fourier transform. The estimate \eqref{eq: 5.11} for
			$p\in[1,\infty]$ was shown in \cite{EZ2}. The proof of \eqref{eq: 5.11} in the 
			case $p=\infty$ is an adaptation of that done by L.~Veron \cite{Ve}
		for abstract semilinear parabolic equations.
		\end{obs}
		
		\section{The Cauchy problem in $L^1(\R^N)$}
		
		\begin{thm}
			Let $F\in C^1(\R)$ be such that $F(0)=0$. Then for all initial data $u_0\in L^1(\R^N)$ there exists a unique solution $u\in
			C\bigl([0,\infty); L^1(\R^N)\bigr)\cap L^\infty_{\rm loc}(0,\infty
			;L^\infty(\R^N))$ to $\eqref{eq: 5.1}$  such that
			$$u\in C\bigl((0,\infty );W^{2,p}(\R^N)\bigr)\cap
			C^1\bigl((0,\infty);L^p(\R^N)\bigr)$$
			for all $p\in(1,\infty)$.
			
			The solution satisfies $\eqref{eq: 5.11}$ for all $p\in[1,\infty]$.
			
			Moreover, if $u$ and $v$ are solutions to $\eqref{eq: 5.1}$ with
			initial data $u_0, v_0 \in L^1 (\R^N)$ it holds
			$$\|u(t)-v(t)\|_1\le \|u_0-v_0\|_1,\qquad \forall t\ge 0.$$
		If $u_0\ge v_0$ it holds
			$$v(x,t)\le u(x,t), \qquad\forall t>0, p.c.t. x\in\R^N. $$
		\end{thm}
		
		\begin{proof}
			Let $u_0\in L^1(\R^N)$. We construct an approximation			$\{u_{0,n}\}\subset {\cal D}(\R^N)$ such that
			$$u_{0,n}\to u_0 \text{ in } L^1(\R^N)  when  n\to\infty. $$
			By Theorem 4.1 we know that for every $n\in\N$ there exists a solution
			$u_n\in C\bigl([0,\infty);L^1(\R^N)\bigr)\cap L^\infty
			\bigl(\R^N\times(0,\infty)\bigr)$ to \eqref{eq: 5.1} with initial data $u_{0,n}$. By the  $L^1(\R^N)$  contraction
			property  \eqref{eq: 5.9} it holds
			$$\|u_n(t)-u_k(t)\|_1\le \|u_{0,n}-u_{0,k}\|_1, \qquad\forall t>0,
			\forall n,k\in\N.$$
			
			As $\{u_{0,n}\}$ is a Cauchy sequence in $L^1(\R^N)$, $\{u_n\}$
			is a Cauchy sequence in $C\bigl([0,\infty);L^1(\R^N)\bigr)$ and thus there exists $u\in C\bigl([0,\infty );L^1(\R^N)\bigr)$ such that
			\begin{equation} \label{eq: 5.26}
				u_n\to u \text{ in } L^\infty\bigl(0,\infty;L^1(\R^N)\bigr), \quad \text{ when } \quad  n\to
				\infty.
			\end{equation}
			
			Obviously $u(0)=u_0$. We still need to prove that $u$ satisfies the differential equation \eqref{eq: 5.1}. As $u_n$ satisfies it, we proceed in passing to the limit as $n\to\infty$. The linear terms do not manifest a difficulty.
			Nevertheless \eqref{eq: 5.26} is not sufficient for passing to the limit in the nonlinear term $a\cdot\nabla\bigl(F(u_n)\bigr)$ (except in the case when $F$ is globally Lipschitz, in which, $F\bigl(u_n
			(t)\bigr)\to F\bigl(u(t)\bigr)$ in $L^1(\R^N)$ and $a\cdot\nabla
			\Bigl(F\bigl(u (t)\bigr)\bigr)\to a\cdot\nabla \bigl(F(u(t)\bigr)$ en
			${\cal D}'(\R^N)$ for all $t\ge 0$). With the goal of solving this dififculty,  we observe that, thanks to  the estimate \eqref{eq: 5.11}, for all
			$t>0$ it holds
			\begin{equation} \label{eq: 5.27}
			\{u_n(t)\} \text{ is bounded in } L^\infty (\R^N).
			\end{equation}
			As $F$ is locally Lipschitz, Combining \eqref{eq: 5.26} and \eqref{eq: 5.27} we deduce
			that
			\begin{align}
			u(t)\in L^p(\R^N) &\quad\hbox{for all $p\in[1,\infty)$}\\
			u_n(t)\to u(t) &\quad\hbox{en $L^p(\R^N)$  when  $n\to\infty$, $ \forall
				p\in[1,\infty)$}\\
			F\bigl(u_n(t)\bigr)\to F\bigl(u(t)\bigr) &\quad\hbox{en $L^p(\R^N)$
				 when  $n\to\infty$, $ \forall p\in[1,\infty)$}
			\end{align}
			for all $t>0$.This allows us firstly to pass to the limit in \eqref{eq: 5.11} and extend these estimates to the limit function $u$ and secondly, pass to the limit in the equation \eqref{eq: 5.1} to the deduce that that the limit $u$ is solution \eqref{eq: 5.1}.
			
			Again, for classic regularity results for the linear heat equation and by means of  {\em boot-strap\/} bootstrap argument. We deduce that
			$$u\in C\bigl((0,\infty );W^{2,p}(\R^N)\bigr)\cap C^1\bigl((0,\infty);
			L^p(\R^N)\bigr) $$
			for all $p\in(1,\infty)$.
			
			The uniqueness of the solution is once again an immediate consequence of the $L^1$ contraction property \eqref{eq: 5.9} which can be shown as in Theorem 4.1.
			
			We finish the proof of Theorem 4.3 by using the basic comparison principle. Let $u_0$ and $v_0\in L^1(\R^N)$ be such that
			$u_0(x)\ge v_0(x)$, a.e.\ $x\in \R^N$. Let $u$ and $v$ be the corresponding solutions to
			\eqref{eq: 5.1}.
			
			Multiplying by $\sgn(u-v)^-$ the equation satisfied by $u-v$ and
			integrating in $\R^N$ we obtain
			$$\frac{d}{dt} \int_{\R^N}\bigl(u(x,t)-v(x,t)\bigr)^-\,dx\le 0.$$
			As $(u_0-v_0)^-=0$, we deduce that $(u-v)^-=0$, which implies $u\ge
			v$.
		\end{proof} 
		
		\section{Gradient estimates}
		
		In this part, we obtain estimates for the gradient of the
		solutions to \eqref{eq: 5.1} in the case where
		\begin{equation} \label{eq: 5.28}
			|F'(s)|\le C|s|^{1/N},\qquad \forall s\in \R.
		\end{equation}
		The following result holds. 
		\begin{thm}
			Let $F\in C^1(\R)$ be such that $F(0)=0$ and satisfies \eqref{eq: 5.27}. 
			So, for all $u_0\in L^1(\R^N)$ the solution $u=u(x,t)$
			to \eqref{eq: 5.1} satisfies
			\begin{equation} \label{eq: 5.29}
				\|\nabla u(t)\|_p \le C_p t^{-\np-\frac{1}{2}},\qquad\forall t>0
			\end{equation}
			for all $p\in[1,\infty]$.
		\end{thm}
		
		\begin{proof}
			Let us introduce the rescaled functions 
			\begin{equation} \label{eq: 5.30}
				u_\lambda(x,t)=\lambda^N u(\lambda x,\lambda^2 t).
			\end{equation}
			These functions satisfy
			\begin{align} \label{eq: 5.31}
			u_{\lambda,t}-\Delta u_\lambda&=\lambda^{N+1} a\cdot\nabla\nonumber 
			F(\lambda^{-N} u_\lambda)\nonumber \\
			u_\lambda(0)=u_{0,\lambda}&=\lambda^N u_0(\lambda x)
			\end{align}
			and therefore, for any $\tau>0$ satisfy the integral equation
			\begin{equation} \label{eq: 5.32}
				u_\lambda(t+\tau)= G(t)*u_\lambda (\tau)+\lambda^{N+1}
				\int^t_0\!\!\!\int^t_0 G(t-s)*a\cdot\nabla F\bigl(\lambda^{-N}
				u_\lambda(s+\tau)\bigr)\,ds.
			\end{equation}
			Taking gradients we obtain
			\begin{align} \label{eq: 5.33}
				\nabla u_\lambda(t+\tau)=\nabla
					G(t)*u_\lambda(\tau)\\
				{+\lambda \int^t_0 \nabla G(t-s)*F'\bigl(\lambda^{-N}
					u_\lambda (s+\tau)\bigr)a\cdot\nabla u_\lambda(s+\tau)\,ds.}\nonumber
			\end{align}
			
			Since
			\begin{equation} \label{eq: 5.34}
				\|\nabla G(t)\|_p \le C_p t^{-\np-\frac{1}{2}},\qquad\forall t>0
			\end{equation}
			for all $p\in[1,\infty]$ and $\|u_\lambda(t)\|_1\le
			\|u_{0,\lambda}\|_1=\|u_0\|_1$ for all $\lambda >0$, $t>0$, taking
			 $L^p$ norms in \eqref{eq: 5.33} we obtain
			\begin{align} \label{eq: 5.35}
				\|\nabla u_\lambda (t+\tau)\|_p\le
					C_p\|u_0\|_1 t^{-\np-\frac{1}{2}}  \\
				+\lambda |a|\int^t_0 \|\nabla G(t-s)\|_1\|F'\bigl(\lambda
					^{-N }u_\lambda (s+\tau)\bigr)\|_\infty\|\nabla
					u_\lambda(s+\tau)\|_p\,ds.\nonumber
			\end{align}
			From \eqref{eq: 5.11} we deduce that
			$$\|u_\lambda (t)\|_p \le C_p\|u_0\|_1 t^{-\np},\qquad\forall t\in
			(0,1)$$
			for all $p\in[1,\infty]$ with $C_p>0$ independent of  $\lambda$. Thus
			\begin{equation} \label{eq: 5.36}
				\|u_\lambda (s+\tau)\|_\infty \le C_\infty
				\|u_0\|_1\tau^{-\frac{N}{2}},\qquad\forall s\ge 0. 
			\end{equation}
			
			Combining \eqref{eq: 5.28}, \eqref{eq: 5.34}, \eqref{eq: 5.35} and \eqref{eq: 5.36} we obtain
			\begin{align} \label{eq: 5.37}
				 \|\nabla u_\lambda (t+\tau)\|_p\le
					C_p\|u_0\|_1 t^{-\np-\frac{1}{2}} \\
				 +C|a|\|u_0\|^{1/N}\tau^{-\frac{1}{2}}
					\int^1_0 (t-s)^{-1/2}\|\nabla \nonumber
					u_\lambda (s+\tau)\|_p\,ds.
			\end{align}
			
			We need the following Gronwall Lemma.			
			\begin{lem}
			Let $T>0$, $A,B>0$, $\alpha\in[0,1]$ and $\varphi\in C([0,T])$
				a non-negative function such that
				\begin{equation} \label{eq: 5.38}
					\varphi (t)\le At^{-\alpha}+B \int^t_0 (t-s)^{-1/2}\varphi
					(s)\,ds. 
				\end{equation}
				
				Then, there exists $C>0$ that depends only on $T,\alpha$ and $B$ such that
				\begin{equation} \label{eq: 5.39}
				\varphi (t)\le CA t^{-\alpha}, \qquad\forall t\in(0,T].
				\end{equation}
			\end{lem}
			
			\begin{proof}[Proof of Lemma 4.9]
				Multiplying in \eqref{eq: 5.38} by $t^\alpha$ we obtain
				\begin{equation} \label{eq: 5.40}
					t^\alpha\varphi (t)\le A+Bt\int^t_0 (t-s)^{-1/2}s^{-\alpha}\psi
					(s)\,ds 
				\end{equation}
				with $\psi(s)=\sup_{t\in[0,s]}\{t^\alpha\varphi(t)\}$. Since
				$\psi$ is non-decreasing it holds
				$$t^\alpha\varphi (t) \le A+B^\alpha \psi (t)
				\int^t_0(t-s)^{-1/2}s^{-\alpha}\,ds=A+B C_\alpha t^{1/2}\psi (t) $$
				con ${C_\alpha=\int^1_0(1-s)^{1/2}s^{-\alpha}\,ds}$. Taking supremums in
				this inequality we see that
				$$\psi(t)\le A+B C_\alpha t^{1/2}\psi (t).$$
				therefore, if $t\in[0,t_0]$ with $t_0=(2BC_\alpha)^{-2}$, it holds
				\begin{equation} \label{eq: 5.41}
					\psi (t)\le 2A, \qquad\forall t\in[0,t_0]. 
				\end{equation}
				If $T\le t_0$, \eqref{eq: 5.38} is obtained from \eqref{eq: 5.40}. Si $T>t_0$, take $\delta\le
				t_0$. Decomposing the integral in \eqref{eq: 5.39} in the segments $[0,\delta]$,
				$[\delta,t-\delta]$ and $[t-\delta,t]$ we obtain
				\begin{align}
				t^{\alpha }\varphi (t)&\le
				A+BT^\alpha\delta^{-\alpha-1/2}\int^t_0\psi(s)\,ds\\
				&\quad+Bt^\alpha\psi(t)\biggl\{\int^\delta_0(t-s)^{-1/2}s^{-\alpha}\,ds
				+\int^t_{t-\delta}(t-s)^{-1/2}s^{-\alpha}\,ds\biggr\}\\
				&=A+BT^\alpha\delta^{-\alpha-1/2}\int_0^t\psi(s)\,ds\\
				&\quad +BT^\alpha\biggl\{(t_0-\delta)^{-1/2}\frac{\delta^{1-\alpha}}%
				{(1-\alpha )} +2(t_0-\delta)^{-\alpha}\delta^{1/2}\biggr\}\psi(t).
				\end{align}
				
				Choosing $\delta>0$ sufficiently small such that
				$$BT^\alpha\biggl\{(t_0-\delta )^{-1/2}
				\frac{\delta^{1-\alpha}}{1-\alpha}
				+2(t_0-\delta)^{-\alpha}\delta^{1/2}\biggr\}\le \frac{1}{2} $$
				and taking supremums in this last inequality we obtain
				$$\psi (t)\le 2\biggl(A+BT^\alpha\delta^{-\alpha-1/2}
				\int^t_0\psi(s)\,ds\biggr). $$
				
				Applying the classic Gronwall Lemma and using (4.41) we obtain
				$$\psi(t)\le 2A+4ABT^\alpha\delta^{-\alpha-1/2}t_0\exp
				\bigl(2BT^\alpha\delta^{-\alpha-1/2}(t-t_0)\bigr),\qquad t\in[t_0,T].$$
				This concludes the proof of (4.38).
			\end{proof}
			
			\paragraph*{End of the proof of Theorem 4.4}
			Applying Lemma 4.9 in
			(4.36) with $p\ge1$ such that
			${\frac{N}{2}\biggl(1-\frac{1}{p}\biggr)+\frac12<1}$ (i.e.,
			${p\in\biggl[1,\frac{N}{N-1}\biggr)}$) we deduce that $$\|\nabla
			u_\lambda (t+\tau)\|_p\le C_p\|u_0\|_1t^{-\np-\frac12}, \qquad\forall
			t\in [0,1] $$ and therefore
			\begin{equation} \label{eq: 5.42}
				\|\nabla u_\lambda (2\tau)\|_p\le C(p,\tau)\|u_0\|_1,\qquad \forall
				\lambda >0.
			\end{equation}
			It can easily be shown that (4.41) implies (4.29). Therefore, (4.28) holds for all ${p\in\biggl[1,\frac{N}{N-1}\biggr)}$. On the other hand, from \eqref{eq: 5.28} we deduce that
			\begin{equation} \label{eq: 5.43}
				\|\nabla u_\lambda (t+\tau)\|_p\le C(p,\tau), \qquad\forall t\ge 0,
				\forall \lambda >0,\forall \tau>0,
				\forall p\in\biggl[1,\frac{N}{N-1}\biggr).
			\end{equation}
			
			To prove \eqref{eq: 5.28} for all $p\in [1,\infty]$ we use an iterative argument. Taking $L^p$ norms in \eqref{eq: 5.33} we obtain
			\begin{align} \label{eq: 5.44}
				{\|\nabla_\lambda(t+\tau)\|_p
					\le C_p t^{-\np-\frac12} \|u_0\|_1 } \\
				{+|a|\lambda\int_0^t \|\nabla G(t-s)\|_r \|F'\bigl(\lambda^{-N}
					u_\lambda(s+\tau)\bigr)\|_\infty
					\|\nabla u_\lambda (s+\tau)\|_q\,ds} \nonumber
			\end{align}
			with $r$, $q$ such that \eqn{\frac{1}{r} +\frac{1}{q} = 1+\frac{1}{p}}.
			By virtue of \eqref{eq: 5.43} we can choose $q$ in the interval
			${\biggl[1,\frac{N}{N-1}\biggr)}$. The parameter $r$ will also be chosen in ${\biggl[1,\frac{N}{N-1}\biggr)}$ for guaranteeing the integrability of $\|\nabla G(t-s)\|_r$ in the interval $[0,t]$. In this way we will obtain all the${p\in\biggl[\frac{N}{N-1},\frac{N}{N-2}\biggr)}$ . In these ranges of
			$r$, $q$ and $p$, from \eqref{eq: 5.43} we obtain
			\begin{align} \label{eq: 5.45}
				\|\nabla_\lambda(t+\tau)\|_p
				&\le C_p t^{-\np-\frac12} \|u_0\|_1\\
				&\qquad+|a| C(q,\tau)\int_0^1
				(t-s)^{-\np-\frac12}\,ds \nonumber\\
				&\le C\bigl\{t^{-\np-\frac12} + t^{-\np+\frac12}\bigr\} \nonumber.
			\end{align}
			From \eqref{eq: 5.45} we deduce 
			$$\|\nabla u_\lambda (2\tau)\|_p \le C(p,\tau),
			\qquad\forall p\in\biggl[\frac{N}{N-1},\frac{N}{N-2}\biggr)$$
			which implies \eqref{eq: 5.29} for
			${\in\biggl[\frac{N}{N-1},\frac{N}{N-2}\biggr)}$.
			
			Iterating this argument. we obtain \eqref{eq: 5.29} for all $p\in[1,\infty]$.
		\end{proof}
		
		\chapter{Weakly non-linear asymptotic behavior}
		
		\section{Introduction}
		
		In this chapter, we address the asymptotic behavior of weakly non-linear heat equations such as
		\begin{align} \label{eq: 6.1}
		u_t-\Delta u&= a\cdot\nabla\bigl(F(u)\bigr)
		&\hbox{ in $\R^N\times(0,\infty)$}\\
		u(0)&=u_0 \nonumber 
		\end{align}
		
	The problem consists in proving, under suitable conditions on the nonlinearity $F$, that
		\begin{equation} \label{eq: 6.2}
			t^{\np}\|u(t)-MG(t)\|_p\to 0 \quad \text{ when } \quad  t\to\infty
		\end{equation}
		$G$ being the fundamental solution to the linear heat equation
		\begin{equation} \label{eq: 6.3}
			G(x,t)=(4\pi t)^{-N/2} \exp \biggl(-\frac{|x|^2}{4t}\biggr)  
		\end{equation}
		and $M$ the mass of the solution: ${M=\int u_0(x)\,dx}$. In other words, one seeks to prove that  
		when  $t\to\infty$ the solutions to
		\eqref{eq: 6.1}, as a first approximation behave like the solutions to the
		linear heat equation.
		
		As we mentioned in the introduction, this phenomenon occurs if
		$$F(s)=|s|^{q-1}s$$
		with \eqn{q>1+\frac{1}{N}}. IIn the general case of nonhomogeneous nonlinearity, this phenomenon is expected to occur when  $F$ behaves like $|s|^{q-1}s$ with \eqn{q>1+\frac{1}{N}} in a neighbourhood of 
		$s=0$. In the next section we will prove that \eqref{eq: 6.2} holds true if
		$$\lim_{|s|\to0}\frac{F(s)}{|s|^{(N+1)/N}}=0.$$
		
		These results are optimal when $F(s)=|s|^{1/N} s$, as we see in the next chapter, \eqref{eq: 6.2} is not fulfilled since the solutions to \eqref {eq: 6.1} in that case behave since self-similar solutions of \eqref {eq: 6.1}.
		
		In Section 5.3 we will study the case of the initial data in $L^1(\R^N; 1+|x|) \cap L^\infty(\R^N)$. In this case we can estimate the convergence speed in \eqref{eq: 6.2}. In Chapter I we saw that this is possible for the linear equation.
		
		In \cite{Z2} we obtain the second term in the asymptotic behavior of
		the solutions to \eqref{eq: 6.1}  when  $F(s)=|s|^{q-1}s$ and \eqn{q>1+\frac{1}{N}}. 
		It can be shown that the second term depends on $q$ in a way that it's necessary to distinguish the cases \eqn{q>1+\frac{2}{N}}, \eqn{q=1+\frac{2}{N}} and
		\eqn{\frac{1}{N}<q<1+\frac{2}{N}}.
		
		\section{Initial data in $L^1(\R^N)$}
		
		The main result of this section is the following.		
		\begin{thm}
			Suppose  that the nonlinearity $F\in C^1(\R)$ satisfies
			\begin{equation} \label{eq: 6.4}
				\lim_{|s|\to0}\frac{F(s)}{|s|^{(N+1)/N}}=0.
			\end{equation}
			
			Then, for all $u_0\in L^1(\R^N)$ with \eqn{M=\int u_0(x)\,dx}, the solution to $\eqref{eq: 6.1}$ satisfies $\eqref{eq: 6.2}$ for all $p\in[1,\infty]$.
		\end{thm}
		
		\begin{proof}
			By Theorem 4.3 in Chapter 4 we know  that $u(1)\in L^1(\R^N)\cap
			L^\infty(\R^N)$. Then $u$ satisfies the integral equation
			$$u(t+1)=G(t)*u(1)
			+\int_0^t a\cdot\nabla G(t-s)*F\bigl(u(s+1)\bigr)\,ds.$$
			
From Theorem 1.1 of Chapter 1 we know that by conservation of the mass, \eqn{\int u(x,1)\,dx=M}, one has
			\begin{equation}
				t^{\np}\|G(t)*u(1)-MG(t)\|_p \to 0,  \quad \text{ when } \quad  t\to\infty
			\end{equation}
			for all $p\in[1,\infty]$. On the other hand
			$$t^{\np}\|G(t+1)-G(t)\|_p\to 0, \quad \text{ when } \quad t\to\infty $$
			so it remains to prove that
			$$t^{\np}\biggl\|\int_0^t a\cdot\nabla
			G(t-s)*F\bigl(u(s+1)\bigr)\,ds\biggr\|_p \to 0.$$
			
			Let
			$$\epsilon(s)=\frac{F(s)}{|s|^{(N+1)/N}}.$$
			
			From \eqref{eq: 6.4} we know that
			$$\epsilon(s)\to 0,  \quad \text{ when }  |s|\to0 $$
			
			It holds
			\begin{align*}
				&{\biggl\|\int_0^t a\cdot\nabla G(t-s)
					*F\bigl(u(s+1)\bigr)\,ds\biggr\|}\\
					&\le\int_0^{t/2} \|a\cdot\nabla G(t-s)\|_p
					\|F\bigl(u(s+1)\bigr)\|_1\,ds\\
					&\qquad+\int_{t/2}^t \|a-\nabla G(t-s)\|_1
					\|F\bigl(u(s+1)\bigr)\|_p\,ds.
			\end{align*}
			
			By means of an explicit computation it can easily be shown that
			$$\|\nabla G(t)\|_p\le C_p t^{-\np-\frac12}.$$
			
			On the other hand
			$$\|F\bigl(u(s)\bigr)\|_p\le \|\epsilon\bigl(u(s)\bigr)\|_\infty
			\|u(s)\|_{p(N+1)/N}^{(N+1)/N}.$$
			
			therefore, using the estimate
			$$\|u(t)\|_p\le C_p t^{-\np} \|u_0\|_1$$
			shown in Theorem 4.3 of Chapter 4 we deduce that
			\begin{align*}
				{\int_{t/2}^t \|a\cdot\nabla G(t-s)\|_1
					\|F\bigl(u(s+1)\bigr)\|_p}\\
				{\begin{cases}
					\le |a| C\int_{t/2}^t (t-s)^{-1/2}
					\|\epsilon\bigl(u(s+1)\bigr)\|_\infty (s+1)^{-\frac{(p(N+1)-N)}{2p}}
					\,ds\\
					\le 2|a| C\delfrac({t}{2})^{-\frac{(p(N+1)-N)}{2p}}
					\delfrac({t}{2})^{1/2} \sup_{s\ge\frac{t}{2}+1}
					\|\epsilon\bigl(u(s)\bigr)\|_\infty\\
					=C'\sup_{s\ge\frac{t}{2}+1} \|\epsilon\bigl(u(s)\bigr)\|_\infty
					t^{-\np}
					\end{cases}}
			\end{align*}
			and since
			$$\lim_{t\to\infty}\sup_{s\ge\frac{t}{2}+1}
			\|\epsilon\bigl(u(s)\bigr)\|_\infty=0$$
			We see that
			$$t^{\np} \int_{t/2}^t \|a\cdot\nabla G(t-s)\|_1
			\|F\bigl(u(s+1)\bigr)\|_p\,ds\to 0, \quad \text{ when } \quad  t\to\infty $$
			
			On the other hand
			\begin{align*}
				&{\int_0^{t/2}\|a\cdot\nabla G(t-s)\|_p
					\|F\bigl(u(s+1)\bigr)\|_1\,ds} \\
				&\le|a| C\int_0^{t/2} (t-s)^{-\np-\frac{1}{2}}
					\|\epsilon\bigl(u(s+1)\bigr)\|_\infty
					(s+1)^{-1/2}\,ds\\
					&\le|a| C\delfrac({t}{2})^{-\np-\frac12}\int_0^{t/2}
					\|\epsilon\bigl(u(s+1)\bigr)\|_\infty(s+1)^{-1/2}\,ds.	
			\end{align*}
			
			therefore it is enough to prove that
			\begin{equation} \label{eq: 6.5}
			t^{-1/2}\int_0^{t/2} \|\epsilon\bigl(u(s+1)\bigr)\|_\infty
			s^{-1/2}\,ds\to 0 \quad  \text{ when } \quad t\to\infty.
			\end{equation}
			
			Fix an arbitrary $\delta>0$. Since $\|\epsilon\bigl(u(t)\bigr)\|_\infty\to
			0$  when  $t\to\infty$, there exists $T>0$ such that
			$$\|\epsilon\bigl(u(s+1)\bigr)\|_\infty\le\delta,\qquad\forall s\ge T$$
			and so
			$$t^{-1/2}\int_T^{t/2}
			\|\epsilon\bigl(u(s+1)\bigr)\|_\infty s^{-1/2}\,ds\le 2\delta t^{-1/2}
			\biggl(\frac{t}{2}-T\biggr)^{1/2}\le \sqrt2 \delta.$$
			
			On the other hand
			$$t^{-1/2}\int_0^T
			\|\epsilon\bigl(u(s+1)\bigr)\|_\infty s^{-1/2}\,ds\to 0 \quad \text{ when } \quad
			t\to\infty. $$
			
			This concludes the proof of \eqref{eq: 6.5} and thus the proof of the theorem.
		\end{proof}
		
		\begin{obs}
		This result was shown in \cite{EZ2} for initial data
			$u_0\in L^1(\R^N)\cap L^\infty(\R^N)$.
		\end{obs}
		
		\section{Initial data in $L^1(\R^N; 1+|x|)\cap L^\infty (\R^N)$}
		
		In order to estimate the rate of convergence in \eqref{eq: 6.2}, as for linear heat equation, it must be assumed that the initial data has a certain decay at infinity. In the linear case it was enough to suppose $u_0\in L^1(\R^N;1+|x|)$. In the nonlinear case we need  $u_0\in L^1(\R^N; 1+|x|)\cap L^\infty (\R^N)$. On the other hand, it will be necessary to know the behavior of the non linearity in a neighbourhood of the origin.
		
		The following result holds.
		\begin{thm}
			Suppose  that $F\in C^1(\R)$ satisfies
			\begin{equation} \label{eq: 6.6}
				|F(s)|\le c|s|^q, \qquad\forall s\in\R:|s|\le 1
			\end{equation}
			with \eqn{q>1+\frac{1}{N}}. Let $u_0\in L^1(\R^N; 1+|x|)\cap L^\infty
			(\R^N)$ be such that \eqn{\int u_0(x)\,dx=M}.
			
			Then
			\begin{equation} \label{eq: 6.7}
				\|u(t)-MG(t)\|_p\le C_p t^{-\np} p(t),\qquad\forall t\ge 1
			\end{equation}
			with
			\begin{equation} \label{eq: 6.8}
			p(t)=\begin{cases}
			-t^{1/2}& q>1+\frac{2}{N}\\
			t^{-1/2} \log (t+2)&q=1+\frac{2}{N}\\
			t^{-(N(q-1)-1)/2}&1+\frac{1}{N}<q<1+\frac{2}{N}.
			\end{cases}
			\end{equation}   
		\end{thm}
		
		\begin{obs}
			The estimate of $p$ is dependent on \eqn{q>1+\frac{2}{N}},
			\eqn{q=1+\frac{2}{N}}, \'o \eqn{1+\frac{1}{N}<q<1+\frac{2}{N}}.
			In \cite{Z2} it is shown that these estimates are optimal.
		\end{obs}
		
		\begin{proof}
			The solution $u$ to \eqref{eq: 6.1} satisfies the integral equation 
			$$u(t)=G(t)*u_0 + \int_0^t a\cdot\nabla G(t-s)*F\bigl(u(s)\bigr)\,ds$$
			
			As $u_0\in L^1(\R^N;1+|x|)$, by Lemma 1.2 of Chapter 1 we know that
			$$\|G(t)*u_0-MG(t)\|_p \le C\|u_0\|_{L^1(\R;|x|)} t^{-\np-\frac12}.$$
			
			Thus it is enough to prove that			
			$$\biggl\|\int_0^t a\cdot\nabla G(t-s)*F\bigl(u(s)\bigr)\,ds\biggr\|_p
			\le Ct^{-\np} p(t).$$
			
			We will proceed, like in the proof of Theorem 5.1, by decomposing the integral in the subintervals $[0,t/2]$ and $[t/2,t]$. As $u_0\in
			L^1(\R^N)\cap L^\infty(\R^N)$ it holds
			$$\|u(t)\|_p\le C_p(t+1)^{-\np},\qquad \forall t>0$$
			for all $p\in[1,\infty]$. In particular $u\in
			L^\infty\bigl(\R^N\times(0,\infty)\bigr)$.
			
			On the other hand, from \eqref{eq: 6.6} we deduce  
			$$\forall R>0, \exists C_R>0: |F(s)|\le C_R |s|^q, \qquad\forall s\in\R:
			|s|\le R.$$
			
			therefore, in the interval $[t/2,t]$ it holds
			\begin{align*}
				\int_{t/2}^t \|a\cdot\nabla G(t-s)\|_1
				\|F\bigl(u(s)\bigr)\|_p\,ds
				&\le|a| C\int_{t/2}^t (t-s)^{-1/2} \|u(s)\|_{pq}^q\,ds\\
				&\le|a|C\int_{t/2}^t(t-s)^{-1/2} (s+1)^{-N(pq-1)/2p}\,ds\\
				&\le|a|Ct^{-N(pq-1)/2p} t^{1/2}\\
				&=|a|Ct^{-\np}t^{-(N(q-1)-1)/2}.
			\end{align*}
			
			It can easily be shown that 
			$$t^{-(N(q-1)-1)/2}\le Cp(t)$$
			for all \eqn{q>1+\frac{1}{N}} and $t\ge 1$.
			
			In the interval $[0,t/2]$ we have
			\begin{align*}
				{\int^{t/2}_t \|a\cdot\nabla G(t-s)\|_p
					\|F\bigl(u(s)\bigr)\|_1\,ds}\\
				{\begin{cases}
					\le|a| C\int^{t/2}_t (t-s)^{-\np-\frac12} \|u(s)\|_q^q\,ds\\
					\le|a|C\delfrac({t}{2})^{-\np-\frac12}\int_0^{t/2}
					(s+1)^{-(N/2)(q-1)}\,ds.\end{cases}}
			\end{align*}
			
			On the other hand,
			$$t^{-1/2}\int_0^{t/2} (s+1)^{-(N/2)(q-1)}\,ds\le Cp(t)$$
			for all \eqn{q>1+\frac{1}{N}} and $t\ge 1$. This proves the theorem.
		\end{proof}
		
		\begin{obs}
			\begin{enumerate}
				\item[1.---]
				In in case when
				$$|F(s)|\le C|s|^q,\qquad\forall s\in\R$$
				Theorem 5.2 holds for initial data $u_0\in
				L^1(\R^N; 1+|x|)\cap L^q (\R^N)$.
				
				\item[2.---]
				When \eqn{q>1+\frac{2}{N}}, we obtain the same convergence speed as in the linear heat equation case (see Lemma 1.2 of Chapter I) as $p(t)=t^{-1/2}$.
				
				\item[3.---]
				The results  can be generalized to any nonlinearity satisfying
				$$|F(s)|\le C|s|^{(N+1)/N}\epsilon(s)$$
				with $\epsilon(\cdot)$ continuous and non-decreasing such that	 $\epsilon(s)\to0$. In this case we obtain \eqref{eq: 6.7} with
				$$p(t)=\max\biggl\{\epsilon\biggl(\biggl(\frac{t}{2}+1\biggr)^{-N/2}
				\biggr),t^{-1/2}\int_0^{t/2}\epsilon\bigl((s+1)^{-N/2}\bigr)
				(s+1)^{-1/2}\,ds\biggr\}.$$
			\end{enumerate}
		\end{obs}
		
		\chapter{Self-similar asymptotic behavior}
		
		\section{Introduction}
		
In this chapter we address the self-similar asymptotic behavior of the system
		\begin{align} \label{eq: 7.1}
		u_t-\Delta u&=a\cdot\nabla\bigl(F(u)\bigr)
		&\hbox{en $\R^N\times(0,\infty)$}\\
		u(0)&=u_0. \nonumber
		\end{align}
		It consists in finding a one-parameter family
		$$\{u_M\}_{M\in\R}=\biggl\{t^{-N/2}f_M\delfrac({x}{\sqrt{t}})
		\biggr\}_{M\in\R}$$
		of solutions to \eqref{eq: 7.1} such that ${\int f_M(x)\,dx=M}$ satisfies
		\begin{equation} \label{eq: 7.2}
			t^{\np}\|u(t)-u_M(t)\|_p\to 0  \quad \text{ when } \quad  t\to\infty   
		\end{equation}
		for $p\in[1,\infty]$.
		
		As we saw in the introduction, when the nonlinearity is homogeneous, $F(s)=|s|^{q-1}s$,
		self-similar solutions can only exist when		\eqn{q=1+\frac{1}{N}}.
		
		In this section we show that if \eqn{q=1+\frac{1}{N}}, the
		self-similar solutions exist and \eqref{eq: 7.2} holds for the general solution
We will work with the self-similar variables introduced in Chapter III, in which the self-similar profiles become stationary solutions of a new convection-diffusion equation, and \eqref{eq: 7.2} is equivalent to the convergence of the paths to these stationary states. The existence of stationary solutions (self-similar profiles) of the new equation and \eqref{eq: 7.2} will be a consequence of an abstract result for dynamic systems with adequate properties of compactness and dissipativity. In order to ensure the compactness of the trajectories, we will work in the weighted Sobolev spaces presented in Chapter III. The dissipative property necessary to apply the aforementioned abstract result will be a consequence of the $L^1(\R^N)$ contraction property which is satisfied by the semigroup associated to \eqref{eq: 7.1}.
		
		In the case of a more general nonlinearity we will see that
		\begin{equation} \label{eq: 7.3}
		\lim_{|s|\to 0}
		\frac{F(s)}{|s|^{1/N}s}=1 
		\end{equation}
		the solutions to \eqref{eq: 7.1} behave like the self-similar solutions of problem \eqref{eq: 7.1} with homogeneous nonlinearity $F(s)=|s|^{1/N}s$. 
		To demonstrate this results.we will use a stability theorem for perturbed dynamical systems due toV.~Galaktionov and J.~L.~Vázquez
		\cite{GaV}.
		
This chapter is organized as follows. In section 7.2 we will state the main results of this chapter.
In section 7.3, we will state and prove the abstract results on the existence of steady states and stability of dynamical systems. In section 7.4 we will obtain estimates for solutions in the weighted Sobolev spaces. In section 7.5 we will prove the main result. We will conclude this chapter by proving the asymptotic self-similar behavior for nonlinearities
which satisfy \eqref{eq: 7.3}.
		
		\section{Statement of the main results}
		
		Let us consider the equation
		\begin{equation} \label{eq: 7.4}
		u_t-\Delta u=a\cdot\nabla (|u|^{1/N}u) \text{ in } \R^N\times(0,\infty).
		\end{equation}
		As we saw in the introduction, if $u=u(x,t)$ is a self-similar solution to \eqref{eq: 7.4}, i.e.,
		$$u(x,t)=t^{-N/2}f\delfrac({x}{\sqrt{t}})$$ 
		the profile $f=f(y)$ has to satisfy the elliptic equation	
		\begin{equation} \label{eq: 7.5}
			-\Delta f-\frac{y\cdot\nabla f}{2} - \frac{N}{2} f=a\cdot\nabla
			(|f|^{1/N}f) \text{ in } \R^N.
		\end{equation}
		The following result on existence of self-similar solutions holds.
		\begin{thm}[\cite{AEZ}]
			For every $M\in\R$ there exists a unique profile $f_M\in L^1(\R)$
			solution to \eqref{eq: 7.5} such that
			$$\int f_M(y)\,dy=M.$$
			
			Moreover,
			\begin{itemize}
				\item[(i)]
				$f_M\in H^2(K)\cap C^\infty(\R^N)\cap W^{2,p}(\R^N), \qquad\forall
				p\in [1,\infty)$;
				
				\item[(ii)]
				$f_M>0$ (resp.\ $<0$) in $\R^N$ if $M>0$ (resp.\ $M<0$);
				
				\item[(iii)]
				$f_{M_1}>f_{M_2}$ in $\R^N$ if $M_1>M_2$.
			\end{itemize}
		\end{thm}
		
		This result is shown in \cite{AEZ} by means of a fixed-point argument applied to the elliptic equation \eqref{eq: 7.5}. In this chapter we will not reproduce that proof, but that we will adopt the dynamical point of view of \cite{Z1}.
		
		Let
		$$u_M(x,t)=t^{-N/2}f_M\delfrac({x}{\sqrt{t}})$$
		where $f_M$ is the self-similar profile of Theorem 6.1.
		
		It can easily be shown that 
		$$u_M(0)=M\delta$$
		$\delta$ being the Dirac mass, i.e.,
		$$\int u_M(x,t)\varphi(x)\,dx\to M\varphi(0),  \quad \text{ when } \quad 
		t\to 0^+ $$
		for all $\varphi\in \BC(\R^N)$.
		
		Regarding the asymptotic behavior of the general solution to \eqref{eq: 7.4}, we have the following result:
		\begin{thm}
			Let $u_0\in L^1(\R^N)$ be such that
			$$\int u_0(x)\,dx=M$$
			and let $u=u(x,t)$ be the solution to $\eqref{eq: 7.4}$ such that $u(0)=u_0$.
			
			Then $\eqref{eq: 7.2}$ holds for all $p\in[1,\infty]$.
		\end{thm}
		
This theorem proves that when $ t \to \infty $ the solutions to \eqref{eq: 7.4} behave as the self-similar solutions. This theorem extends the results proved in Chapters I and II for the linear heat equation and for the Burgers equation in one spatial dimension.

		\section{Steady states and convergence of trajectories for dynamical systems}
		
		Let $(Z,d)$ be a complete metric space and
		$\{S(t)\}_{t\ge 0}$ a dynamical system $Z$ that satisfies:
		\begin{align*}
			\text{\hskip\parindent(i)}&\,S(t)\in C(Z,Z), \quad \forall t\ge 0;  \\
			\text{\hskip\parindent(ii)}&\,S(0)=I= \hbox{identity};  \\
			\text{\hskip\parindent(iii)}&\,S(t+s)=S(t)\circ S(s), \quad \forall t,s\ge 0;
			\\
			\text{\hskip\parindent(iv)}&\,S(t)z\in C\bigl([0,\infty);Z\bigr), \quad \forall z\in Z.
		\end{align*}
		
		We have the following result of existence of steady states and convergence of trajectories
		\begin{thm}
			Suppose  that $\bigcup_{t\ge 0}\{S(t)z\}$ is relatively compact in $Z$ for all $z\in Z$. Suppose  that there exist continuous functions $\psi\in C(Z;\R)$,$\phi\in C(\R,\R)$ such that
			\begin{align}
				&\psi\bigl(S(t)z\bigr)=\psi (z),\qquad\forall t\ge 0,\forall z\in Z  \label{eq: 7.6}\\
				&\phi\Bigl(d\bigl(S(t)z_1,S(t)z_2\bigr)\Bigr)
				\hbox{ is strictly decreasing in $t$ if $z_1\ne z_2$ and $\psi(z_1)=\psi(z_2)$.} \label{eq: 7.7}
			\end{align}
			
			Then for all $M\in R(\psi)$ ($={}$ range of
			$\psi$) there exists a unique  $z_M\in Z$ such that
			\begin{equation} \label{eq: 7.8}
				S(t)z_M=z_M, \qquad\forall t\ge 0 \hbox{ and } \psi(z_M)=M
			\end{equation}
			and for all $z\in Z$ such that $\psi(z)=M$, it holds
			\begin{equation} \label{eq: 7.9}
			S(t)z\to z_M  \quad  when  \quad t\to\infty.
			\end{equation}
		\end{thm}
		
		\begin{proof}
			We proceed in several steps.
			
			\begin{enumerate}
				\item[]{\em Uniqueness.}
				This is an immediate consequence of \eqref{eq: 7.7}. Indeed, if $z_M\in Z$ and $z_M'\in
				Z$ satisfy
				$$S(t)z_M=z_M, S(t)z'_M=z'_M, \qquad\forall t\ge 0,
				\psi (z_M)=\psi (z_M')=M$$
				then $\Bigl(d\bigl(S(t)z_M,S(t)z_M'\bigl)\Bigr)$ es independiente de
				$t$, and therefore, in view of \eqref{eq: 7.7}, $z_M=z_M'$.
				
				\item[]{\em Existence of steady states and convergence of trajectories.} 
		Fix arbitrary $M\in R(\psi)$ and $z\in Z$  such that $\psi(z)=M$.
				
	Let us consider the
				$\omega$-limit set $\omega(z)$ of the trajectory $\{S(t)z\}_{t\ge 0}$: 
				$$\omega(z)=\{\,\eta\in Z:\exists t_n\to\infty,
				S(t_n)z\to\eta\,\}.$$
				
				By the hypothesis of compactness of trajectories, we know that $\omega(z)\neq\phi$. On the other hand,
				\begin{equation} \label{eq: 7.10}
					\omega(z)\subset\{\,\eta\in Z:\psi(\eta )=M\,\} 
				\end{equation}
				as $\psi\in C(Z;\R)$.
		
		Thus it is enough to prove that
				\begin{equation} \label{eq: 7.11}
					\omega(z)\subset\{\,\eta\in Z:S(t)\eta=\eta ,\forall
					t\ge 0\,\}.   
				\end{equation}
				
				Indeed, from \eqref{eq: 7.10} and \eqref{eq: 7.11} we deduce that
				$$\omega(z)\subset\{\,\eta\in Z:S(t)\eta=\eta,\forall t\ge
				0,\psi(\eta)=M\,\}$$
				and since $\omega(z)\neq\phi$, this guarantees the existence of a solution $z_M$ to \eqref{eq: 7.8}. Since, on the other hand, we know that this solution is unique, we deduce that
				$$\omega(z)=\{z_M\}$$
				which is equivalent to \eqref{eq: 7.9}.
				
	With the goal of obtaining \eqref{eq: 7.10} we will prove that given any $\tau>0$
				\begin{equation} \label{eq: 7.12}
				S(\tau)\eta=\eta,\qquad\forall\eta\in\omega(z).
				\end{equation}
				Fix an arbitrary $\tau>0$ and $\eta\in\omega(z)$. Then there exists a
				sequence $t_n\to\infty$ such that
				$$S(t_n)z\to\eta$$
				and therefore, since $S(t)\in C(Z;\R)$ for all $t>0$, by the semigroup property we deduce that
				\begin{equation} \label{eq: 7.13}
					S(t+t_n)z\to S(t)\eta,\qquad\forall t>0.
				\end{equation}
				On the other hand, by \eqref{eq: 7.6}, $\psi\bigl(S(\tau)\eta\bigr)=\psi(\eta)$, and in view of \eqref{eq: 7.7}, the following alternative holds: either \eqref{eq: 7.12} holds or 
				\begin{equation} \label{eq: 7.14}
					\phi\Bigl(d\bigl(S(t+\tau)\eta ,S(t)\eta\bigr)\Bigr)\hbox{ is strictly decreasing.}
				\end{equation}
				Therefore, it is enough to prove that \eqref{eq: 7.14} is not fulfilled.
				
				As $\phi\in C(\R;\R)$, por \eqref{eq: 7.13} it holds
				\begin{equation} \label{eq: 7.15}
					\phi\Bigl(d\bigl(S(t+\tau)\eta ,S(t)\eta\bigr)\Bigr)=\lim_{n\to\infty}
					\varphi_\tau(t+t_n). 
				\end{equation}
				where
				$$\varphi_\tau(t)=\Bigl(d\bigl(S(t+\tau)z,S(t)z\bigr)\Bigr).$$
				
				From \eqref{eq: 7.7} we deduce  that $\varphi_\tau(\cdot)$ is decreasing. As the
				trajectory $\bigcup_{t\ge 0} \{S(t)z\}$ is relatively compact in $Z$
				and $\phi\in C(\R;\R)$ it follows that the limit
				\begin{equation} \label{eq: 7.16}
					\lim_{t\to\infty}\varphi_\tau(t)=\varphi_{\tau,\infty}
				\end{equation}
				exists. Combining \eqref{eq: 7.15} and \eqref{eq: 7.16} we see that
				$$\phi\Bigl(d\bigl(S(t+\tau)\eta,S(t)\eta\bigr)\Bigr)
				=\varphi_{\tau,\infty},\qquad\forall t>0$$ and therefore \eqref{eq: 7.14} does not hold. This completes the proof.
			\end{enumerate}
		\end{proof}
		
		\section{Estimates in the weighted Sobolev spaces}
		
		Let us consider the solution $u=u(x,t)$ to
		\begin{align} \label{eq: 7.17}
		u_t-\Delta u&=a\cdot\nabla(|u|^{1/N}u)&\hbox{en
			$\R^N\times(0,\infty)$}\\
		u(0)&=u_0 \nonumber
		\end{align}   
		with $u_0\in L^1(\R^N)\cap L^\infty(\R^N)$.
		As seen in Theorems 4.1 and 4.3 of Chapter 4, we deduce that  $u\in
		C\bigl([0,\infty);L^1(\R^N)\bigr)
		\cap L^\infty\bigl(\R^N\times(0,\infty)\bigr)$
		and
		$$\|u(t)\|_p\le C_p \|u_0\|_1 t^{-\np},\qquad\forall t>0$$
		from where we deduce that
		\begin{equation} \label{eq: 7.18}
		\|u(t)\|_p\le C'_p(t+1)^{-\np},\qquad\forall t>0
		\end{equation}
		For all
		$p\in[1,\infty]$ with $C'_p$ depending on $\|u_0\|_1$ and $\|u_0\|_p$.

		We define the corresponding function $v=v(y,s)$ in the self-similar variables:
		\begin{equation} \label{eq: 7.19}
			v(y,s)=e^{sN/2}u(e^{s/2}y,e^s-1).
		\end{equation}
		
		It can easily be shown that  $v=v(y,s)$ solves		\begin{align} \label{eq: 7.20}
		v_s-\Delta v-\frac{y\cdot\nabla v}{2}-\frac{N}{2}
		v&=a\cdot\nabla(|v|^{1/N}v)&\hbox{en $\R^N\times(0,\infty)$}\\
		v(0)&=u_0. \nonumber
		\end{align}
		
		On the other hand, from \eqref{eq: 7.18} we deduce that
		\begin{equation} \label{eq: 7.21}
			v\in L^\infty\bigl(0,\infty ;L^1(\R^N)\cap L^\infty(\R^N)\bigr).
		\end{equation}
		
	In order to ensure the compactness of the trajectory $\{v(s)\}_{s\ge 0}$ that is needed to apply the Theorem 6.3, we will work in the weighted Sobolev spaces studied in the chapter 3.
		
		We have the following result.
		\begin{prop} 
			Suppose  that $u_0\in L^2(K)\cap L^\infty(\R^N)$. Then the
			solution $v=v(y,s)$ of \eqref{eq: 7.20} satisfies
			\begin{align}
				v&\in L^\infty\bigl(0,\infty ;L^2(K)\bigr)\cap
				C\bigl([0,\infty);L^2(K)\bigr) \label{eq: 7.22}\\
				v&\in L^\infty\bigl(1,\infty ;H^1(K)\bigr) \label{eq: 7.23}.
			\end{align}
			Moreover, the trajectory $\{v(s)\}_{s\ge 0}$ is relatively compact  in $L^2(K)\cap L^\infty(\R^N)$.
		\end{prop}
		
		\begin{proof}
			We first observe that by the biunivocal correspondence \eqref{eq: 7.19} between the solutions to \eqref{eq: 7.17} and
			\eqref{eq: 7.20}, the system \eqref{eq: 7.20} admits a unique solution that is the given by
			\eqref{eq: 7.19}.
			
			As
			$$\int u(x,t)\,dx=\int u_0(x)\,dx=M, \qquad\forall t>0 $$
			in view of (19) it can easily be shown that 
			$$\int v(y,s)\,dy=M,\qquad\forall s>0.$$
			Thus the orthogonal projection of $v(s)$ on the first eigenspace of the operator $L$ in $L^2(K)$ is independent of $s$ and it holds
			\begin{equation} \label{eq: 7.24}
				v(y,s)=M\varphi_1(y)+\tilde v(y,s)
			\end{equation}
			with $\tilde v(s)\in \varphi_1^\perp$ for all $s>0$, $\varphi_1$ being the first eigenfunction
			$\varphi_1(y)=(4\pi)^{-N/2} K^{-1}(y)=(4\pi)^{-N/2}
			\exp(-|y|^2/4)$.
			
			By virtue of \eqref{eq: 7.24} one may observe that \eqref{eq: 7.22} is equivalent to
			\begin{equation} \label{eq: 7.25}
			\tilde v(s)\in L^\infty\bigl(0,\infty ;L^2(K)\bigr)\cap
			C\bigl([0,\infty);L^2(K)\bigr). 
			\end{equation}
			The following holds:
			\begin{align} \label{eq: 7.26}
			\tilde v_s+L\tilde v-\frac{N}{2}\tilde v&=a\cdot\nabla(|v|^{1/N}v)
			&\hbox{en $\R^N\times(0,\infty)$}\\
			\tilde v(0)&=\tilde u_0 \nonumber
			\end{align}
			with $\tilde u_0=u_0-M\varphi_1.$
			Multiplying the equation \eqref{eq: 7.26} by $\tilde vK$ and integrating in $\R^N$ we obtain			\begin{align} \label{eq: 7.27}
				{\frac12\frac{d}{ds}\|\tilde v(s)\|_K^2+\|\nabla\tilde
					v(s)\|_K^2-\frac{N}{2}\|\tilde v(s)\|_K^2} &\le|a|\delfrac({N+1}{N}) \int|\nabla v| |v|^{1/N}|\tilde v|K\\
					&\le |a|\delfrac({N+1}{N})\|v\|_\infty^{1/N}\int|\nabla(M\varphi_1
					+\tilde v)| |\tilde v|K. \nonumber
			\end{align}
			
			By Proposition 3.1 of Chapter 3 we deduce that
			$$\|\nabla w\|_K^2 \ge \frac{N+1}{2} \|w\|_K^2,
			\qquad\forall w\in\varphi_1^\perp\cap H^1(K)$$
			and therefore
			\begin{equation} \label{eq: 7.28}
				\|\nabla w\|_K^2-\frac{N}{2}\|w\|_K^2\ge \frac{1}{N+1} \|\nabla
				w\|_K^2,  \qquad\forall w\in\varphi_1^\perp\cap H^1(K). 
			\end{equation}
			
			By \eqref{eq: 7.27}, \eqref{eq: 7.28} and thanks to the fact that $v\in
			L^\infty\bigl(\R^N\times(0,\infty)\bigr)$ we obtain
			$$\frac{d}{ds}\|\tilde v(s)\|_K^2+\frac{2}{N+1}\|\nabla\tilde
			v(s)\|_K^2\le C_1\biggl\{\|\tilde v\|_K+\int |\nabla\tilde v| |\tilde
			v|K\biggr\}$$
			for $C_1>0$ sufficiently large . On the other hand
			$$\int|\nabla\tilde v||\tilde v|K\le\epsilon\|\nabla\tilde v\|_K^2
			+\frac{1}{4\epsilon}\|\tilde v\|_K^2$$
			for all $\epsilon>0$ and therefore, eligiendo $\epsilon>0$
			sufficiently small  we deduce that
			\begin{align*}
			\frac{d}{ds}\|\tilde v(s)\|_K^2 +\frac{1}{N+1} \|\nabla\tilde v(s)\|_K^2
			&\le C_2\{\|\tilde v(s)\|_K+\|\tilde v(s)\|_K^2\}\\
			&\le C_3\{1+\|\tilde v(s)\|_K^2\}.
			\end{align*}
			
		From the inequality \eqref{eq: 4.16} of Chapter III it follows that, for all
			$\epsilon>0$ there exists $C_\epsilon>0$ such that
			$$\|w\|_K^2\le\epsilon\|w\|_{1,K}^2 +C_\epsilon\|w\|_2^2,
			\qquad\forall w\in H^1(K).$$
			
		Therefore, choosing $\epsilon>0$ sufficiently small and keeping into account that $\tilde v\in L^\infty\bigl(0,\infty; L^2(\R^N)\bigr)$
			we obtain
			\begin{equation} \label{eq: 7.29}
			\frac{d}{ds}\|\tilde v(s)\|_K^2 +\frac{1}{2(N+1)} \|\nabla\tilde
			v(s)\|_K^2\le C_4  
			\end{equation}
			and therefore
			$$\frac{d}{ds}\|\tilde v(s)\|_K^2+\frac14 \|\tilde v(s)\|_K^2\le C_4.$$
			
			From this differential inequality we obtain that $\tilde v\in
			L^\infty\bigl(0,\infty; L^2(K)\bigr)$ and therefore $v\in
			L^\infty\bigl(0,\infty; L^2(K)\bigr)$. On the other hand, it can easily be checked that $v\in C\bigl([0,\infty); L^2(K)\bigr)$. Thus, \eqref{eq: 7.22} is proven.
			
	From \eqref{eq: 7.22} and thanks to the regularizing effect of the semigroup associated to system \eqref{eq: 7.20}, we deduce  \eqref{eq: 7.23}.
			
			Indeed, $S(s)$ being the semigroup generated by the operator
			\eqn{L-\frac{N}{2}I}, for all $\tau>0$, the function $v$ solution to \eqref{eq: 7.20}
			satisfies the integral equation 
			\begin{equation} \label{eq: 7.30}
				v(s+\tau)=S(s)v(\tau)+\int_0^s
				S(s-\sigma)\bigl[a\cdot\nabla\bigl(|v(\sigma+\tau)|^{1/N}
				v(\sigma+\tau)\bigr) \bigr]\,d\sigma. 
			\end{equation}
			
			As a consequence of the estimate \eqref{eq: 3.40} from Chapter 3 for the semigroup $S(\cdot)$ it holds
			\begin{equation} \label{eq: 7.31}
			\|S(s)w\|_{1,K} \le C\biggl(1+\frac{1}{\sqrt{s}}\biggr)\|w\|_K,
			\qquad\forall w\in L^2(K) 
			\end{equation}
			para $C>0$ sufficiently large. Taking  $H^1(K)$ norms in \eqref{eq: 7.30} and applying \eqref{eq: 7.31} we obtain
			\begin{align*}
				&{\|v(s+\tau)\|_{1,K}}\\
				&\le C\biggl(1+\frac{1}{\sqrt{s}}\biggr)\|v(\tau)\|_K \\
				&\quad+C|a|\int_0^s\biggl(1+\frac{1}{\sqrt{s-\tau}}\biggr)
					\|\nabla(|v|^{1/N}v)(\sigma+\tau)\|_K\,d\sigma\\
				&\le C\biggl(1+\frac{1}{\sqrt{s}}\biggr)\|v(\tau)\|_K\\
				&\quad+C|a|\delfrac({N+1}{N}) \|v\|_{L^\infty(\R^N\times(0,s+\tau))}^{1/N}
					\int_0^s\biggl(1+\frac{1}{\sqrt{s-\tau}}\biggr)
					\|v(\sigma+\tau)\|_{1,K}\,d\sigma.
			\end{align*}
			As $v\in L^\infty\bigl(\R^N\times(0,\infty)\bigr)$, we obtain
			\begin{equation} \label{eq: 7.32}
				\|v(s+\tau)\|_{1,K}\le C\biggl\{\biggl(1+\frac{1}{\sqrt{s}}\biggr)
				\|v(\tau)\|_K +\int_0^s\biggl(1+\frac{1}{\sqrt{s-\sigma}}\biggr)
				\|v(\sigma+\tau)\|_{1,K}\,d\sigma\biggr\}  
			\end{equation}
			para $C>0$ sufficiently large.
			
			We now need the following Gronwall lemma.
			\begin{lem}
		Let $T>0$, $A$, $B\ge0$ and $\varphi\in C([0,T])$ a non-negative function
				such that
				\begin{equation} \label{eq: 7.33}
				\varphi(t)\le A(1+t^{-1/2})+B\int_0^t
				\bigl(1+(t-s)^{-1/2}\bigr)\varphi(s)\,ds.
				\end{equation}
				
		Then, there exists a constant $C>0$ that depends only on $T$ and $B$ such that
				\begin{equation} \label{eq: 7.34}
					\varphi(t)\le CAt^{-1/2}, \qquad\forall t\in[0,T]. 
				\end{equation}
			\end{lem}
			
			\begin{proof}
				The proof is analog to that of Lemma 4.2 in Chapter 4.
	Multiplying \eqref{eq: 7.33} by $t^{1/2}$ we obtain
				\begin{align} \label{eq: 7.35}
					t^{1/2}\varphi(t)&\le A(1+t^{1/2})
					+Bt^{1/2}\int_0^t\bigl(1+(t-s)^{-1/2}\bigr)
					s^{-1/2}s^{1/2}\varphi(s)\,ds\\
					&\le A(1+T^{1/2})
					+Bt^{1/2}\int_0^t\bigl(1+(t-s)^{-1/2}\bigr)
					s^{-1/2}\psi(s)\,ds  \nonumber
				\end{align}
				with $\psi(t)=\sup_{s\in[0,t]} s^{1/2}\varphi(s)$. As $\psi$ is non-decreasing it follows that 
				$$t^{1/2}\varphi(t)\le A'
				+Bt^{1/2}\psi(t)\int_0^t\bigl(s^{-1/2}+(t-s)^{-1/2}s^{-1/2}\bigr)
				\,ds$$
				with $A'=A(1+T^{1/2})$. As
				$$\int_0^t\bigl(s^{-1/2}+(t-s)^{-1/2}s^{-1/2}\bigr)\,ds
				=2t^{1/2}+\int_0^1(1-s)^{-1/2}\,ds=2t^{1/2}+\alpha$$
				we obtain
				$$t^{1/2}\varphi(t)\le A'+B(2t+\alpha t^{1/2})\psi(t).$$
				
				Fix $t_0\in[0,T]$ such that ${B(2t_0+\alpha t^{1/2})=\frac{1}{2}}$.
				So, for all $t\in[0,t_0]$ we deduce that
				$$t^{1/2}\varphi(t)\le A'+\frac12\psi(t),\qquad\forall t\in[0,t_0]$$
				and taking supremums in this inequality we observe that
				\begin{equation} \label{eq: 7.36}
					\psi(t)\le 2A',\qquad\forall t\in[0,t_0].
				\end{equation}
				
				For $t\ge t_0$, choose ${\delta\le\frac{t_0}{2}}$. Decomposing the last integral in \eqref{eq: 7.35} in the segments $[0,\delta]$,
				$[\delta,t-\delta]$ and $[t-\delta,t]$, using the fact that $\psi$ is non-decreasing and taking supremums in \eqref{eq: 7.35} we obtain
				\begin{align} \label{eq: 7.37}
					\psi(t)&\le A'+B'\biggl(\int_0^\delta
					\bigl(s^{-1/2}+(t-s)^{-1/2}s^{-1/2}\bigr)\,ds {+\int_{t-\delta}^t
						\bigl(s^{-1/2}+(t-s)^{-1/2}s^{-1/2}\bigr)\,ds\biggr)\psi(t)}\nonumber\\
					&+B'\int_\delta^{t-\delta}
					\bigl(s^{-1/2}+(t-s)^{-1/2}s^{-1/2}\bigr)\psi(s)\,ds.  \nonumber
				\end{align}
				with $B'=BT^{1/2}$. Choosing $\delta>0$ sufficiently small we observe that
				$$B'\biggl(\int_0^\delta
				\bigl(s^{-1/2}+(t-s)^{-1/2}s^{-1/2}\bigr)\,ds\biggr)\le\frac{1}{2},
				\qquad\forall t\in[t_0,T]$$
				y since On the other hand
				$$s^{-1/2}+(t-s)^{-1/2}s^{-1/2}\le\delta^{-1/2}(1+\delta^{-1/2}),
				\qquad\forall t\in[t_0,T],\forall s\in[\delta,t-\delta]$$
				we obtain
				\begin{align*}
					{\psi(t)\le 2A'
						+B'\delta^{-1/2}(1+\delta^{-1/2})\int_0^t\psi(s)\,ds}\\
					{\le
						A'+2A'B'\delta^{-1/2}(1+\delta^{-1/2})t_0+B'\delta(1+\delta^{-1/2})
						\int_{t_0}^t\psi(s)\,ds.}
				\end{align*}
				
				Applying the classical Gronwall lemma we conclude that
				\begin{equation} \label{eq: 7.38}
					\psi(t)\le\bigl(2A'+2A'B'\delta^{-1/2}(1+\delta^{-1/2})t_0\bigr)
					\exp\bigl(B'\delta^{-1/2}(1+\delta^{-1/2})(t-t_0)\bigr).
				\end{equation}
				
				Combining \eqref{eq: 7.36} and \eqref{eq: 7.38} we see that
				\begin{equation} \label{eq: 7.39}
				\psi(t)\le CA,\qquad\forall t\in[0,T] 
				\end{equation}
				with $C>0$ depending on $B$ and $T$. Clearly \eqref{eq: 7.39} implies \eqref{eq: 7.34}.
			\end{proof}
			
			\paragraph*{End of proof to Proposition 6.4:}
			Applying Lemma 6.5 in the inequality \eqref{eq: 7.32} we see that there exists $C>0$ such that
			\begin{equation} \label{eq: 7.40}
				\|v(s+\tau)\|_{1,K}\le Cs^{-1/2}\|v(\tau)\|_K,\qquad\forall\tau>0,
				\forall s\in[0,1]. 
			\end{equation}
			
			Taking $s=1$ and keeping into account that $v\in L^\infty\bigl(0,\infty;
			L^2(K)\bigr)$ we obtain
			$$\|v(1+\tau)\|_{1,K}\le C,\qquad\forall\tau> 0$$
			which implies \eqref{eq: 7.23}.
			
			From \eqref{eq: 7.23} and by the compactness of the injection $H^1(K)\subset L^2(K)$ we deduce that $\{v(s)\}_{s\ge0}$ is relatively compact  in $L^2(K)$. From the continuity of the injection $L^2(K)\subset L^1(\R^N)$ we deduce in particular that $\{v(s)\}_{s\ge0}$ is relatively compact  en
			$L^1(\R^N)$.
			
			On the other hand, from the estimates shown in Theorems 4.3 and 4.4
			of Chapter 4 we deduce that
			$$v\in L^\infty\bigl(1,\infty;W^{1,p}(\R^N)\bigr)$$
			for all $p\in[1,\infty]$. As $W^{1,p}(\R^N)\subset L^\infty(\R^N)$
			for $p>N$, by interpolation, we deduce that $\{v(s)\}_{s\ge0}$ is relatively compact in $L^\infty(\R^N)$.
			
			Therefore, the trajectory is relatively compact  in $L^2(K)\cap
			L^\infty(\R^N)$. This concludes the proof.
		\end{proof}

		\section{Strong contraction property in $L^1(\R^N)$.}
		
		The following result holds.
		\begin{prop}
		Let $u_0$, $u_1\in L^2(K)\cap L^\infty(\R^N)$ be such that $u_0\ne u_1$ y
			$$\int u_0(y)\,dy=\int u_1(y)\,dy.$$
			Let $v$ and $w$ be the corresponding solutions to \eqref{eq: 7.20}.
			
			Then $\|v(s)-w(s)\|_1$ is a strictly decreasing function of
			$s$.
		\end{prop}
		
		\begin{proof}
			It is enough to show that $\bigl\|\bigl(v(s)-w(s)\bigr)^+\bigr\|_1$ and
			$\|\bigl(v(s)-w(s)\bigr)^-\|_1$ are strictly decreasing functions. As the proof is analog in both cases we will only consider the function 			
	$$\theta(s)=\|\bigl(v(s)-w(s)\bigr)^+\|_1=\int_{\Omega(s)}
			\bigl(v(y,s)-w(y,s)\bigr)\,dy$$
			where
			$$\Omega(s)=\{\,y\in\R^N: v(y,s)>w(y,s)\,\}.$$
			
			By Theorem 4.3 of Chapter 4 and taking into account \eqref{eq: 7.19}, which holds between the solutions to \eqref{eq: 7.17} and \eqref{eq: 7.20}, we see that $v\in C
			\bigl((0,\infty);W^{2,p}(\R^N)\bigr)$ for all $p\in (1,\infty)$. In particular, $v(s)$ and $w(s)$ are continuous functions of $y$ and therefore,
			$\Omega(s)$ is an open set for all $s>0$. Likewise
			$$\Omega_*(s)=\{\,y\in\R^N: v(y,s)<w(y,s)\,\}$$
			is an open set for all $s>0$.
			
			Let us see that $\Omega(s)$ and $\Omega_*(s)$ are non-empty for all $s>0$. Indeed suppose that there exists $s_0>0$ such that $\Omega(s_0)=\emptyset$ (or
			$\Omega_*(s_0)=\emptyset$). As
			$$\int v(y,s)\,dy=\int u_0(y)\,dy=\int u_1(y)\,dy=\int w(y,s)\,dy,
			\qquad\forall s>0$$
			we deduce 
			\begin{equation} \label{eq: 7.41}
				v(y,s_0)=w(y,s_0), \qquad\forall y\in\R^N. 
			\end{equation}
			The following lemma of backward uniqueness ensures that \eqref{eq: 7.41}
			implies $u_0=u_1$, which goes against the hypotheses of Proposition 6.6.
			
			\begin{lem}
			Let $v$ and $w$ be solutions to $\eqref{eq: 7.20}$ with initial data $u_0$ and $u_1$ respectively in $L^2(K)$. Suppose  that there exists $s_0>0$ such that
				$v(s_0)=w(s_0)$. Then, $u_0=u_1$.
			\end{lem}
			
			\begin{proof}
		By \eqref{eq: 7.22} we know that $v,w\in C\bigl([0,\infty); L^2(K)\bigr)$. On another hand, integrating the inequality \eqref{eq: 7.29} with respect to $s$ and using
				\eqref{eq: 7.22} we deduce  that $v,w\in L^2\bigl(0,s_0; H^1(K)\bigr)$.
				
				Set
				$$z(y,s)=v(y,s)-w(y,s)$$
				which satisfies
				\begin{align*}
				z_s+Lz-\frac{N}{2} z&=a\cdot\nabla\bigl(b(y,s)z\bigr)&\hbox{ en
					$\R^N\times[0,s_0]$}\\
				z(s_0)&=0
				\end{align*}
				con
				$$b(y,s)=\frac{|v|^{1/N} v(y,s)-|w|^{1/N} w(y,s)}{v(y,s)-w(y,s)}.$$
				
				As $v,w\in L^\infty\bigl(\R^N\times(0,\infty)\bigr)$, we see that $b\in 
				L^\infty\bigl(\R^N\times[0,s_0]\bigr)$ therefore, since the gradient operator is continuous from $L^2(K)$ in $\bigl(H^1(K)\bigr)^*$ (dual de
				$H^1(K)$) we see that
				$$\|a\cdot\nabla\bigl(b(s)z(s)\bigr)\|_{(H^1(K))^*}\le|a| C\|b(s)z(s)\|_K
				\le |a|C\|b(s)\|_\infty \|z(s)\|_K.$$
				Theorem 1.1 on backward uniqueness J.~M.~Ghidaglia \cite{Gh} (see also J.-L.~Lions and B.~Malgrage \cite{LioMa}), applied along with
				\eqn{A=L-\frac{N}{2}I},  $V=L^2(K)$, $D(A)=H^1(K)$ y
				$H=\bigl(H^1(K)\bigr)^*$, implies that $z(s)=0$ for all
				$s\in[0,s_0]$ and therefore $u_0=u_1$. 
			\end{proof}
			
			\paragraph*{End of proof of Proposition 6.6:}
			Multiplying by $\sgn\bigl(v(s)-w(s)\bigr)^+$ the equation satisfied by $v-w$ and integrating in $\R^N$ we obtain
			\begin{align*}
			\theta'(s)&=\int_{\Omega(s)} \bigl(v_s(y,s)-w_s(y,s)\bigr)\,dy\\
			&=\int_{\Omega(s)} \Delta\bigl(v(y,s)-w(y,s)\bigr)\,dy\\
			&\qquad+\int_{\Omega(s)}\biggl[y\cdot\frac{\nabla v(y,s)}{2}+\frac{N}{2}
			v(y,s)-y\cdot\frac{\nabla w(y,s)}{2}-\frac{N}{2} w(y,s)\biggr]\,dy\\
			&\qquad+\int_{\Omega(s)} a\cdot\biggl[\nabla\bigl(|v(y,s)|^{1/N} v(y,s)
			\bigr)-\nabla\bigl(|w(y,s)|^{1/N} w(y,s)\bigr)\biggl]\,dy.
			\end{align*}
			
			It holds
			\begin{align*}
				{\int_{\Omega(s)}\biggl[y\cdot\frac{\nabla v(y,s)}{2}+\frac{N}{2}
					v(y,s)-y\cdot\frac{\nabla w(y,s)}{2}-\frac{N}{2} w(y,s)\biggr]\,dy}\\
				{\begin{cases}
					=\frac12\int_{\Omega(s)}\div\biggl[y\cdot\bigl(v(y,s)-w(y,s)\bigr)
					\biggr]\,dy\\
					=\frac12\int\div\biggl[y\cdot\bigl(v(y,s)-w(y,s)\bigr)^+\biggr]\,dy
					=0
					\end{cases}}\end{align*}
			put that $y\cdot\bigl(v(y,s)-w(y,s)\bigr)^+\in W^{1,1}(\R^N)$ as
			$v(s),w(s)\in H^1(K)$ for all $s>0$.
			
			Likewise, for the same reason
			$$\int_{\Omega(s)} a\cdot\nabla\bigl(|v(y,s)|^{1/N} v(y,s)
			-|w(y,s)|^{1/N} w(y,s)\bigr)\,dy=0.$$
			
			Therefore
			\begin{equation} \label{eq: 7.42}
				\theta'(s)=\int_{\Omega(s)}  \Delta\bigl(v(y,s)-w(y,s)\bigr)\,ds.
			\end{equation}
			
			Identity \eqref{eq: 7.42} makes sense due to the regularizing effect of the semigroup associated with the equation \eqref{eq: 7.20} we have
			$v,w\in C\bigl((0,\infty);H^2(K)\bigr)\cap
			C^1\allowbreak\bigl((0,\infty);L^2 (K)\bigr)$ and Thus 
			$v,w\in C\bigl((0,\infty);W^{2,1}(\R^N)\bigr)\cap
			C^1\bigl((0,\infty);L^1(\R^N)\bigr)$.
		
			From \eqref{eq: 7.42} we deduce that $\theta$ non-increasing  as			\begin{equation} \label{eq: 7.43}
				\int_{\{\varphi>0\}}\Delta\varphi\le0\,dx, \qquad\forall\varphi\in
				W^{2,1}(\R^N)   
			\end{equation}
			(see \cite{AEZ} for a detailed proof  of \eqref{eq: 7.43}).  Note that if the set $\{\varphi>0\}$ is regular, one would have
			${\int_{\{\varphi>0\}}\Delta\varphi\,dx=\int_{\partial\{\varphi>0\}}
				\frac{\partial\varphi}{\partial\nu}d\sigma\le0}$ as ${\frac{\partial\varphi}{\partial\nu}\le0}$ on the boundary $\{\varphi>0\}$. In the general case, we obtain \eqref{eq: 7.43} thanks to the Sard lemma.
			
			Let us now see that $\theta$ is strictly decreasing. We argue by contradiction. If $\theta(s_0)=\theta(s_1)$ for $s_1>s_0$, from
			\eqref{eq: 7.42} and since $\theta$ is increasing we deduce that
			\begin{equation} \label{eq: 7.44}
				\int_{\Omega(s)} \Delta\bigl(v(y,s)-w(y,s)\bigr)\,dy=0,
				\qquad\forall s\in[s_0,s_1]. 
			\end{equation}
			
			We need the following lemma shown in \cite{AEZ}.
			
			\begin{lem}
				Let $\varphi\in W^{2,1}(\R^N)$ be such that
				$$\int_{\{\varphi>0\}} \Delta\varphi(x)\,dx=0.$$
				
				Then $\varphi^+\in W^{2,1}(\R^N)$ y
				$$\Delta\varphi^+=\sgn(\varphi^+) \Delta\varphi \text{ in } \R^N.$$
			\end{lem}
			
			In view of \eqref{eq: 7.44} and applying Lemma 6.8 to the function $\varphi=v(s)-w(s)$
			we deduce that the function
			\begin{equation} \label{eq: 7.45}
				z(y,s)=\begin{cases}
				v(y,s)-w(y,s),&s\in [s_0,s_1], y\in \Omega(s)\\
				0,&s\in [s_0,s_1], y\notin \Omega(s)
				\end{cases}  
			\end{equation}
			satisfies the equation
			\begin{equation} \label{eq: 7.46}
				z_s-\Delta z-\frac{y\cdot\nabla z}{2}-\frac{N}{2} z
				=a\cdot\nabla\bigl(b(y,s)z\bigr) \text{ in } \R^N\times(s_0,s_1) 
			\end{equation}
			with
			\begin{equation} \label{eq: 7.47}
				b(y,s)=\begin{cases}
				\frac{|v|^{1/N} v(y,s)-|w|^{1/N}w(y,s)}{v(y,s)-w(y,s)},
				&s\in [s_0,s_1], y\in \Omega(s)\\
				0,&s\in [s_0,s_1], y\notin \Omega(s)
				\end{cases}
			\end{equation}
			
			It holds
			\begin{equation} \label{eq: 7.48}
				b\in L^\infty\bigl(\R^N\times(s_0,s_1)\bigr) 
			\end{equation}
			and
			\begin{align} \label{eq: 7.49}
				{z=0 \text{ in } \Omega=\bigcup_{s\in(s_0,s_1)}[\Omega(s)\times\{s\}]}\\
				{=\{\,(y,s)\in\R^N\times(s_0,s_1):v(y,s)>w(y,s)\,\}.}  \nonumber
			\end{align}
			
			As $v,w\in C\bigl(\R^N\times(0,\infty)\bigl)$ we see that $\Omega$ is an open subset of $\R^N\times(s_0,s_1)$. On the other hand, we know that it is non-empty.
			
			Suppose that, using of a unique continuation argument, for certain
			$s_2,s_3\in(s_0,s_1)$ with $s_2<s_3$, we prove
			\begin{equation} \label{eq: 7.50}
				z=0 \text{ in } \R^N\times(s_2,s_3). 
			\end{equation}
			Then, by definition of $z$ it holds that $\Omega(s)=\emptyset$ For all
			$s\in[s_2,s_3]$, which, like we have seen before, contradicts the hypothesis $u_0\ne u_1$.
			
When proving \eqref{eq: 7.50} by means of a unique continuation argument, there is a technical difficulty that lies in the fact that the potential $ b \in L^\infty $ appears in the equation under a sign of derivation.
			
			To avoid this difficulty, we consider a ball $B(y_0,\epsilon)$ contained in $\R^N$ and $s_2,s_3\in(s_0,s_1)$ such that $s_2<s_3$ and
			\begin{equation} \label{eq: 7.51}
				B(y_0,\epsilon)\times(s_2,s_3)\subset \Omega. 
			\end{equation}
			This is possible since $\Omega$ is a non-empty open subset of $\R^{N+1}$.
			
			We define 
			\begin{equation} \label{eq: 7.52}
				n(y,s)=\int_0^{\frac{(y_0-y)\cdot a}{|a|^2}} z(y+\lambda
				a,s)\,d\lambda.
			\end{equation}
			
			In \eqref{eq: 7.52} the integral is taken on the segment $I$ of $\R^N$ with a 
			$y$ with
			${y+\delfrac({y_0-y)\cdot a}{|a|^2})a}$, that is the projection of $y$ on the hyperplane $y\cdot a=y_0\cdot a$ of $\R^N$ that contains
			$y_0$ being orthogonal to the convection vector $a$.
			
		If $y\in\R^N$ is an element of the cylinder
			\begin{equation} \label{eq: 7.53}
			C=\biggl\{\,y\in\R^N:\biggl|y+\frac{a}{|a|^2}\bigl((y_0-y)\cdot
			a\bigr)-y_0
			\biggr|<\epsilon\,\biggr\}
			\end{equation}
			then, Integrating the equation \eqref{eq: 7.46} on the segment $I$ we obtain
			that $n=n(y,s)$ satisfies the equation
			\begin{equation} \label{eq: 7.54}
			n_s-\Delta n-\frac{y\cdot\nabla n}{2}-\frac{(N-1)}{2}n=
			-ba\cdot\nabla n \text{ in } C\times(s_2,s_3).
			\end{equation}
			On the other hand
			\begin{equation} \label{eq: 7.55}
				n=0 \text{ in } \{C\cap B(y_0,\epsilon)\}
				\times (s_2,s_3),  
			\end{equation}
			by unique continuation (cf.\ L.~H\"ormander
			\cite{Hor}, Th.~8.9.1., p.\ 224), from \eqref{eq: 7.54} and \eqref{eq: 7.55} we deduce that
			$$n=0 \text{ in } C\times(s_2,s_3)  $$
			and thus 
			\begin{equation} \label{eq: 7.56}
				z(y,s)=-a\cdot\nabla n(y,s)=0 \text{ in } C\times(s_2,s_3).
			\end{equation}
			
		Let $H$ be the hyperplane of $\R^N$, perpendicular to the vector $a$ and which
			contains the origin. Given $y\in\R^N$, we write it as
			$$y=\biggl(\bar y, \frac{(y\cdot a)}{|a|}\biggr)$$
			where $\bar y\in \R^{N-1}$ is the projection of $y$ on the hyperplane $H$.
			
			We define 
			$$\rho(\bar y,s)=\int_{-\infty}^\infty z(\bar y,\lambda,s)\,d\lambda=
			\int_{\infty}^\infty z(y+ta,s)\,dt.$$
			
			From \eqref{eq: 7.56} we deduce that
			$$\rho(\bar y,s)=0,\qquad\forall(\bar y,s)\in[H\cap
			B(0,\epsilon)]\times(s_2,s_3).$$
			
			On the other hand, integrating the equation \eqref{eq: 7.46} in the direction of the vector
			$a$ we deduce that
			$$\rho_s-\Delta_{\bar y}\rho-\frac{\bar y\cdot\nabla_{\bar y}\rho}{2}
			-\delfrac({N-1}{2})\rho=0 \text{ in } H\times(s_2,s_3). $$
			($\Delta_{\bar y}$ and $\nabla_{\bar y}$ denote the laplacian and the
			gradient in the variables $\bar y$).
			
			By unique continuation we obtain
			$$\rho=0 \text{ in } H\times(s_2,s_3) $$
			and since $z\ge 0$,
			$$z=0 \text{ in } \R^N\times(s_2,s_3). $$
			
			This concludes the proof of  Proposition 6.6.
		\end{proof}
		
		\section{Proof of the main results}
		
		In this part, we conclude the proofs of Theorems 6.1 and 6.2.
	To this end, we will apply Theorem 6.3.
		
		Let $Z$ be the Banach space $L^2(K)\cap \BC(\R^N)$, where $\BC(\R^N)$
		is the Banach space of continuous and bounded functions endowed with the supremum norm. Let ${\cal S}(s)$ be the semigroup generated by the semilinear equation \eqref{eq: 7.20}.
		
		As we have seen before, the solutions $v$ to \eqref{eq: 7.20} with data
		inicial $u_0\in L^2(K)\cap \BC(\R^N)$ satisfy
		$$u\in C\bigl([0,\infty); L^2(K)\cap \BC(\R^N)\bigr).$$
		
		On the other hand,
		$$\int v(y,s)\,dy=\int u_0(y)\,dy,\qquad\forall s\ge 0.$$
		
		Therefore the function
		\begin{align*}
		\psi:&\, L^2(K)\cap \BC(\R^N)\to\R\\
		&\varphi\to \psi(\varphi)=\int\varphi(y)\,dy
		\end{align*}
		satisfies hypothesis (6) of Theorem 6.3.
		
		On the other hand, by virtue of Proposition 6.6 we know that the function
		$\phi(\cdot)=\|\cdot\|_1$ satisfies hypothesis (7) of Theorem 6.3.
		
		Applying Theorem 6.3 we obtain that $\forall M\in \R$,
		\begin{equation} \label{eq: 7.57}
		\text{ \eqref{eq: 7.5} admits a unique solution } f_M\in L^2(K)\cap \BC(\R^N) \text{ such that } \int f_M(y)\,dy=M,
		\end{equation}
		\begin{align} \label{eq: 7.58}
			\forall u_0\in  L^2(K)\cap \BC(\R^N) \text{ such that }
			\int u_0(y)\,dy=M,  
			\text{ the solution v to \eqref{eq: 7.20} } \\
			\text{ satisfies } \nonumber \\
			\\
			\|v(s)-f_M\|_{L^2(K)\cap L^\infty(\R^N)}  \to0,
			 when  s\to\infty. \nonumber
		\end{align}
		
		We conclude the proofs of Theorems 6.1 and 6.2 in several steps.
		
		\subparagraph*{Step  1.}
		Regularity and positivity of the profiles $f_M$.
		
		The function
		$$u_M(x,t)=(t+1)^{-N/2} f_M(x/\sqrt{t+1})$$
		is solution to \eqref{eq: 7.17} with $u_0=f_M$. therefore, since $f_M\in L^1(\R^N)$,
		by Theorem 4.3 of Chapter 4 we deduce  that $u(t)\in W^{2,p}(\R^N)$
		for all $t>0$ and $p\in(1,\infty)$, which implies that $f_M\in
		W^{2,p}(\R^N)$ for all $p\in(1,\infty)$.
		
		On the other hand, since $a\cdot\nabla(|f_M|^{1/N}
		f_M)\in\bigl(H^1(K)\bigr)^*$, by elliptic regularity it holds $f_M\in
		H^1(K)$. So, $a\cdot\nabla(|f_M|^{1/N} f_M)\in L^2(K)$ and again, by elliptic regularity, $f_M\in H^2(K)$. From the continuity of the injection $H^2(K)\subset W^{2,1} (\R^N)$ we deduce  that $f_M\in
		W^{2,1}(\R^N)$.
		
		By the comparison principle from Theorem 4.3 we see that if $u_0\ge0$, the solution $u$ to \eqref{eq: 7.17} satisfies $u\ge0$ and therefore, the solution $v$ of \eqref{eq: 7.20} satisfies $v\ge 0$.
		
		If $M\ge 0$, take $u_0\in L^2(K)\cap \BC(\R^N)$ such that
		$$u_0\ge 0,\quad \int u_0(y)\,dy=M.$$
		As the corresponding solution $v$ \eqref{eq: 7.20} satisfies $v\ge0$ and since $v(s)\to f_M$ in $L^\infty(\R^N)$  when  $s\to\infty$,
	we deduce that $f_M\ge 0$. In the same way we can show that
		$M\le0$, thus $f_M\le0$.
		
		By the Hopf maximum principle, we conclude that if $M>0$ (resp.\
		$M<0$) then $f_M(x)>0$ (resp.\ $f_M(x)<0$) for all $x\in\R^N$
		(cf.\ H.~M.~Protter and H.~I.~Weinberger \cite{PrW}, Th.\ 5, p.\ 61).
		
		As the nonlinearity $F(s)=|s|^{1/N}s$ satisfies $F\in
		C^\infty(\R-\{0\})$, by local elliptic regularity we deduce that
		$f_M\in C^\infty(\R^N)$ for all $M\in\R$.
		
		\subparagraph*{Step  2.} Uniqueness of the self-similar profiles.
		
		We have seen  that the self-similar profile es unique  in the clase $L^2(K)\cap
		\BC(\R^N)$. Veamos that es unique  in $L^1(\R^N)$. Suppose that there exist profiles $f_M$, $\tilde f_M\in L^1(\R^N)$ solutions to \eqref{eq: 7.5} such that
		$f_M\ne\tilde f_M$ and
		$$\int f_M(y)\,dy=\int\tilde f_M(y)\,dy=M.$$
		
		So,
		$$u_M(x,t)=(t+1)^{-N/2} f_M\bigl((t+1)^{-1/2} x\bigr);\quad
		\tilde u_M(x,t)=(t+1)^{-N/2}\tilde f_M\bigl((t+1)^{-1/2} x\bigr)$$
	are solutions to \eqref{eq: 7.17}. Clearly
		\begin{equation} \label{eq: 7.59}
			\int\bigr(u_M(x,t)-\tilde u_M(x,t)\bigr)^+\,dx=
			\int\bigr(f_M(x)-\tilde f_M(x)\bigr)^+\,dx,\qquad\forall t>0. 
		\end{equation}
		However the argument used in the proof of  Proposition 6.6 allows to prove that ${\int\bigr(u_M(x,t)-\tilde u_M(x,t)\bigr)^+\,dx}$ is a strictly decreasing function of $t$. This contradicts \eqref{eq: 7.59},
		proves the uniqueness in $L^1(\R^N)$ and concludes the proof of Theorem 6.1.
		
		\subparagraph*{Step  3.} Monotonicity of the self-similar profiles.
		
		The proof of the uniqueness of the self-similar profiles done in the previous part allows to prove that two self-similar profiles cannot intersect. In the case $M_1>M_2$
		this implies that $f_{M_1}>f_{M_2}$ at any point.		
		\subparagraph*{Step  4.} Asymptotic behavior for $u_0\in
		L^2(K)\cap \BC(\R^N)$.
		
		From \eqref{eq: 7.58} we deduce  that if $u_0\in L^2(K)\cap \BC(\R^N)$, the solution
		$u=u(x,t)$  to \eqref{eq: 7.17} satisfies \eqref{eq: 7.2} for all $p\in [1,\infty]$.
		
		Given $u_0\in L^1(\R^N)$ such that
		$$\int u_0(x)\,dx=M$$
	we approximate it through a sequence $\{u_{0,n}\}\subset L^2(K)\cap
		\BC(\R^N)$ such that
		$$\int u_{0,n}(x)\,dx=M; \qquad u_{0,n}\to u_0 \text{ in } L^1(\R^N)  when 
		n\to\infty. $$
		
		Let $u_n(x,t)$ be the solution to \eqref{eq: 7.17} with initial data $u_{0,n}$. We know that
		$$\|u_n(t)-u_M(t)\|_1\to 0  when  t\to\infty $$
		with $u_M$ being the self-similar solution of mass $M$. Since 
		$$\|u(t)-u_n(t)\|_1\le \|u_0-u_{0,n}\|_1, \qquad\forall t>0,\forall
		n\in \N$$
		we deduce 
		\begin{equation} \label{eq: 7.60}
			\|u(t)-u_M(t)\|_1\to 0  when  t\to\infty.
		\end{equation}
		
		Regarding the rescaled solutions $u_\lambda(x,t)=\lambda^N
		u(\lambda x,\lambda^2 t)$, \eqref{eq: 7.60} is equivalent to
		\begin{equation} \label{eq: 7.61}
			u_\lambda(\cdot,1)\to f_M(\cdot) \text{ in } L^1(\R^N)  when 
			\lambda\to\infty.
		\end{equation}
		
		On the other hand, from Theorems 4.3 and 4.4 of Chapter 4 we deduce that
		\begin{equation} \label{eq: 7.62}
			\{u_\lambda(\cdot,1)\}_{\lambda>0} \hbox{ is bounded in }
			W^{1,p}(\R^N),\qquad\forall p\in[1,\infty)
		\end{equation}
	and in particular
		\begin{equation} \label{eq: 7.63}
			u_\lambda(\cdot,1)\to f_M(\cdot) \text{ in } L^p(\R^N), \forall
			p\in[1,\infty].
		\end{equation}
		
		As It can easily be shown , \eqref{eq: 7.63} is equivalent to \eqref{eq: 7.2}. This concludes the proof of Theorem 6.2. \qed
		
		\section{More general non-linearities}
		
		The purpose of this section is to prove the following generalization of Theorem 6.2.
		
		\begin{thm}
			Suppose that the non-linearity $F\in C^1(\R)$ satisfies $\eqref{eq: 7.3}$.
			Then, for all $u_0\in L^1(\R^N)$ such that ${\int u_0(x)\,dx=M}$, la
			solution $u=u(x,t)$ to $\eqref{eq: 7.1}$ satisfies $\eqref{eq: 7.2}$ For all
			$p\in[1,\infty]$, where $u_M$ is the self-similar solution to $\eqref{eq: 7.4}$.
		\end{thm}
		
		\begin{proof}
			Thanks to the estimate $L^1-L^\infty$ of Theorem 4.3 in Chapter 4, it is enough to study the case where $u_0\in
			L^1(\R^N)\cap L^\infty (\R^N)$. By  Theorem 4.3 of Chapter 4 we have
			\begin{equation} \label{eq: 7.64}
				\|u(t)\|_p\le C_p(t+1)^{-\np},\qquad\forall t >0 
			\end{equation}
			for all $p\in[1,\infty]$.
			
			As we have not made any assumption on the behavior of the
			derivative of $ F $, Theorem 4.4 of Chapter 4 is not applicable. 
			By the same method of proof we obtain
			\begin{equation} \label{eq: 7.65}
				\|u(t)\|_{1-\epsilon,p}\le C(\epsilon,p)
				t^{-\np-\frac{(1-\epsilon)}{2}}, \qquad\forall t>0
			\end{equation}
			for all $\epsilon>0$, $p\in[1,\infty]$, with
			$\|\cdot\|_{1-\epsilon,p}$ being the norm in $W^{1-\epsilon,p}(\R^N)$.
			
			The function $v$ defined in \eqref{eq: 7.19} satisfies in this case the following
			non-autonomous equation:			\begin{align} \label{eq: 7.66}
			\begin{dcases} 
			v_s-\Delta v-y\cdot\frac{\nabla v}{2}-\frac{N}{2}v=
			e^{s(N+1)/2} a\cdot\nabla\bigl(F(e^{-Ns/2} v)\bigr) &\hbox{en }\R^N\times(0,\infty)\\
			v(y,0)=u_0(y).
			\end{dcases}
			\end{align}
			Anew, \eqref{eq: 7.2} is equivalent to proving that
			\begin{equation} \label{eq: 7.67}
				 v(s)\to f_M \text{ in } L^p(\R^N)  when  s\to\infty
			\end{equation}
			for all $p\in[1,\infty]$.
			
			The density argument used in the previous section shows that it's enough to consider initial data $u_0\in L^2(K)\cap\BC(\R^N)$.
			
			In view of \eqref{eq: 7.64}--\eqref{eq: 7.65} we know that
			\begin{align} \label{eq: 7.68}
			v&\in L^\infty\bigl(0,\infty; L^p(\R^N)\bigr),\quad\forall p\in[1,\infty]\\
			v&\in L^\infty\bigl(1,\infty;W^{1-\epsilon,p}(\R^N)\bigr),\quad\forall
			p\in[1,\infty],\forall\epsilon>0. \nonumber
			\end{align}
			
			Using \eqref{eq: 7.68}, the condition \eqref{eq: 7.3} imposed on the nonlinearity and by the method used in the proof  of Proposition 6.4 we obtain 
			\begin{align} \label{eq: 7.69}
			v\in L^\infty\bigl(0,\infty; L^2(K)\bigr)\\
			v\in L^\infty\bigl(1,\infty;H^{1-\epsilon}(K)\bigr),&\forall\epsilon>0. \nonumber
			\end{align} 
			and therefore the trajectory $\{v(s)\}_{s\ge0}$ is relatively compact 
			in $L^2(K)$.
			
			We will sue the following result on asymptotic behavior for perturbed dynamical systems, due to V.~Galaktionov and J.~L.~Vázquez
			\cite{GaV}.
			
			\begin{thm}[\cite{GaV}]
		Let us consider a dynamical system in a Banach space $X$ generated by the evolution equation
				\begin{equation} \label{eq: 7.70}
					w_t=A(w).
				\end{equation}
				and a perturbation
				\begin{equation} \label{eq: 7.71}
					w_t=B(t,w).
				\end{equation}
				
				Assume that the following properties hold:
				\begin{enumerate}
					\item[(a)]
					The trajectories $\{w(t)\}_{t\ge0}$ of $\eqref{eq: 7.71}$ are relatively compact in $X$. Moreover, given $\tau>0$, if $w^\tau(t)=w(t+\tau)$, the set $\{w^\tau\}_{\tau>0}$ is relatively compact in $L_{\rm loc}^\infty (0,\infty;X)$.
					
					\item[(b)]
					Given a solution $w\in C\bigl([0,\infty);X)$ to \eqref{eq: 7.71} and if $t_j \rightarrow \infty$ is such that $w^{t_j}(\cdot)$ converges to some function $v \in L^\infty_{\rm loc}(0, \infty; X)$, then  $v$ is a solution to $\eqref{eq: 7.70}$.
					
					\item[(c)]
					The $\omega$-limit set of the equation \eqref{eq: 7.70} in $X$:
					$$\left\{f\in X \Biggm|
					\vcenter{\hsize.5\textheight\noindent
						$\exists w\in C\bigl([0,\infty);X\bigr)$ solution to \eqref{eq: 7.70}
						and a sequence $t_j\to\infty$ such that $w(t_j)\to f$ in $X$}
					\right\}$$
					is compact in $X$ and uniformly stable in the following sense: For all  $\epsilon>0$ there exists $\delta=\delta(\epsilon)>0$ such that if $w$ is a solution to $\eqref{eq: 7.70}$ with $d\bigl(w(0),\Omega\bigr)\le\delta$ then
					$d\bigl(w(0),\Omega\bigr)\le\epsilon$, for all $t>0$.
				\end{enumerate}
				
				Then, under these assumptions, the $\omega$-limit sets of the
				solutions $w\in C\bigl([0,\infty);X)$ to $\eqref{eq: 7.71}$ are contained in
				$\Omega$.
			\end{thm}
			
			With the goal of applying this result to the system \eqref{eq: 7.60} we decompose
			$v$ in the following way
			\begin{equation} \label{eq: 7.72}
			v(y,s)=f_M(y)+\tilde v(y,s).
			\end{equation}
			As ${\int v(y,s)\,dy=\int f_M(y)\,dy=M}$ for all $s\ge0$, $\tilde
			v(s)\in\varphi_1^\perp$ for all $s\ge0$. On the other hand, \eqref{eq: 7.67} is equivalent to
			\begin{equation} \label{eq: 7.73}
				\tilde v(s)\to0 \text{ in } L^p(\R^N),  when  s\to\infty. 
			\end{equation}
			The function $\tilde v$ is solution to the system:
			\begin{align} \label{eq: 7.74}
			\tilde v_2+L\tilde v-\frac{N}{2}\tilde v&=e^{(N+1)s/2}a\cdot\nabla
			\Bigl[F\Bigl(e^{-sN/2}\bigl(f_M+\tilde v(s)\bigr)\Bigr)\Bigr]\\
			&\qquad-a\cdot\nabla(|f_M|^{1/N} f_M)\quad in \quad \R^N\times(0,\infty)\\
			\tilde v(0)&=u_0-f_M. \nonumber 
			\end{align}  
			
	We apply the Theorem 6.5 with $X=L^2(K)\cap\varphi_1^\perp$. The system
			\eqref{eq: 7.74} is \eqref{eq: 7.71} and the non-perturbed system \eqref{eq: 7.70} is
			\begin{align} \label{eq: 7.75}
				{\tilde v_s+L\tilde
					v-\frac{N}{2}\tilde v=a\cdot\nabla\bigl(|f_M+\tilde v|^{1/N}
					(f_M+\tilde v)\bigr)}\\
				{-a\cdot\nabla(|f_M|^{1/N}f_M)
					\text{ in } \R^N\times(0,\infty).} 
			\end{align}
			
			We see that the three assumptions of Theorem 6.5 are satisfied:
			\begin{enumerate}
				\item[(a)]
				From \eqref{eq: 7.69} we deduce  that $\tilde v\in
				L^\infty\bigl(1,\infty;H^{1-\epsilon}(K)\bigr)$. therefore the trajectories $\{\tilde v(s)\}_{s\ge0}$ of \eqref{eq: 7.74} are relatively compact in $L^2(K)$. By parabolic regularity, one can show that
				$v_s\in L^\infty\Bigl(1,\infty; \bigl(H^\epsilon(K)\bigl)^*\Bigl)$ for all $\epsilon>0$. As $\tilde v\in L^\infty\Bigl(1,\infty;
				H^{1-\epsilon}(K)\Bigr)\cap
				W^{1,\infty}\Bigl(1,\infty;\bigl(H^\epsilon(K)\bigr)^*\Bigr)$, by Aubin-like compactness results (cf.\cite{Si}) we deduce  that the family of translations $\{\tilde v^\tau\}_{\tau>0}$ is relatively compact in $L_{\rm loc}^\infty\bigl(0,\infty;L^2(K)\bigr)$.
				
				\item[(b)]
				By virtue of the condition \eqref{eq: 7.3} that the non-linearity satisfies, it can be easily be checked that if $s_j\to\infty$, then the translations $\{\tilde
				v^{s_j}\}_j$ converge in $L_{\rm loc}^\infty\bigl(0,\infty; L^2(K)\bigr)$
				to a solution to the equation \eqref{eq: 7.75}.
				
				The $\omega$-limit set in $L^2(K)\cap \varphi_1^\perp$ reduces to
				$\Omega=\{0\}$ because, due to  Theorem 6.2, all solution to \eqref{eq: 7.75} with 
				initial data in $L^2(K)\cap\varphi_1^\perp$ satisfies
				$$\tilde v(s)\to0 \text{ in } L^2(K) \quad  \text{ when } \quad  s\to\infty.  $$
				
				Let us see lastly the uniform stability of $\Omega=\{0\}$. By the $L^1(\R^N)$ contraction property we see that if $\tilde v$ and $\tilde w$ are two solutions to \eqref{eq: 7.75} with initial data $\tilde v_0$ and $\tilde w_0$ it holds
				$$\|\tilde v(s)-\tilde w(s)\|_1\le\|\tilde v_0-\tilde w_0\|_1
				\le C\|\tilde v_0-\tilde w_0\|_K,\qquad\forall s\ge 0.$$
				In particular, taking $\tilde w_0=0$ it holds
				\begin{equation} \label{eq: 7.76}
					\|\tilde v(s)\|_1\le \|\tilde v_0\|_1\le C\|\tilde v_0\|_K,
					\qquad\forall s\ge0.
				\end{equation}
				
				Using this estimates and techniques from the proof of Proposition 6.4 we obtain the stability of $\Omega=\{0\}$.
				
				As a consequence of Theorem 6.5 we deduce that the solutions to \eqref{eq: 7.74}
				satisfy
				$$\tilde v(s)\to 0 \text{ in } L^2(K)  \quad \text{ when } \quad s\to\infty.  $$
				or, equivalently, the solutions \eqref{eq: 7.66} satisfy
				$$v(s)\to f_M  \text{ in } L^2(K) \quad \text{ when } \quad  s\to\infty.  $$
				
			Since on the other hand, $\{v(s)\}_{s\ge0}$ is bounded in
				$W^{1-\epsilon,p}(\R^N)$ for all $\epsilon>0$,  $p\in[1,\infty]$, one obtains \eqref{eq: 7.67}.
				
	This concludes the proof of Theorem 6.4.
			\end{enumerate}
		\end{proof}
		
		\chapter{Comments}

		\hskip\parindent1.---\enspace
		The techniques from previous chapters 
		allows us to address more general equations of the form
		\begin{align} \label{eq: 8.1}
		u_t-\Delta u&=\div\bigl(F(u)\bigr)\text{ in } \R^N\times(0,\infty)\\
		u(0)&=u_0\in L^1(\R^N) \nonumber 
		\end{align} 
		with $F=(F_1,\ldots ,F_N)\in C^1(\R;\R^N)$ such that $F(0)=0$.
		
		Let
		\begin{align} 
			b_i&=\lim_{|s|\to 0}\frac{F_i(s)}{s} \label{eq: 8.2}\\
			\text{y}
			a_i&=\lim_{|s|\to 0} \frac{F_i(s)-b_is}{|s|^{1/N}s} \label{eq: 8.3}. 
		\end{align}
		The function
		\begin{equation} \label{eq: 8.4}
			v(x,t)=u(x-bt,t) 
		\end{equation}
		with $b=(b_1,\ldots,b_N)$ satisfies
		\begin{align} \label{eq: 8.5}
		v_t-\Delta v&=\div\bigl(H(s)\bigr)\\
		v(0)&=u_0 \nonumber
		\end{align}
		con
		\begin{equation} \label{eq: 8.6}
			H_i(s)= F_i(s)-b_is.
		\end{equation}
		By virtue of \eqref{eq: 8.3}, by the methods shown in Chapter 7, one can prove that the solutions to \eqref{eq: 8.5} behave like the self-similar solutions
		to the equation
		\begin{equation} \label{eq: 8.7}
			v_t-\Delta v=a\cdot\nabla (|v|^{1/N}v)
		\end{equation}
		with $a=(a_1,\ldots,a_N)$.
		
		In this way, we obtain the following results .
		\begin{thm}
		Let $F\in C^1(\R;\R^N)$ be such that los l\'imites $\eqref{eq: 8.2}$ and $\eqref{eq: 8.3}$ exist.
			Let $u_0\in L^1(\R^N)$ be such that
			$$\int u_0(x)\,dx=M.$$
			Then the solution $u=u(x,t)$ to $\eqref{eq: 8.1}$ satisfies
			$$t^{\np}\|u(t)-u_M(t)\|_p\to 0  \quad \text{ when } \quad  t\to\infty $$
			for all $p\in[1,\infty]$, where
			$$u_M(x,t)=t^{-\frac{N}{2}} f_M\bigl(t^{-1/2}(x+bt)\bigr)$$
			and $f_M$ is the self-similar profile of mass $M$ to the equation $\eqref{eq: 8.7}$.
		\end{thm}
		
		Obviously if $a=0$, $f_M$ is the profile of the heat kernel
		$$f_M=M(4\pi)^{-N/2}\exp\biggl(-\frac{|x|^2}{4}\biggr).$$

		\bigskip
		2.---\enspace
		The techniques presented in Chapter 7 allow for addressing systems of the form
		\begin{align} \label{eq: 8.8}
		u_t-\Delta u&=a\cdot\nabla (|u|^{1/N}u)+g(x,t)\\
		u(0)&=u_0.\nonumber 
		\end{align}  
		
		By means of a rescaling argument, it can be shown that this system admits self-similar solutions precisely if $g$ is of the form
		\begin{align} 
			g(x,t)=(t+1)^{-\frac{N}{2}-{\scriptscriptstyle 1}}h
			\bigl((t+1)^{-1/2}x\bigr) \label{eq: 8.9}\\
			\int h(y)\,dy=0. \label{eq: 8.10}
		\end{align}
	\eqref{eq: 8.10} assures the conservation of the mass of the solutions.
		
		In the usual self-similar variables, if $g$ is of the form
		\eqref{eq: 8.9} the system \eqref{eq: 8.8} is converted into		
		\begin{align} \label{eq: 8.11}
		v_s+Lv-\frac{N}{2}v&=a\cdot\nabla(|v|^{1/N}v)+h(y)\\
		v(0)&=u_0. \nonumber 
		\end{align}
		
	Ifi $h\in L^2(K)\cap L^\infty(\R^N)$ and \eqref{eq: 8.10} is satisfied (i.e., $h\in
		\varphi_1^\perp$), it is possible to think that there exists a one-parameter family of stationary solutions to \eqref{eq: 8.11}, i.e., of solutions of the elliptic system
		\begin{align} \label{eq: 8.12}
		Lf_M-\frac{N}{2} f_M&=a\cdot\nabla
		(|f_M|^{1/N}f_M)+h \text{ in } \R^N\\
		\int f_M(y)\,dy&=M \nonumber 
		\end{align}
		and that, if $u_0\in L^2(K)$ with ${\int u_0(y)\,dy=M}$ the solution to \eqref{eq: 8.11} satisfy
		$$v(s) \to f_M \text{ in } L^2(K)  \quad \text{ when } \quad s\to\infty.  $$
		
		The Theorems 7.1 and 7.2 of Chapter 7 provide affirmative answers to these questions for $h=0$.
		
		However, in the case $h=0$ an additional difficulty appears: the estimates of the trajectories in $L^1(\R^N)$. Indeed, multiplying 
		in \eqref{eq: 8.11} por $\sgn(v)$ we obtain
		$$\frac{d}{ds} \|v(s)\|_1\le \|h\|_1$$
		and therefore
		$$\|v(s)\|_1\le \|u_0\|_1 +\|h\|_1s,\qquad \forall s\ge 0$$
		which does not imply that $v\in L^\infty\bigl(0,\infty;L^1(\R^N)\bigr)$.
		
		In \cite{Z1} we have shown that if $h$ is sufficiently small in
		$L^1(\R^N;|y|)$, then the trajectory
		$v$ of \eqref{eq: 8.11} is bounded in $L^1(\R^N)$. Usual arguments allow us to show that $\{v(s)\}_{s\ge 0}$ is relatively compact in $L^2(K)$.
		
		Applying the techniques of Chapter 7 we obtain the following result.

		\begin{thm}
			Let $h\in L^2(K)\cap L^\infty(\R^N)$ be such that
			$$\int h(y)\,dy=0.$$
			There exists $\epsilon>0$ such that if
			$$\int |h(y)||y|\,dy<\epsilon$$
			the elliptic equation $\eqref{eq: 8.12}$ admits a unique solution $f_M\in
			L^1(\R^N)\cap L^\infty(\R^N)$ for every $M\in \R$.
			
			Moreover, for all $u_0\in L^1(\R^N)$ such that $$\int
				u_0(x)\,dx=M$$ the solution $u=u(x,t)$ to
			\begin{align*}
			u_t-\Delta u&=a\cdot\nabla
			(|u|^{1/N}u)+(t+1)^{-\frac{N}{2}-{\scriptscriptstyle1}}
			h\bigl((t+1)^{-1/2}x\bigr)\\
			u(0)&=u_0
			\end{align*}
			satisfies
			$$t^{\np}\|u(t)-u_M(t)\|_p\to 0  \quad \text{ when } \quad  t\to\infty $$
			for all $p\in[1,\infty]$ where $$u_M(x,t)=t^{-\frac{N}{2}}
				f_M(t^{-1/2}x).$$
		\end{thm}
		
It should be noted that both in relation to the existence of self-similar solutions and in relation to asymptotic behavior, there are no restrictions on the mass $ M $.
		
		The method used in the proof of this theorem, analogous to that developed in Chapter 7, will be applied immediately without restriction on the norm of $h$ in $L^1(\R^N;|x|)$ if we have an estimate in $L^1(\R^N)$ for the solution to \eqref{eq: 8.4} with initial data $u_0=0$.

		\bigskip
		3.---\enspace
		The extension of the results of Chapter 7 to  convection-diffusion equations with degenerate diffusion of the type
		\begin{align*}
			u_t-\Delta (|u|^{m-1}u) &=a\cdot\nabla (|u|^{q-1}u) \text { in }
			\R^N\times(0,\infty),\\
			\text{ or } \\
			u_t-\div(|\nabla u|^{p-2}\nabla u) &=a\cdot\nabla (|u|^{q-1}u)  \text{ in }
			\R^N\times(0,\infty)
		\end{align*}
		is an open problem. 
		It would be particularly interesting to know if for these equations, the strong contraction property is satisfied in $L^1(\R^N)$.

		\bigskip
		4.---\enspace
		In the works of I.~L.~Chern and T.~P.~Liu \cite{CL} and S.~Kawashima \cite{Kw1, Kw2} 
the asymptotic behavior of parabolic convection-diffusion systems and mixed systems of parabolic-hyperbolic type is described in one spatial dimension.

It would be interesting to extend these results to higher spatial dimensions using the self-similar solutions built in Chapter 7.		
		
		\bigskip
		5.---\enspace
Let $\Omega$ be a regular domain of $\R^N$. Let $F\in
		C^1(\R,\R^N)$ be such that $F(0)=0$. Let us consider the evolution equation
		\begin{equation} \label{eq: 8.13}
			u_t-\Delta u=\div\bigl(F(u)\bigr) \text{ in } \Omega\times(0,\infty).
		\end{equation}
	
		Integrating formally this equation on $\Omega$ it is observed that for the mass of the solutions to be preserved, it is necessary to impose the following boundary condition:
		\begin{equation} \label{eq: 8.14}
			\frac{\partial u}{\partial\nu}=\nu\cdot F(u) \text{ on } \partial\Omega \times(0,\infty)
		\end{equation}
		where $\nu$ is the exterior normal of the domain $\Omega$.
		
This boundary condition ensures that the flow across the border is zero.
		
The system \eqref{eq: 8.13}--\eqref{eq: 8.14} can be completed the initial condition		
\begin{equation} \label{eq: 8.15}
			u(0)=u_0\in L^1(\Omega).
		\end{equation}
		
It would be interesting to study existence and uniqueness of solutions for this problem, as well as its asymptotic behavior (in the article \cite{AM}, H. ~ Amann studies the existence and uniqueness of solutions for parabolic equations with nonlinear boundary conditions).
		
The asymptotic behavior may strongly depend on the domain $ \Omega $. In particular, whether or not it is bounded.
		
If $ \Omega $ is bounded, the existence and uniqueness of a stationary solution for each mass can be expected. In other words,  one can expect that the system
		\begin{align} \label{eq: 8.16}
		&-\Delta f_M=\div\bigl(F(f_M)\bigr)&\hbox{ in $\Omega$} \nonumber \\
		&\frac{\partial f_M}{\partial\nu}=\nu\cdot F(f_M)&\hbox{ on $\partial\Omega$}\\
		&\int_\Omega f_M(x)\,dx=M \nonumber 
		\end{align}
		has a unique solution for all $M\in\R$. 

It would be interesting to study the existence and uniqueness of these stationary solutions, and then, to address the asymptotic behavior of the solutions of the evolution problem.
		
When $\Omega$ is a cone of $\R^N$, using the self-similar variables and defining
		\begin{equation} \label{eq: 8.17}
			v(u,s)=e^{sN/2}u(e^{s/2}y,e^s-1)
		\end{equation}
		We see that $v$ satisfies
		\begin{align} \label{eq: 8.18}
		&v_s+Lv-\frac{N}{2} v=e^{-(N+1)s/2} \div\bigl(F(e^{sN/2}v)\bigr)
		&\hbox{ in $\Omega\times(0,\infty)$} \nonumber\\
		&\frac{\partial v}{\partial\nu}=-e^{-(N+1)s/2} \bigl(\nu\cdot F(e^{sN/2}v)\bigr)
		-\frac{(y\cdot\nu)}{2}v&\hbox{ on  $\partial\Omega\times(0,\infty)$}\\
		&v(0)=u_0. \nonumber 
		\end{align}
		
		The asymptotic behavior of the solutions will then depend on the properties of $F$. If $F=(F_1,\ldots,F_N)$ with $F_1(s)=|s|^{q_i-1}s$ in an neighbourhood of the origin, ${q_i>\frac{N+1}{N}}$, one may expect that solutions $v$ behave like those of the linear parabolic equation		
		\begin{align} \label{eq: 8.19}
		v_s+Lv-\frac{N}{2}&=0&\Omega\times(0,\infty)\\
		\frac{\partial\nu}{\partial\nu}+(y\cdot\nu)\frac{v}{2}&=0&\hbox{ 
			$\partial\Omega\times(0,\infty)$}\\
		v(0)&=u_0.
		\end{align}		
		On the other hand, one might expect that  $s\to\infty$ the solutions to
		\eqref{eq: 8.19} converge to a steady states, namely solutions to the elliptic problem
		\begin{align} \label{eq: 8.20}
		Lf_M-\frac{N}{2} f_M&=0&\hbox{$\Omega$} \nonumber \\
		\frac{\partial f_M}{\partial\nu}+\frac{(y\cdot\nu)}{2} f_M&=0&\hbox{$\partial\Omega$}\\
		\int_\Omega f_M(x)\,dx&=M. \nonumber 
		\end{align} 
		
		In the case in that $F_i(s)=a_1|s|^{1/N}s$ in a neighbourhood of the origin, perhaps the solutions to \eqref{eq: 8.18} behave like the autonomous system
		\begin{align} \label{eq: 8.21}
		v_s+Lv-\frac{N}{2}v&=a\cdot\nabla(|v|^{1/N}v)&\hbox{ in
			$\Omega\times(0,\infty)$}\\
		\frac{\partial\nu}{\partial\nu}&=-(a\cdot \nu)|v|^{1/N}v-\frac{(y\cdot\nu)}{2} v
		&\hbox{ on $\partial\Omega\times(0,\infty)$} \nonumber \\
		v(0)&=u_0 \nonumber 
		\end{align}
	which in turn will converge into the stationary solutions of the following elliptic equation
		\begin{align} \label{eq: 8.22}
		Lf_M-\frac{N}{2} f_M=a\cdot\nabla(|f_M|^{1/N} f_M)&\quad\hbox{ in $\Omega$}\\
		\frac{\partial f_M}{\partial\nu}+\frac{(y\cdot\nu)}{2} f_M=
		-(a\cdot \nu)|f_M|^{1/N}f_M&\quad\hbox{ on $\partial\Omega$}\\
		\int_\Omega f_M(y)\,dy=M. \nonumber 
		\end{align} 

All the results stated in this section regarding the system \eqref{eq: 8.13} - \eqref{eq: 8.15} are nothing more than conjectures. They are therefore open problems.
		
		\bigskip
		6.---\enspace
The methods developed in the previous chapters allows also to address the asymptotic behavior of the solutions to 
convection-diffusion equations with initial data $ u_0 $ such that
		$$u_0(x)\to l  \quad \text{ when } \quad  |x|\to\infty $$
		with $l\in\R-\{0\}$, under the condition that
		\begin{equation} \label{eq: 8.23}
			u_0-l\in L^1(\R^N). 
		\end{equation}
		
		Let us consider for example a homogeneous nonlinearity:
		\begin{align} \label{eq: 8.24}
		u_t-\Delta u&=a\cdot\nabla(|u|^{q-1}u)\\
		u(0)&=u_0. \nonumber 
		\end{align} 
		
	Assuming that $u_0$ satisfies \eqref{eq: 8.23}, we observe that
		$$v(x,t)=u(x,t)-l$$
		satisfies
		\begin{align} \label{eq: 8.25}
		v_t-\Delta v&=a\cdot\nabla\bigl(g(v)\bigr)&\hbox{ in
			$\R^N\times(0,\infty)$}\\
		v(0)&=v_0 \nonumber 
		\end{align}
		with $v_0=u_0-l$ and
		\begin{equation} \label{eq: 8.26}
			g(s)=|s+l|^{q-1}(s+l)-|l|^{q-1}l.
		\end{equation}
		
		If $q\ge1$, then $g\in C^1(\R)$, and		\begin{equation} \label{eq: 8.27}
			g(s)=q|l|^{q-1}s+\frac{q(q-1)}{2}l^{q-2}s^2+o(s^2).
		\end{equation}
		
		The function
		\begin{equation} \label{eq: 8.28}
			w(x,t)=v(x-aq|l|^{q-1}t,t) 
		\end{equation}
		satisfies
		\begin{align} \label{eq: 8.29}
		w_t-\Delta w&=a\cdot\nabla\bigl(h(w)\bigr)\\
		w(0)&=v_0 \nonumber
		\end{align}
		with
		\begin{equation} \label{eq: 8.30}
			h(s)=q\frac{(q-1)}{2} l^{q-2}s^2+o(s^2). 
		\end{equation}
		
		In dimensions $N>1$ it can be seen that
		$$\lim_{|s|\to0}\frac{h(s)}{|s|^{1/N}s}=0.$$
		By virtue of the results from Chapter 6 we obtain 
		$$t^{-\np}\|w(t)-MG(t)\|_p\to 0  \quad \text{ when } \quad  t\to\infty $$
		for all $p\in[1,\infty]$, where
		\begin{equation} \label{eq: 8.31}
			M=\int\bigl(u_0(x)-l\bigr)\,dx 
		\end{equation}
		and $G$ is the heat kernel.
		
		Thus, we obtain:
		\begin{thm}
			If $N>1$, $q\ge1$ and the initial data $u_0$ satisfies $\eqref{eq: 8.23}$ with $l\ne
			0$ and \eqref{eq: 8.31}, the solution $u=u(x,t)$ to $\eqref{eq: 8.24}$ satisfies
			$$t^{\np}\|u(\cdot,t)-l-MG(\cdot+aq|l|^{q-1}t,t)\|_p\to 0, \text{ if }
			t\to\infty $$
			for all $p\in[1,\infty]$.
		\end{thm}
		
		As
		$$\lim_{|s|\to0}\frac{h(s)}{s^2}=q\frac{(q-1)}{2} l^{q-2}$$
	in dimension $N=1$, a different result holds  since, by virtue of the results of Chapter 6, we know that the solution of \eqref{eq: 8.29} behaves like the self-similar solution
		$$w(x,t)=t^{-1/2} f_M(t^{-1/2}x)$$
	of the equation
		\begin{equation} \label{eq: 8.32}
			w_t-w_{xx}=\frac{aq(q-1)}{2} l^{q-1}(w^2)_x
		\end{equation}
		such that
		$$\int_{-\infty}^\infty f_M(x)\,dx=M.$$
		
		We may obtain the following result.
		\begin{thm}
			If $N=1$, $q\ge1$ and the initial data $u_0$ satisfies $\eqref{eq: 8.23}$ with $l\ne0$
			and $\eqref{eq: 8.31}$, the solution $u=u(x,t)$ to $\eqref{eq: 8.24}$ satisfies
			$$t^{\np}\bigl\|u(\cdot,t)-l-t^{-1/2}
			f_M\bigl(t^{-1/2}(\cdot+aq|l|^{q-1}t)
			\bigr)\bigr\|_p\to0,  \quad \text{ when } \quad  t\to\infty $$
			for all $p\in[1,\infty]$, where $f_M$ is the self-similar profile with mass
			$M$ of the Burgers equation $\eqref{eq: 8.32}$.
		\end{thm}
		
		\bigskip
		7.---\enspace
In \cite{EZ3} we have studied the self-similar solutions to convection-diffusion equations of the form
		\begin{equation} \label{eq: 8.33}
			u_t-\Delta u+|u|^{p-1}u=a\cdot\nabla(|u|^{q-1}u) \text{ in }  
			\R^N\times(0,\infty).
		\end{equation}
		
		The case $a=0$ was addressed in \cite{BrPeT} and \cite{EK1} where it's proven, in
		particular, that if \eqn{1<p<\frac{2}{N}}, there exists a unique self-similar solution of the form
		\begin{equation} \label{eq: 8.34}
		u(x,t)=t^{-\frac{1}{p-1}} f(t^{-1/2}x) 
		\end{equation}
		with $f\in H^1(K)$,$f>0$.
In \cite{EK1} it is likewise shown that there exist
multiple solutions in $H^1(K)$ that change signs.

On the other hand, in \cite{BrPeT} the existence of a family of self-similar solutions whose profiles satisfy
		$$\lim_{|x|\to\infty} |x|^{2/(p-1)}f_0(x)=b\ge0$$
		is shown. It is a one-parameter family of  parameter $b$. 
In \cite{EKM} it is shown that these solutions allow for proving the asymptotic behavior of the solutions to \eqref{eq: 8.33} with $a=0$, while again assuming that the datum $u_0$ of $u$ is such that the limit of $|x|^{2/(p-1)}u_0(x)$  when  $|x|\to\infty$ exists and is non-negative.
		
In \cite{EZ3} we have demonstrated the existence of self-similar solutions of the type \eqref{eq: 8.34} for all $a\in\R^N$,\eqn{1<p<1+\frac{2}{N}} and
\eqn{q=\frac{p+1}{2}}. The uniqueness has been proven only for $ |a| $ sufficiently large.
		
It is important to note that the positive solution is unique in the class $H^1(K)$ but not in general, because $$u(x,t)=t^{-\frac{1}{p-1}}\delfrac({1}{p-1})^{\frac{1}{(p-1)}}.$$ is a self-similar solution to \eqref{eq: 8.33}.
		
As we have previously indicated, in the case $ a = 0 $ there are radial and positive self-similar solutions of \eqref{eq: 8.33} that tend to zero when
$ | x | \to \infty $ but whose profile does not belong to $ H^1(K) $ (\cite{BrPeT}).
		
		\bigskip
		8.---\enspace

The functions of type \eqref{eq: 8.34} are very singular solutions of the equation \eqref{eq: 8.33}. Indeed,  when \eqn{1<p<1+\frac{2}{N}}, it can be seen that
		\begin{align*}
			{t^{\left(\frac{1}{p-1}-\frac{N}{2}\right)}
				\int_{\R^N} u(x,t)\varphi(x)\,dx=t^{-\frac{N}{2}}
				\int_{\R^N} f(t^{-1/2}x)\varphi(x)\,dy}\\
			{\to\biggl(\int_{\R^N} f(x)\,dx\biggr)\varphi(0),\qquad\forall
				\varphi\in \BC(\R^N).}
		\end{align*}
As \eqn{\frac{1}{p-1}>\frac{N}{2}}, it can be seen that the solution $u$ in the instant $t=0$, presents a singularity stronger than the Dirac mass, that is, than the source-type solutions.
The source-type solutions of \eqref{eq: 8.33} are those which satisfy the initial condition
		\begin{equation} \label{eq: 8.35}
			u(0)=c\delta
		\end{equation}
where $c\in\R$ and $\delta$ is the Dirac mass in the origin.
		
In \cite{BrF} it is shown that if $a=0$ and \eqn{1<p<1+\frac{2}{N}} the system
\eqref{eq: 8.33}, \eqref{eq: 8.35} admits a unique solution. On the other hand, if $a=0$ and ${p\ge1+\frac{2}{N}}$, H.~Brezis and A.~Friedman prove in the same work that there are no solutions to \eqref{eq: 8.33}, \eqref{eq: 8.35}.
		
The case $a\ne 0$ was first considered by J.~Aguirre and M.~Escobedo \cite{AE}, who prove that if \eqn{1<p<1+\frac{2}{N}} y
\eqn{1<q<1+\frac{1}{N}}, the problem \eqref{eq: 8.33}, \eqref{eq: 8.35} admits a unique
solution.
		
N.~Liu in \cite{Li} proves that if ${p\ge 1+\frac{2}{N}}$ y
${1<q\le 1+\frac{p+1}{2}}$ there cannot exist solutions to \eqref{eq: 8.33}, \eqref{eq: 8.35}.
		
Regarding the very singular solutions, in \cite{Li} it's shown that if ${1<p<1+\frac{2}{N}}$ and ${1\le q\le 1+\frac{p+1}{2}}$, the solution to \eqref{eq: 8.33}, \eqref{eq: 8.35},  when  $c\to\infty$, goes to a very singular solution to  \eqref{eq: 8.33}. These results for $a=0$ were shown by S.~Kamin and L.~Peletier \cite{KaPe}.
		
		\bigskip
		9.---\enspace
		In the introductory chapter of these notes, we mentioned that the asymptotic behavior of the solutions de
		\begin{equation} \label{eq: 8.36}
			u_t-\Delta u=a\cdot\nabla(|u|^{q-1}u)
		\end{equation}
		with \eqn{1<q< 1+\frac{1}{N}} it is radically different from the case \eqn{q>1+\frac{1}{N}}.
		
		Let us first consider the case one spatial dimension $N=1$:
		\begin{align} \label{eq: 8.37}
		u_t-u_{xx}+\frac{1}{q}(|u|^{q-1}u)_x&=0&\hbox{in $\R\times(0,\infty)$}\\
		u(x,0)&=u_0(x). \nonumber 
		\end{align} 
		As we mentioned, the solutions to \eqref{eq: 8.37}  when  $t\to\infty$ behave like the entropy solutions of the hyperbolic equation:
		\begin{align} \label{eq: 8.38}
		u_t+\frac{1}{q}(|u|^{q-1}u)_x&=0&\hbox{in $\R\times(0,\infty)$}\\
		u(x,0)&=M\delta \nonumber
		\end{align} 
		whose uniqueness was shown by T.~P.~Liu and M.~Pierre \cite{LP1} and that are explicitly given as follows:	
		\begin{align} \label{eq: 8.39}
		u_M(x,t)&=\delfrac({x}{t})^{1/(q-1)} \chi_{(0,r(t))}\\
		r(t)&=cM^{(q-1)/q} t^{1/q},\quad c=\delfrac({q}{q-1})^{(q-1)/q}. \nonumber
		\end{align}
		
		In \cite{EVZ1} the following result is shown:
				\begin{thm}[\cite{EVZ3}]
			Let $u_0\in L^1(\R)$ be such that ${M=\int u_0}$. Then, for all $1\le
			p<\infty$ it holds
			\begin{equation} \label{eq: 8.40}
				t^{\frac{1}{q}\left({\scriptscriptstyle 1}-\frac{1}{p}\right)}
				\|u(t)-u_M(t)\|_p\to0 \quad \text{ when } \quad  t\to\infty
			\end{equation}
			where $u_M$ is the entropy solution of $\eqref{eq: 8.38}$.
		\end{thm}
		
From this results we deduce that the asymptotic behavior of the solutions to \eqref{eq: 8.37} is described by the one-parameter family of fundamental entropy solutions of \eqref{eq: 8.38}. The diffusion term of the system \eqref{eq: 8.37} therefore tends to disappear when $ t \to \infty $. We will say that this type of behavior is {\em strongly non-linear}.
		
		Note that from \eqref{eq: 8.37} and the structure of the solution $u_M$ to \eqref{eq: 8.38} we deduce that the solutions to \eqref{eq: 8.37} satisfy:
		\begin{equation} \label{eq: 8.41}
			\|u(t)\|_\infty\le Ct^{-\frac{1}{q}},\qquad\forall t>0
		\end{equation}
which, given that $ 1 <q <2 $, gives us faster decay  when $ t \to \infty $ to that of the linear heat equation or when $ q \ge2 $.

In fact, the first difficulty that we must solve to prove Theorem 7.5 is obtaining an estimate of type \eqref{eq: 8.41}. We can obtain this estimate from the following entropy inequality:
		\begin{lem}[\cite{EVZ1}]
			Let $1<q<2$ and $u=u(x,t)$ be a non-negative solution of \eqref{eq: 8.37}.
			Then it holds
			\begin{equation} \label{eq: 8.42}
				(u^{q-1})_x\le1/t.
			\end{equation}
		\end{lem}
		
		\begin{proof}
			Observe that the function $w(u^{q-1})_x$ satisfies the equation
			\begin{equation} \label{eq: 8.43}
				w_t-w_{xx}+\biggl[z-\beta\frac{w}{z}\biggr]w_x
				+w^{2^3}+\beta\frac{w}{z^2}=0
			\end{equation}
			with $z=u^{q-1}$ and $\beta=(2-q)/(q-1)$. It can be checked that 
			\eqn{W(t)=\frac{1}{t}} is a super-solution to the equation \eqref{eq: 8.43}. 
By comparison we deduce
			\begin{equation*}
				w(x,t)\le W(t).
			\end{equation*}
		\end{proof}
		
From the entropy inequality \eqref{eq: 8.43} we deduce that if $u$ is a non-negative solution to \eqref{eq: 8.37} with ${M=\int u_0}$, then
		\begin{equation} \label{eq: 8.44}
			0\le u(x,t)\le \delfrac({qM}{(q-1)t})^{1/q}.
		\end{equation}
Indeed, given arbitrary $t>0$ and $x_0\in\R$, thanks to  \eqref{eq: 8.42}, it holds that
		$$u^{q-1}(x,t)\ge B-\delfrac({(x_0-x)}{t}) if x\le x_0  $$ with $B=u^{q-1}(x_0,t)$. Therefore,
		$$u^{q-1}(x,t)\ge\frac{1}{t}(x-x_1) \quad \text{ if } 0\le x-x_1\le Bt $$ with $x_1=x_0-Bt$. 
	Integrating in this inequality with respect to $x$
		we obtain
		$$M=\int u(x,t)\,dx\ge\int_0^{Bt}\delfrac({x}{t})^{1/(q-1)}\,dx
		=\delfrac({q-1}{q}) u^q(x_0,t)t$$
		from where we deduce  \eqref{eq: 8.44}.
		
Having demonstrated \eqref{eq: 8.41} for non-negative solutions, this estimate can be extended to any solution by the maximum principle.
		
The estimate \eqref{eq: 8.41} allows to conclude the proof of Theorem 7.5 by means of a rescaling argument. We define 
		$$u_\lambda(x,t)=\lambda u(\lambda x,\lambda^qt)$$
		which satisfies
		\begin{align*}
		&u_{\lambda,t}-\lambda^{q-2} u_{\lambda,xx} +u^{q-1}_\lambda
		u_{\lambda,x} =0\\
		&u_\lambda(x,0)=\lambda u_0(\lambda x)
		\end{align*}
		and we prove that
		$$u_\lambda(\cdot,t)\to u_M(\cdot,t) \text{ in } L^p(\R)  \quad \text{ when } 
		\lambda\to\infty $$
		for all $1\le p<\infty$ and all $t>0$ where $u_M$ is the unique entropy solution of \eqref{eq: 8.38}.
	
These ideas allows to address the case of spatial dimensions  $N\ge 2$, granted additional technical difficulties.
		
Suppose that $a=(0,\ldots,0,1)$ and we consider the equation
		\begin{align} \label{eq: 8.45}
		u_t-\Delta u&=\partial_N(|u|^{q-1}u)\\
		u(x,0)&=u_0(x)
		\end{align}
	with \eqn{1<q<1+\frac{1}{N}} (henceforth $\partial_N$ designates the partial derivative
		in the direction $x_N$) with a non-negative initial data.
		
In this case, the entropy inequality that we obtain for the non-negative solutions is:
		\begin{equation} \label{eq: 8.46}
			\partial_N(u^{q-1})\le\frac{1}{t} 
		\end{equation}
		from where it follows that
		\begin{equation} \label{eq: 8.47}
			\|u(t)\|_\infty\le Ct^{-(N+1)/2q}. 
		\end{equation}
Again, the estimate \eqref{eq: 8.47} gives a better decay than that of the solutions to the heat equation or  \eqref{eq: 8.45} for ${q\geq 1+\frac{1}{N}}$.
		
After proving the estimate \eqref{eq: 8.47} we introduce the scaling
	$$u_\lambda(x,t)=\lambda^\alpha u(\lambda^{1/2} x',\lambda^\beta
		x_N,\lambda t)$$
		con $x'=(x_1,\ldots,x_{N-1})$,$\alpha=(N+1)/2q$ y
		$\beta=(1+N+q-Nq)/2q$. The function $u_\lambda$ satisfies
		\begin{align*}
		&u_{\lambda,t}-\Delta' u_\lambda-\partial_N(|u_\lambda|^{q-1}u)=
		\lambda^{1-2\beta}\partial_N^2 u_\lambda\\
		&u_\lambda(x,0)=\lambda^\alpha u_0(\lambda^{1/2} x',\lambda^\beta x_N)
		\end{align*}
		where $\Delta'$ is the laplacian in the variables $(x_1,\ldots,x_{N-1})$.
		
	It can be shown that, for $t>0$ fixed,
		$$u_\lambda(\cdot,t)\to u_M(\cdot,t) \text{ in } L^p(\R^N)  \quad \text{ when }  t\to\infty $$
		for all $1\le p<\infty$, where $u_M(x', x_N,t)$ is the non-negative entropy solution of the equation
		\begin{align} \label{eq: 8.48}
		u_t-\Delta' u&=\partial_N(|u|^{q-1}u)\\
		u(x,0)&=M\delta \nonumber 
		\end{align} 
		
		In this way we can conclude that
		$$t^{\alpha(p-1)/p}\|u(\cdot,t)- u_M(\cdot,t)\|_p   \quad \text{ when } 		\lambda\to\infty $$
		for all $1\le p<\infty$.
		
The most notable difference with respect to the proof done in one spatial dimension is the need to prove the existence and uniqueness of the entropy solution for \eqref{eq: 8.48}. This result when $ N = 1 $ was proven in \cite{LP1}. 
In \cite{EVZ2, EVZ3} we prove the existence of a unique non-negative entropy solution to \eqref{eq: 8.48}. 
The uniqueness without any restriction on the sign of the solution is an open problem.

The entropy condition that we introduce in \cite{EVZ1, EVZ2}, and whuch allows for characterizing the unique nonnegative fundamental solution is a variant
of the classical condition of S. ~ Kruzhkov \cite{Kr} that takes into account the presence of the term of partial diffusion. 
More precisely, it's the following condition: for all non-negative test function $\psi=\psi(x') \in \mathcal{D}(\R^{N-1})$ one must have
		\begin{align*}
	&{\frac{\partial}{\partial t}|u-\psi(x')|-\Delta'|u-\psi(x')|}\\
	&{\le \partial_N\bigl[\bigl(|u|^{q-1} u-|\psi(x')|^{q-1}\psi(x')\bigr)
				\sgn\bigl(u-\psi(x')\bigr)\bigr]
				+\sgn\bigl(u-\psi(x')\bigr)\Delta'\psi(x').}
		\end{align*}
Under this entropy condition, the uniqueness of the  entropy solutions with initial data in $L^1(\R^N)\cap L^\infty(\R^N)$ is shown in \cite{EVZ3} by means of the Kruzhkov method \cite{Kr}. 
The uniqueness of the fundamental non-negative solution requires additional arguments and is covered in \cite{EVZ2}.

	\bibliographystyle{alpha}
	\bibliography{refs}{}
	

	\Adresses
	
\end{document}